\documentclass{article}
\usepackage{arxiv}
\usepackage[T1]{fontenc}
\usepackage[utf8]{inputenc}
\usepackage[authoryear,round]{natbib}

\usepackage{amsmath,amsfonts,bm}

\def\eqref#1{equation~\ref{#1}}

\def\1{\bm{1}}

\DeclareMathAlphabet{\mathsfit}{\encodingdefault}{\sfdefault}{m}{sl}
\SetMathAlphabet{\mathsfit}{bold}{\encodingdefault}{\sfdefault}{bx}{n}

\newcommand{\E}{\mathbb{E}}

\DeclareMathOperator*{\argmin}{arg\,min}

\renewcommand{\eqref}[1]{(\ref{#1})}

\usepackage{hyperref}
\hypersetup{
  hidelinks,
  pdftitle={Beyond Aggregate Drift: Optimal Client-Specific Windows in Federated Learning},
  pdfauthor={Yuriy Dorn, Dmitrii Vishovan, Kirill Novoselov, Ilya Kuruzov, Darina Dvinskikh, Alexander Gasnikov}
}
\usepackage{url}
\usepackage{amsmath,amssymb,mathtools}
\usepackage{amsthm}
\usepackage{graphicx}
\usepackage{booktabs}
\usepackage{subcaption}
\usepackage{enumitem}
\usepackage{xcolor}
\usepackage{array}
\usepackage{algorithm}
\usepackage{algorithmic}
\usepackage{microtype}


\theoremstyle{plain}
\newtheorem{theorem}{Theorem}[section]
\newtheorem{lemma}[theorem]{Lemma}
\newtheorem{corollary}[theorem]{Corollary}
\newtheorem{proposition}[theorem]{Proposition}
\theoremstyle{definition}
\newtheorem{definition}[theorem]{Definition}

\theoremstyle{remark}

\providecommand{\cX}{\mathcal{X}}
\newcommand{\Prob}{\mathbb{P}}
\newcommand{\inner}[2]{\langle #1, #2\rangle}
\newcommand{\norm}[1]{\left\lVert #1\right\rVert}
\newcommand{\proj}{\Pi_{\cX}}
\newcommand{\kl}{\mathrm{KL}}
\newcommand{\tv}{\mathrm{TV}}
\newcommand{\avgerr}{\mathrm{AvgErr}}
\newcommand{\xstar}{x^{\star}}
\newcommand{\xinit}{x_{\circ}}
\newcommand{\ii}{\mathrm{i}}
\DeclareMathOperator{\sinc}{sinc}

\newcommand{\Hg}{\mathcal{H}_{\gamma}}
\newcommand{\Sg}{\mathcal{S}_{\gamma}}
\newcommand{\Cg}{\mathcal{C}_{\gamma}}
\newcommand{\transient}{\mathcal{T}}
\newcommand{\batch}{\mathcal{B}}
\newcommand{\Rint}{\mathcal{I}}
\newcommand{\bias}{\delta}
\newcommand{\Filt}{\mathcal{F}}
\newcommand{\Risk}{\mathcal{R}}
\newcommand{\Pclass}{\mathfrak{P}}
\newcommand{\Aclass}{\mathfrak{A}}
\newcommand{\sched}{\mathsf{N}}
\newcommand{\alg}{\mathsf{A}}
\newcommand{\inst}{\mathsf{P}}
\newcommand{\bsign}{\varsigma}
\newcommand{\grid}{\mathcal{G}}
\newcommand{\vscore}{\widehat{\ell}}
\newcommand{\etaew}{\eta_{\mathrm{ew}}}
\newcommand{\Qrisk}{\mathcal{Q}}

\title{Beyond Aggregate Drift:\\ Optimal Client-Specific Windows in Federated Learning}

\author{\normalfont\parbox{0.96\textwidth}{\centering
  Yuriy Dorn\textsuperscript{3},
  Dmitrii Vishovan\textsuperscript{1},
  Kirill Novoselov\textsuperscript{1},\\
  Ilya Kuruzov\textsuperscript{1,2},
  Darina Dvinskikh\textsuperscript{4}, and
  Alexander Gasnikov\textsuperscript{1,5,6}\\[0.8em]
  \small
  \textsuperscript{1}Innopolis University\\
  \textsuperscript{2}Moscow Independent Research Institute of Artificial Intelligence\\
  \textsuperscript{3}Lomonosov Moscow State University\\
  \textsuperscript{4}Higher School of Economics\\
  \textsuperscript{5}Trusted AI Research Center, RAS\\
  \textsuperscript{6}Steklov Mathematical Institute, RAS
}}
\date{}
\renewcommand{\shorttitle}{Beyond Aggregate Drift}
\renewcommand{\headeright}{}
\renewcommand{\undertitle}{}

\begin{document}

\maketitle

\begin{abstract}
We study federated learning from continuously arriving data: each client samples from its own distribution, these distributions drift over time, and the server must track the minimizer of the clients' weighted average loss. Clients are not alike as they collect data at different rates, with different noise, and change at different speeds. The main question is how much history should each one keep? A single
shared window must serve every client at once, and so makes a fast-changing client carry the sampling noise of a slow, nearly stationary one.

We give each client a memory of its own, sized to that client's noise, arrival rate, and drift, and refresh it by restarted accelerated method that starts from a common point. The tracking error is then bounded by
\[
  R\left(\sum_m w_m^{3/2}\,\sigma_m\sqrt{\gamma_m/\lambda_m}\right)^{2/3},
\]
for clients with weight $w_m$, noise $\sigma_m$, arrival rate $\lambda_m$, and
drift rate $\gamma_m$ on a domain of diameter $R$. This is never worse than what any common window can achieve, and strictly better whenever noise and drift are spread unevenly across clients.

A lower bound that puts different groups of clients on different drift clocks shows this quantity is the right one. The two bounds agree up to a constant whenever the clients occupy boundedly many refresh scales, and in general differ only by the sixth root of the number of scales. Experiments show the gain comes from the memory rather than from the solver: it persists across accelerated, stochastic-gradient, and one-step methods, and against a separately tuned common-memory baseline.
\end{abstract}

\section{Introduction}
\label{sec:intro}

Many learning systems receive data continuously from several clients
whose statistics change over time: sensors in a monitoring network,
institutions or devices contributing to a federated model, users of a
recommendation service. Because the target is the \emph{current}
population objective, an observation loses value as the distribution that
produced it moves away, and the learner must decide how much history to
keep. For a single stream the classical answer is a sliding window whose
length balances the bias caused by drift against the variance of
estimating from few samples \citep{gama2014survey,mazzetto2023adaptive}. But clients rarely change at the same speed:
one sensor may supply a few noisy observations of a population that
hardly moves, while another supplies many observations of a population
that changes quickly. One common memory for all clients then faces a
dilemma. A short memory throws away the stable client's history, which
is still informative, whereas a long memory keeps the changing client's
history, which is already stale. The statistical question behind this
dilemma is how long each client's past observations remain useful, and
what a common memory costs.

\paragraph{The problem.}
There are $M$ clients with public weights $w_m$. Client $m$ supplies
independent time-stamped samples at rate $\lambda_m$ from a distribution
$\mathcal D_m^t$ that moves in time. The server tracks the weighted
population objective $F^t=\sum_m w_m f_m^t$, where $f_m^t$ is the
expected loss under $\mathcal D_m^t$ and $F^t$ is convex on a compact
domain of diameter $D$, by reporting a causal trajectory $(x_t)_{t\ge0}$:
the point reported at time $t$ may depend only on information the server
has received by then, and each synchronization with the clients takes
$\tau$ time units. The learner is charged the time-averaged
suboptimality of $x_t$ with respect to the current objective $F^t$, a
continuous-time counterpart of dynamic regret. Each client is described
by three public numbers: its rate $\lambda_m$, a variance bound
$\sigma_m^2$ for its stochastic gradients, and a bound $\gamma_m$ on the
speed of its population gradient. The learner's decision is the
\emph{data memory}: which of each client's past samples enter the
stochastic problem that the server solves when it refreshes its report.

\paragraph{What is known, and what is not.}
With one common memory of length $h$, the single-stream bias--variance
balance carries over: the aggregate drift $\bar\gamma=\sum_m w_m\gamma_m$
is balanced against the pooled uncertainty
$\Xi^2=\sum_m w_m^2\sigma_m^2/\lambda_m$, and the best common length
yields the upper-bound scale $D\bar\gamma^{1/3}\Xi^{2/3}$
(Appendix~\ref{app:benchmarks}). This scale is
loose whenever the noisiest clients are the ones that barely change,
since it charges their uncertainty against the drift of the others. Whether the looseness is real cannot be decided by a lower
bound that lets an adversary reallocate a total drift budget among the
clients, because here each client's drift bound is prescribed separately.
Two further obstacles stand between the natural remedy, a separate
memory per client, and a guarantee: retained samples are reused across
refreshes, so the iterates depend on the data they are evaluated on, and
the drift bounds may be unknown in practice.

\paragraph{Contributions.}
We characterize what client-specific memories gain over a common one, and
we show how to select among them without knowing the drift bounds.
\begin{enumerate}[leftmargin=1.35em,itemsep=2pt,topsep=3pt]
\item \textbf{A heterogeneous tracking quantity.}
Separate completed data windows, a reset stochastic solve after each
refresh, and disjoint batches within each solve attain
\[
 D\Hg
 =D\Bigl(\sum_m w_m^{3/2}\sigma_m\sqrt{\gamma_m/\lambda_m}\Bigr)^{2/3}
 \le D\bar\gamma^{1/3}\Xi^{2/3}
\]
with explicit optimization, sample-count, and elapsed-time terms
(Theorem~\ref{thm:refresh-upper}). The argument reuses stored samples
across solves under pointwise gradient-moment assumptions only, and it
turns the uncertainty of a stationary client into a vanishing
finite-horizon transient; projected SGD attains the same quantity.
\item \textbf{Lower bounds that respect the client vector.}
Independent client groups with unequal clocks give a group-sensitive
lower bound (Corollary~\ref{cor:refresh-dyadic}) whose uniform
simplification is a constant multiple of $D\Hg/J^{1/6}$ for $J$ dyadic
memory scales, and an explicit joint regime controls every term of the
upper bound. Separately, an exact periodic calculation shows that the
risk ratio between common and client-specific uniform completed-window
estimators is unbounded (Proposition~\ref{prop:periodic-main}) and yields
a test for when forgetting helps.
\item \textbf{Selection without drift inputs.}
Under a bounded-loss assumption, independent validation and causal expert
aggregation compete with the best fixed joint memory vector
(Theorem~\ref{thm:adaptive-summary}), retaining the pooled-variance
comparison and paying explicit validation and source-age costs; the
aggregation is standard, and the point is its validity under causal
reuse.
\item \textbf{Evidence and boundaries.}
Prospective tests on periodic streams confirm predicted benefits and
failures of a specified memory rule, and controlled experiments separate
the memory mechanism from the solver. On temperature laws inside the
regime of Corollary~\ref{cor:refresh-joint}, the windows of
Corollary~\ref{cor:refresh-opt} cut a December gap by $51\%$ and $52\%$
against the bound-optimal and tuned common lengths; the gain tracks the
predicted heterogeneity and shrinks when the clocks agree. A clinic
network with unequal drift shows the same window shape cutting the gap
under exact buffered minimization, with a smaller reduction once
Algorithm~\ref{alg:refresh} leaves an optimization remainder. These
comparisons establish no universal advantage for client-specific windows.
\end{enumerate}

\section{Related work}
\label{sec:related}

This paper differs from prior guarantees in its prescribed clientwise
drift vector, simultaneously running unequal clocks, and the resulting
pooled statistical quantity, all introduced in
Sections~\ref{sec:refresh-model}--\ref{sec:adaptive-memory};
Table~\ref{tab:referee-comparison} in Appendix~\ref{app:prior-comparison}
compares assumptions and criteria.

\paragraph{Streaming and federated learning.}
\citet{marfoq2023streams} study fixed-weight population targets with
local caches and sample reuse; their laws are stationary, and their
optimization analysis relies on uniform gradient-discrepancy control.
\citet{huynh2025markovian} treat dependent Markovian streams and
nonconvex objectives relative to fixed stationary laws, and
\citet{zhao2026retention} optimize admission and retention under sampling
and buffer budgets for stationary independent streams; we track
persistently changing populations from independent samples with
uncharged storage and computation, and none of these models contains the
others. LASG's selective communication \citep{chen2020lasg} and
FedDrift's multi-model drift clustering \citep{jothimurugesan2023feddrift}
address different mechanisms; we contact every client and track one
fixed-weight objective.

\paragraph{Drift and adaptation.}
Master-FL \citep{ganguly2024nonstationary} uses detection and restarts to
control dynamic regret under unknown aggregate variation of empirical
losses, whereas our target fixes population weights and separates
clientwise gradient drift from information rates.
\citet{mazzetto2023adaptive} select a recent window under uniform
convergence and discrepancy assumptions. Our adaptive result instead
competes with fixed joint memory experts through an independent
validation stream. The aggregation and the delayed-feedback reduction
are standard \citep{freund1997hedge,joulani2013delayed}; the contribution
is that the pooled statistical comparison survives in the causal reuse
model, not a new drift-detection principle.

\paragraph{Optimization and lower bounds.}
Acceleration with reciprocal-count variance control is established
\citep{tyurin2025tight}, and single-stream tracking has guarantees under
strong convexity \citep{cutler2023drift}; historical reweighting admits
discrepancy and complexity bounds \citep{awasthi2023batchdrift}, and
\citet{mazzetto2023density} prove average-risk block lower bounds for
drifting distributions, whereas our construction retains unequal
client-group clocks under a prescribed drift vector.

\section{Setting}
\label{sec:refresh-model}

\paragraph{Objective.}
There are $M$ clients with known weights $w_m>0$, $\sum_m w_m=1$.
On a compact convex set $\cX\subset\mathbb R^d$ of diameter $D>0$, with
Euclidean projection $\proj$ and Euclidean norms throughout, we track
\[
 F^t(x)=\sum_{m=1}^M w_m f_m^t(x),\qquad
 f_m^t(x)=\E_{\xi\sim\mathcal D_m^t}\,\ell(x;\xi),\qquad t\ge0 .
\]
The sample loss $\ell$ is public. The population path
$(\mathcal D_m^t)_{m,t}$ is fixed before any sample noise is drawn and is
never revealed to the learner. Each $F^t$ is convex and $L$-smooth,
$L>0$, and a public bound $\|\nabla F^t(x)\|\le G$ holds for all $t\ge0$
and $x\in\cX$; compactness guarantees that the current minimizer exists,
and the gradient bound controls the loss before the first data-based
report.

\paragraph{Independent time-stamped samples.}
Client $m$ has a locally finite arrival count $N_m(t)$ with $N_m(0)=0$.
Conditionally on the complete exogenous array of arrival times, all
samples are independent across arrivals and clients, and a sample arriving
at time $s$ has law $\mathcal D_m^s$. A stored sample may be evaluated at
any later query point, and the oracle assumptions are pointwise: for each
fixed $x\in\cX$ and $\xi\sim\mathcal D_m^s$,
\begin{equation}
\label{eq:refresh-noise}
 \E\nabla\ell(x;\xi)=\nabla f_m^s(x),\qquad
 \E\|\nabla\ell(x;\xi)-\nabla f_m^s(x)\|^2\le\sigma_m^2 .
\end{equation}
No uniform concentration over sample-dependent query points is assumed.

\paragraph{Prescribed clientwise drift.}
A public vector $\gamma=(\gamma_m)_{m=1}^M$ with $\gamma_m\ge0$ bounds
the speed of each client's population gradient:
\begin{equation}
\label{eq:refresh-drift}
 \|\nabla f_m^t(x)-\nabla f_m^s(x)\|\le\gamma_m|t-s|,
 \qquad x\in\cX,\ s,t\ge0 .
\end{equation}
The bound is imposed client by client; an adversary may not reallocate
the aggregate budget $\bar\gamma:=\sum_m w_m\gamma_m$ across clients.
The fixed-memory rule of Section~\ref{sec:refresh} takes certified noise
and drift bounds as inputs; Section~\ref{sec:adaptive-memory} removes the
drift bounds from the inputs under an additional bounded-loss assumption
and an independent validation stream, without estimating them from the
optimization buffers.

\begin{definition}[Regular arrivals]
\label{def:refresh-regular}
For known positive rates $(\lambda_m)_{m=1}^M$, with
$\lambda_{\min}=\min_m\lambda_m$, a realized schedule is
$\lambda$-regular if, for every $m$ and all $t-s\ge2/\lambda_m$,
\[
 \tfrac12\lambda_m(t-s)\le N_m(t)-N_m(s)\le2\lambda_m(t-s).
\]
\end{definition}
The schedule $N_m(t)=\lfloor\lambda_mt\rfloor$ is $\lambda$-regular,
whereas a Poisson path in general is not; the lower inequality
guarantees enough buffered data, and the upper one limits the
information available in a hard instance. We work conditionally on a
fixed regular schedule, and all expectations concern the sample noise
and the learner's private randomization.

\paragraph{Communication and loss.}
One synchronization (client messages followed by a server broadcast) costs $\tau\ge0$ time units, regardless of the message sizes. Clients may
compute arbitrarily on their samples and the broadcast state. Computation and storage carry no time charge, and raw samples may be transmitted, so raw buffers may be pooled at the server. Consequently, the round counts we report describe one implementation. The reported point may depend only on information already received at the server, and its loss is
\begin{equation}
\label{eq:refresh-risk}
 \avgerr_T=\frac1T\int_0^T
 \E\Bigl[F^t(x_t)-\min_{x\in\cX}F^t(x)\Bigr]\,dt .
\end{equation}

\paragraph{Two statistical-drift quantities.}
For each client define
\begin{equation}
\label{eq:refresh-quantities}
 a_m=w_m\gamma_m,\qquad
 b_m=\frac{w_m^2\sigma_m^2}{\lambda_m},\qquad
 \Sg=\sum_m\sqrt{a_mb_m},\qquad
 \Hg=\Sg^{2/3},
\end{equation}
so that $\bar\gamma=\sum_m a_m$, $\Xi^2:=\sum_m b_m$ is the pooled uncertainty of the introduction, and $\Hg$ is the quantity displayed there. Client-specific memories replace the common-memory scale $\bar\gamma^{1/3}\Xi^{2/3}$ (derived in Appendix~\ref{app:benchmarks}) by $\Hg$, and by the Cauchy--Schwarz inequality
\begin{equation}
\label{eq:refresh-comparison}
 \Hg\le\bar\gamma^{1/3}\Xi^{2/3},
\end{equation}
with equality exactly when the vectors $(a_m)_m$ and $(b_m)_m$ are proportional. The gap can be large: for two clients with $a_1>0$, $a_2=0$ and $b_1,b_2>0$, the client-specific quantity is $(a_1b_1)^{1/3}$, whereas the common-memory quantity is $[a_1(b_1+b_2)]^{1/3}$. The stationary client's uncertainty $b_2$ has left the drift floor altogether; Theorem~\ref{thm:refresh-upper} shows that it reappears only as a finite-horizon transient.

\section{Client-specific memory under known drift bounds}
\label{sec:refresh}

A common memory ties the bias of a rapidly changing client to the noise
of every other client. When each client's drift bound is prescribed, a
different balance is possible: each client keeps its own window, and the
server solves a fresh stochastic problem on the frozen windows whenever
it refreshes its report.

\subsection{Data buffers, reset solves, and disjoint batches}
\label{sec:refresh-method}

Fix an integer $K\ge1$ (gradient rounds per solve), client memory
lengths $h_m\ge2K/\lambda_m$, and a reporting stride $\Delta>0$ with
$\Delta\ge K\tau$.
The memory of client $m$ is a \emph{completed window}: at each time
$ph_m$, $p=1,2,\ldots$, the buffer is replaced by the observations that
arrived in $((p-1)h_m,ph_m]$, and before the first complete window the
buffer holds the arrived prefix. A client with $\gamma_m=0$ may use
$h_m=\infty$ and keep its entire prefix; we then read $1/h_m=a_mh_m=0$.

Solve $j=0,1,\ldots$ starts at $s_j=s_0+j\Delta$ with
$s_0=2K/\lambda_{\min}$. At $s_j$ the server freezes a copy of every
client's current buffer, which is the latest completed window or, if
$s_j<h_m$, the prefix. The frozen interval of client $m$ has length
$h_{m,j}=\min\{s_j,h_m\}$, contains $n_{m,j}\ge K$ observations on a
$\lambda$-regular schedule, and consists of timestamps in
$(s_j-2h_{m,j},s_j]$; buffer replacements during the solve do not affect
the frozen copy. The solve lasts $c:=K\tau$ time units, its output is
reported on the interval $\Rint_j=[t_j,t_{j+1})$ with $t_j:=s_j+c$, and
a fixed public point $\xinit\in\cX$ is reported before $t_0$.

\begin{algorithm}[t]
\caption{Client-specific data refresh with reset accelerated solves}
\label{alg:refresh}
\begin{algorithmic}[1]
\REQUIRE $\xinit\in\cX$, $K$, $(h_m)$, $\Delta>0$ with $\Delta\ge K\tau$,
          $L$, certified $(\sigma_m)$
\FOR{solves $j=0,1,\ldots$ starting at $s_j=2K/\lambda_{\min}+j\Delta$}
\STATE freeze the client buffers; split each into $K$ nonempty batches
       $\batch_{m,j,1},\ldots,\batch_{m,j,K}$ of sizes differing by at
       most one, by timestamp order only
\STATE reset $x^0=v^0=\xinit$
\FOR{$k=1,\ldots,K$}
\STATE $\theta_k=2/(k+1)$;\quad
       $y^k=(1-\theta_k)x^{k-1}+\theta_kv^{k-1}$
\STATE each client returns the mean gradient of batch $\batch_{m,j,k}$
       at $y^k$; the server forms the $w$-weighted aggregate $g^k$
\IF{$k=1$}
\STATE clients also send their batch-size vectors; the server sets
       $V_{j,l}=\sum_m w_m^2\sigma_m^2/|\batch_{m,j,l}|$ for
       $l=1,\ldots,K$, $S_j=\sum_{l=1}^Kl^2V_{j,l}$, and
       $\eta_j=\min\{1/(4L),\,D/\sqrt{2S_j}\}$ (with $\eta_j=1/(4L)$ if
       $S_j=0$)
\ENDIF
\STATE $v^k=\proj(v^{k-1}-k\eta_jg^k)$;\quad
       $x^k=(1-\theta_k)x^{k-1}+\theta_kv^k$
\ENDFOR
\STATE report $x_j=x^K$ on $\Rint_j=[s_j+K\tau,\,s_{j+1}+K\tau)$
\ENDFOR
\end{algorithmic}
\end{algorithm}

Algorithm~\ref{alg:refresh} states the solve. Because $y^1=\xinit$ is
fixed, the batch counts travel with the first gradient message and the
server knows $\eta_j$ before its first update; every round contacts all
clients and costs $\tau$, and no round waits for fresh samples. Each
solve resets its iterates and gradient state: an observation may reappear
in later solves, but within one solve the $K$ batches are disjoint. The
reset is what permits this reuse under the pointwise
assumptions~\eqref{eq:refresh-noise}: query $y^k$ depends only on the
earlier batches of the \emph{same} solve, so the analysis reveals them in
order without conditioning on previous outputs that used overlapping
data (Appendix~\ref{app:refresh-upper}). Since messages are unrestricted,
the frozen buffers can instead be uploaded in one synchronization and
the same updates run at the server, in which case the bounds below hold
with $K\tau$ replaced by $\tau$ (Appendix~\ref{app:raw-buffer}): the
elapsed time $K\tau$ belongs to the gradient-message implementation, not
to the model.

\subsection{Tracking guarantee and optimal memory lengths}

\begin{theorem}[Client-specific refresh]
\label{thm:refresh-upper}
Under the assumptions of Section~\ref{sec:refresh-model}, condition all
expectations on a fixed $\lambda$-regular schedule and let $c=K\tau$ and
$\beta_0=4\sqrt{2/3}$. For Algorithm~\ref{alg:refresh}, every
$t\in\Rint_j$ satisfies
\begin{align}
\label{eq:refresh-pointwise}
\E[F^t(x_j)-\min_\cX F^t]
&\le \frac{4LD^2}{K(K+1)}
 +\beta_0D\sqrt{\sum_m\frac{b_m}{h_{m,j}}}
 +2D\sum_m a_mh_{m,j}
 +D\bar\gamma(c+\Delta),
\end{align}
and for every $T\ge s_0+c$ the reported trajectory satisfies
\begin{align}
\label{eq:refresh-average}
\avgerr_T
&\le \frac{4LD^2}{K(K+1)}
 +D\Bigl(2\sum_m a_mh_m+
         \beta_0\sqrt{\sum_m b_m/h_m}\Bigr)
 +D\bar\gamma(c+\Delta)
 +\transient_T,\\
\label{eq:refresh-transient}
\transient_T
&:=\beta_0D\,\Xi
   \Bigl(\frac{2}{\sqrt T}+\frac{\Delta}{T\sqrt{s_0}}\Bigr)
   +\frac{GD(s_0+c)}{T}.
\end{align}
\end{theorem}

The proof (Appendix~\ref{app:refresh-upper}) combines a projected
accelerated potential for a stochastic oracle with bounded gradient bias
and a filtration that exposes only the disjoint batches of the current
solve. The bias bound, controlled client by client through the sample
ages (including the solve and reporting lag), enters multiplied by $D$
independently of the number of internal iterations, and the stochastic
term comes from the independent batches. The four terms
of~\eqref{eq:refresh-average} are the optimization error, the
statistical-drift balance, the elapsed-time cost, and the finite-horizon
transient $\transient_T$; a client with an infinite window appears only
in the transient, through $\Xi$, so stationary clients are never charged
a permanent drift floor.

\begin{corollary}[Optimized heterogeneous windows]
\label{cor:refresh-opt}
Suppose $\Sg>0$ and $\sigma_m>0$ for every client with $\gamma_m>0$. For
such clients set
\[
 h_m^\star=(2/3)^{1/3}\Sg^{-1/3}\sqrt{b_m/a_m},
 \qquad h_m=\max\{h_m^\star,\,2K/\lambda_m\},
\]
and set $h_m=\infty$ when $\gamma_m=0$. Then
\begin{align}
\label{eq:refresh-optimized}
\avgerr_T
\le {}&6(2/3)^{1/3}D\Hg
 +\frac{4LD^2}{K(K+1)}
 +4KD\sum_m\frac{a_m}{\lambda_m}
 +D\bar\gamma(K\tau+\Delta)+\transient_T .
\end{align}
Before the sample-count floors are imposed, $(h_m^\star)$ exactly
minimizes the statistical-drift expression in parentheses
in~\eqref{eq:refresh-average}.
\end{corollary}

The optimal lengths satisfy $h_m^\star/h_n^\star=\sqrt{b_ma_n/(a_mb_n)}$:
a client receives a longer memory when its weighted uncertainty is large
relative to its drift. For $\tau>0$ one may take $\Delta=K\tau$, and the
nontransient remainders in~\eqref{eq:refresh-optimized} are then at most
$4LD^2/K^2+DK(4\sum_m a_m/\lambda_m+2\bar\gamma\tau)$; they record the
sample requirements and elapsed time of the gradient-message
implementation (the $K$-dependent communication charge disappears under
the raw-buffer implementation), and we do not claim that they are minimax
sharp; comparison~\eqref{eq:refresh-comparison} concerns the
statistical-drift term alone. The solver is not the mechanism: reset
projected SGD on the same disjoint batches, averaged over its post-update
iterates, satisfies the same guarantees with $LD^2/K$ in place of
$4LD^2/[K(K+1)]$ and $\beta_{\rm sgd}=2\sqrt2$ in place of $\beta_0$
(Proposition~\ref{prop:refresh-sgd}), so $\Hg$ is determined by the
memories and the independent batches, and acceleration only improves the
optimization remainder.

\section{Is client-specific memory necessary?}
\label{sec:refresh-lower}

Theorem~\ref{thm:refresh-upper} and inequality~\eqref{eq:refresh-comparison}
compare upper bounds. This section gives two converse statements: a
minimax lower bound that matches $D\Hg$, in an explicit regime, up to a
factor depending only on the number of distinct memory scales, and an
exact periodic calculation in which the best shared window loses to a
client-specific policy by an unbounded factor.

\subsection{A minimax lower bound with fixed clientwise drift}

\paragraph{The comparison class.}
For a fixed schedule $\sched$ and horizon $T$, let $\Pclass_T(\sched;\gamma)$
be the set of admissible instances: a public sample loss with its oracle,
and client stream laws satisfying Section~\ref{sec:refresh-model} with
the prescribed parameters, on the segment
$\cX_D=[-D/2,D/2]\times\{0\}^{d-1}$. Let $\Aclass_T(\sched)$ be the set
of causal rules that receive every raw sample as soon as it arrives, may
compute arbitrarily, and may use a private seed independent of the
samples; a rule may know the loss map, the public parameters, and the
schedule, is fixed before the hidden stream laws are chosen, and learns
the population path only through arrived observations. Writing
$\Risk_T(\alg,\inst\mid\sched)$ for the loss~\eqref{eq:refresh-risk} of
rule $\alg$ on instance $\inst$, the minimax risk is
\[
 \Risk_T^\star(\gamma)
 :=\inf_{\alg\in\Aclass_T(\sched)}\
   \sup_{\inst\in\Pclass_T(\sched;\gamma)}
   \Risk_T(\alg,\inst\mid\sched).
\]
Since the class grants immediate access to all arrived samples, every
lower bound on $\Risk_T^\star(\gamma)$ also applies to protocols that pay
for communication; since the supremum ranges over the loss, the bound is
not asserted for every externally fixed loss.

\paragraph{The construction.}
The hard instance runs independent triangular sign blocks on disjoint
client groups, each with its own clock, on one segment
(Theorem~\ref{thm:refresh-lower}, Appendix~\ref{app:formal-statements}).
Posterior uncertainty about the current signs stays separate across
groups, a moment argument converts it into expected linear optimization
loss, and a vector integral inequality preserves the loss over the
horizon, turning a hard instance at one instant into a time-average
bound.

\begin{corollary}[Comparison across memory scales]
\label{cor:refresh-dyadic}
Under the positivity assumptions of Theorem~\ref{thm:refresh-lower},
suppose that at least one $\gamma_m>0$. For the clients with
$\gamma_m>0$, let $q_m=\sqrt{b_m/a_m}$ and partition them into $J$
nonempty dyadic bins $\mathcal J_1,\ldots,\mathcal J_J$ of the values
$q_m$. For each bin put
\[
 Q_j=\sum_{m\in\mathcal J_j}\frac{\lambda_m\gamma_m^2}{\sigma_m^2},
 \quad T_j=(128Q_j)^{-1/3},\quad
 C_j=\frac{T_j}{2}\sum_{m\in\mathcal J_j}w_m\gamma_m,
 \quad
 \mathcal S_j=\sum_{m\in\mathcal J_j}\sqrt{a_mb_m}.
\]
If $T_j\ge2/\min_{m\in\mathcal J_j}\lambda_m$ for every $j$,
$T\ge\max_jT_j$, and $G\ge\sum_jC_j$, then
\begin{align}
 \Risk_T^\star(\gamma)
 &\ge c_{\rm het}D
       \Bigl(\sum_{j=1}^J\mathcal S_j^{4/3}\Bigr)^{1/2}
       \label{eq:refresh-group-lower}\\
 &\ge c_{\rm het}D\,\frac{\Hg}{J^{1/6}},
 \qquad c_{\rm het}=\frac{1/\sqrt3-1/4}{32\cdot128^{1/3}} .
       \label{eq:refresh-dyadic-lower}
\end{align}
The bins can be chosen so that
$J\le1+\lceil\log_2(q_{\max}/q_{\min})\rceil$.
\end{corollary}

The group-sensitive bound~\eqref{eq:refresh-group-lower} remembers how
the mass $\Sg=\sum_j\mathcal S_j$ is spread across scales; the
simplification~\eqref{eq:refresh-dyadic-lower} is H\"older's inequality,
tight for equal bin masses and sharper otherwise. Combined
with~\eqref{eq:refresh-optimized}, it gives a constant-factor match for
bounded $J$ whenever the optimization, timing, and transient terms of the
upper bound are $O(D\Hg)$; Corollary~\ref{cor:refresh-joint}
(Appendix~\ref{app:formal-statements}) makes these conditions explicit,
under them
$c_{\rm het}D\Hg/J^{1/6}\le\Risk_T^\star(\gamma)\le[6(2/3)^{1/3}+3]D\Hg$,
and Appendix~\ref{app:refresh-regime} exhibits a family that satisfies
them with arbitrarily many scales and positive communication time. On
that family no retuning of the disjoint-group construction closes the
$J^{1/6}$ gap (a limitation of the construction, not a minimax
impossibility); whether the factor is necessary for unrestricted scales
remains open. Appendix~\ref{app:refresh-lower} also lower-bounds the
finite-horizon cost of stationary clients
(Proposition~\ref{prop:refresh-stationary}), and
Appendix~\ref{app:formal-statements} treats samples whose source time
precedes their arrival.

\subsection{An exact separation for uniform completed windows}
\label{sec:periodic-memory}

Inequality~\eqref{eq:refresh-comparison} does not by itself show that
a common-memory method loses. We therefore compare actual risks within a
specified estimator class: reset empirical minimizers that weight every
observation in each client's latest completed buffer uniformly and hold
their output until a buffer changes; a \emph{shared} window uses one
fixed length $h$ at every client, and the class also contains the full
prefix.

\begin{proposition}[A strict separation within this estimator class]
\label{prop:periodic-main}
Let $\ell(x;Y)=x^2/2-Yx$ on $[-2,2]$, with two equally weighted clients
and $\tau=0$. For an integer $r\ge2$, client one has mean
$\cos(t+\Phi)$, rate $r$, and Rademacher noise of scale $1/r$; client
two has mean zero, rate one, and independent Rademacher noise of scale
$\sqrt2$. The phase $\Phi$ is arbitrary and fixed. Every finite shared
window $h\ge1$ has limiting inferior risk at least $1/16$, and the full
prefix attains $1/16$. The separate choice $h_1=1/r$, $h_2=\infty$ has
long-run risk
\begin{equation}
\label{eq:periodic-main}
 \frac18\Bigl[1-\frac{\sin(1/r)}{1/r}+\frac1{r^2}\Bigr]
 \le\frac7{48r^2}.
\end{equation}
Hence the ratio between the best shared-window risk and the risk of this
separate policy is at least $3r^2/7$, on a fixed compact domain.
\end{proposition}

The proof (Appendix~\ref{app:periodic-memory}) separates the
within-block variation of the signal from the variance of uniform sample
averages: the stationary client's noise makes every finite shared window
worse than the prefix, while the changing client benefits from short,
accurate buffers. The same calculation yields an exact frequency-response
criterion for when a finite window beats the prefix at a single client
(the reduction in squared temporal bias must exceed the variance
$\sigma_m^2/n_m$ of the average of the $n_m$ window samples), which
predicts failures as well as benefits.
The separation concerns uniform completed-window estimators; a learner
with common access to the full history can implement client-specific
weights itself and is not covered.

\section{Choosing among joint memories without drift inputs}
\label{sec:adaptive-memory}

Corollary~\ref{cor:refresh-opt} identifies good memories when the drift
bounds are available; when they are not, a finite collection of joint
memory vectors can be combined on the basis of an independent validation
stream, under one additional assumption: the sample loss has bounded
oscillation at each client,
$|\ell(x;\xi)-\ell(y;\xi)|\le B_m$ for all $x,y\in\cX$, and
$B:=\sum_m w_mB_m>0$. Bounded-feature logistic loss on a compact domain
satisfies this. The noise and rate bounds remain known, but the learner
does not know $\gamma$.

\paragraph{Procedure.}
Alternate arrivals at each client are assigned to a training and a
validation stream. For each of $E$ fixed memory vectors $h^1,\ldots,h^E$,
a separate reset SGD \emph{state} runs on the shared training array with
disjoint batches within each solve; these experts never see validation
samples or mixture weights. The server reports a convex mixture of the
expert outputs. Fresh validation samples from each reporting interval
score all experts, and two interleaved exponential-weight forecasters
(for even and for odd intervals) ensure that a full interval of feedback
has arrived before the forecaster that used it reports again. The
aggregation principle is standard \citep{freund1997hedge}; the
independent split and a source-time correction make it valid here.
Appendix~\ref{app:adaptive-memory} specifies the communication pipeline
and proves the complete form of the following statement
(Theorem~\ref{thm:adaptive-memory}).

\begin{theorem}[Oracle comparison for joint memory choices]
\label{thm:adaptive-summary}
Assume the model of Section~\ref{sec:refresh-model} with the bounded
oscillation above, except that the drift bounds need not be known. Let
$s_0=6K/\lambda_{\min}$, $c=K\tau$,
$\Delta\ge\max\{c,\,6/\lambda_{\min}\}$, and $T=s_0+c+R\Delta$ for an
integer $R\ge1$. Choose $E\ge2$ deterministic memory vectors,
independently of the data, with $6K/\lambda_m\le h_m^e\le T$, and let
$\Risk_T^e$ be the risk of the training-stream SGD expert $e$, including
initialization. The risk $\Risk_T^{\rm mix}$ of the causal mixture
satisfies
\begin{equation}
\label{eq:adaptive-main}
 \Risk_T^{\rm mix}
 \le \min_{e\in[E]}\Risk_T^e
       +2B\sqrt{\frac{\Delta\log E}{T}}
       +D\bar\gamma\Delta .
\end{equation}
Neither the updates nor the learning rate uses a drift bound.
\end{theorem}

Any interval whose length $v-u$ is at least $6/\lambda_m$ contains at
least $\lambda_m(v-u)/6$ training samples of client $m$, so every expert
enjoys the fixed-memory SGD guarantee with $b_m$ replaced by $3b_m$. The
comparator in~\eqref{eq:adaptive-main} therefore retains the joint
bias--variance expression
$D\min_e\{2\sum_m a_mh_m^e+2\sqrt2\sqrt{\sum_m3b_m/h_m^e}\}$, plus the
SGD, timing, and initialization terms of
Proposition~\ref{prop:refresh-sgd}: the pooled variance is not replaced
by a sum of clientwise errors. A Cartesian geometric grid approximates
the continuous minimum of this expression within a factor of two and
contains the full prefix $h_m=T$ for every client.

\paragraph{What the theorem does and does not give.}
It compares with the best \emph{fixed} joint memory vector over the
stated horizon, not with a sequence of locally best windows; the
validation split consumes observations, the $E$ expert solves can be
expensive, and keeping the union of their buffers may require the full
training history. It establishes statistical adaptation under a stronger
loss assumption, not an efficient storage policy.

\section{Experiments: when does forgetting help?}
\label{sec:refresh-exp}

Two networks with measured drift test Corollary~\ref{cor:refresh-opt}
from opposite sides. At climate stations the model of
Section~\ref{sec:refresh-model} holds by construction and the windows
are used exactly as prescribed; at clinics the drift is realistic, the
assumptions hold only approximately, and the solver is
Algorithm~\ref{alg:refresh}. All comparisons are paired. We write
$g=1-\Hg/(\bar\gamma^{1/3}\Xi^{2/3})$ for the relative gap in the
comparison $\Hg\le\bar\gamma^{1/3}\Xi^{2/3}$ of
Section~\ref{sec:refresh-model}, the gain that the upper bounds
predict.

\begin{figure}[t]
\centering
\includegraphics[width=0.85\linewidth]{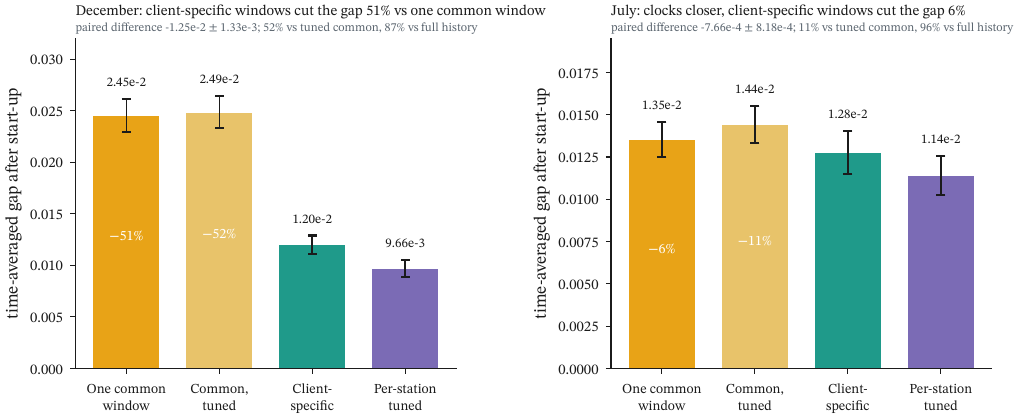}
\caption{Temperature stations: time-averaged gap over twelve evaluation
seeds with $95\%$ intervals (Appendix~\ref{app:station-model}). ``One
common window'' has the bound-optimal length; client-specific windows
follow Corollary~\ref{cor:refresh-opt} without tuning; per-station
tuning is an oracle. December has $g=0.32$, July $g=0.066$; subtitles
give paired differences.}
\label{fig:station-main}
\end{figure}

\paragraph{Temperature stations.}
Three stations of the U.S.\ Climate Reference Network
\citep{diamond2013uscrn}, chosen by insolation before any simulation,
report every five minutes: a desert, a maritime site, and a site in
polar night. A station's population path is its diurnal cycle fitted to
ten Decembers, and a draw adds a resampled day-to-day residual, so the
drift and noise bounds hold by construction. The loss is nearly linear,
and one round per refresh suffices ($K=1$, where
Algorithm~\ref{alg:refresh} and SGD coincide). The desert cycle is
fast, the maritime one slow, the polar path constant: $g=0.32$, and the
prescribed windows are $0.56$ and $1.28$ hours and the full history,
which no single length can offer. Used without a multiplier, they
reduce the gap by $51\%$ against the bound-optimal common window and by
$52\%$ against one tuned on separate seeds
(Figure~\ref{fig:station-main}, left). In July the cycles align and $g$ falls to $0.066$. Against the
bound-optimal common window the reduction is $6\%$, and the paired
interval $[-0.00158,0.00005]$ contains zero. Against the window tuned
on separate seeds the reduction is $11\%$, with interval
$[-0.0025,-0.0008]$ (right).
Across $200$ random station triplets the realized gain grows with $g$
(Spearman $0.73$), and for $g$ near zero a tuned common window can win
by a few percent: the comparison of upper bounds predicts both where
client-specific memory pays and where it does not
(Appendix~\ref{app:station-model}).

\paragraph{Recorded order.}
Replayed in recorded order, held-out Decembers leave the model:
persistent weather acts as drift, and the climatological inputs, which
keep the whole history at the polar site, lose in every year to a
common window tuned on the other nine. A drift input on the time scale
of weather, computed on the training folds, matches that window without
any search (Appendix~\ref{app:station-order}): what fails is the drift
input, not the window rule.

\paragraph{Clinics.}
Twelve clinics in six specialties share a logistic model with one
intercept per clinic that predicts an emergency or inpatient encounter
within 90 days from synthetic records \citep{walonoski2018synthea}. A
cohort the model does not observe, patients with recent acute care,
follows public influenza and COVID-19 surveillance curves, so the
optimum moves at speeds that differ by two orders of magnitude across
clinics, while the specialized clinics, which see the fewest patients,
are stationary. The drift and noise inputs are evaluated along the
exact optimal trajectory rather than certified, and the single scale of
every memory is tuned on separate streams, so this experiment compares
the shape of the memory, not its constant. Under exact minimization of
the buffered risk, client-specific windows reduce the gap by $33\%$
against the best common window, which is the full history, and by
$26\%$ against a common window with a per-clinic data floor, reaching
$89\%$ of the improvement of tuning all twelve windows. With
Algorithm~\ref{alg:refresh} ($K=8$) the optimization error dilutes
every difference: the reduction is $16\%$ against the best common
window, and the data-floored common window comes within $1\%$
(Appendix~\ref{app:clinics}). Where drift is seasonal and returns to
its mean, the worst-case prescription is conservative: per-clinic
tuning keeps the full history for pediatrics and primary care, where
Corollary~\ref{cor:refresh-opt} forgets.

\section{Conclusion and limitations}
\label{sec:concl}

Client-specific memory separates the statistical uncertainty of stable
clients from the forgetting that drifting clients require: $\Hg$ is
attainable with explicit finite-horizon costs, a lower construction with
unequal drift clocks retains each group's contribution, an exact periodic
calculation separates common from client-specific windows within a
specified estimator class, and independent validation competes with a
finite family of fixed memories without any drift bound as input.
Several limitations remain. Whether the $J^{1/6}$ factor between the
upper and lower bounds is necessary for unrestricted memory scales is
open; adaptive aggregation needs a bounded-loss assumption, spends samples
on validation, and may require many experts and full-history storage; the
experiments are the station laws, the clinic network, and the controlled
synthetic studies. Full history can be preferable under smooth, seasonal,
or shared drift. Temporal dependence, efficient
selection over large joint grids, and communication budgets that exclude
raw-data pooling call for further analysis; the source-age extension
assumes an unbiased oracle for the source time rather than inferring
censored labels from silence.

\subsection*{AI use statement}

Generative AI tools were used in preparing this manuscript: editing text and partial proof verification via Lean. The authors review the results and take full responsibility for them.

\subsection*{Ethics statement}

The experiments use public air-temperature records of the U.S.\ Climate
Reference Network, synthetic electronic health records generated with
Synthea, and aggregate public surveillance series (CDC ILINet and HHS
COVID-19 hospital admissions, retrieved through the Delphi Epidata
service \citep{farrow2015epidata}). No real patient-level data were
used, no new human-subject data were collected, and no intervention was
performed. The clinic experiment tests an optimization mechanism; it is
not a clinical prediction rule.

\subsection*{Reproducibility statement}

The model and the algorithm are specified in
Sections~\ref{sec:refresh-model} and~\ref{sec:refresh}; a table of
notation is in Appendix~\ref{app:notation}. Complete proofs appear in
Appendices~\ref{app:refresh-upper}--\ref{app:adaptive-memory}. The
supplement contains the experiment and verification scripts, the
scripts that retrieve the station records and the surveillance series,
disjoint tuning and evaluation seeds, seed-level outputs, and the
protocols of Appendices~\ref{app:clinics}
and~\ref{app:station-temperature}. Numerical bound diagnostics evaluate
every remainder and check the joint regime without simulating its long
arrival paths. Appendix~\ref{app:benchmarks} summarizes the benchmark
comparisons and the scope of the supplementary verification.

\section*{Acknowledgements}
The work of Yuriy Dorn was supported by the MSU Institute for Artificial Intelligence.
The work of I. Kuruzov and A. Gasnikov was supported by the Ministry of Science and Higher Education of the Russian Federation (Grant No. 075-15-2024-529).

\bibliography{refs}
\bibliographystyle{iclr2027_conference}

\clearpage
\appendix
\section{Notation}
\label{app:notation}

Tables~\ref{tab:notation-a} and~\ref{tab:notation-b} list the symbols
used in the main text and in the proofs, with the place where each is
introduced. Every symbol has one meaning throughout the paper;
experiment-specific constants (introduced in
experiment-specific constants (introduced in
Appendices~\ref{app:clinics} and~\ref{app:station-temperature}) use
letters that do not appear in the theory. We follow the usual conventions
of the optimization literature: $L$ is the Lipschitz constant of the
gradient, $\eta$ denotes stepsizes and learning rates, $\delta$ the
inexactness (bias) of a gradient oracle, $\varepsilon$ centered noise,
$\xstar$ a minimizer, and iterates carry the iteration index as a
superscript; no result uses strong convexity, so the strong-convexity
parameter $\mu$ and the condition number $\kappa$ do not appear. All
norms are Euclidean, $[n]=\{1,\ldots,n\}$, and $\ii$ denotes the
imaginary unit.

\begin{table}[!htbp]
\caption{Symbols of the model, the algorithm, and the upper bounds.}
\label{tab:notation-a}
\centering\footnotesize
\renewcommand{\arraystretch}{1.1}
\begin{tabular}{@{}p{0.26\linewidth}p{0.54\linewidth}p{0.13\linewidth}@{}}
\toprule
Symbol & Meaning & Where\\
\midrule
$M$, $m$, $w_m$ & number of clients, client index, public weights ($\sum_mw_m=1$) & Sec.~\ref{sec:refresh-model}\\
$\cX$, $d$, $D$, $\proj$ & compact convex domain in $\mathbb R^d$, its diameter, Euclidean projection & Sec.~\ref{sec:refresh-model}\\
$\ell(x;\xi)$, $\mathcal D_m^t$ & public sample loss; law of a sample of client $m$ generated at time $t$ & Sec.~\ref{sec:refresh-model}\\
$f_m^t$, $F^t$ & population objective of client $m$ and weighted objective at time $t$ & Sec.~\ref{sec:refresh-model}\\
$L$, $G$ & smoothness constant of $F^t$; public bound on $\|\nabla F^t\|$ & Sec.~\ref{sec:refresh-model}\\
$\lambda_m$, $\lambda_{\min}$, $N_m(t)$ & arrival rate, smallest rate, arrival count of client $m$ & Sec.~\ref{sec:refresh-model}\\
$\sigma_m^2$, $\gamma_m$, $\gamma$, $\bar\gamma$ & gradient-noise variance bound, clientwise drift bound, drift vector, $\sum_mw_m\gamma_m$ & Sec.~\ref{sec:refresh-model}\\
$\tau$, $T$, $\avgerr_T$ & synchronization time, horizon, time-averaged current-objective gap & Sec.~\ref{sec:refresh-model}\\
$a_m$, $b_m$, $\Sg$, $\Hg$, $\Xi$ & $w_m\gamma_m$, $w_m^2\sigma_m^2/\lambda_m$, $\sum_m\sqrt{a_mb_m}$, $\Sg^{2/3}$, $\Xi^2=\sum_mb_m$ & \eqref{eq:refresh-quantities}\\
\addlinespace
$K$, $h_m$, $h_m^\star$, $\Delta$ & rounds per solve, memory length of client $m$ and its ideal value, reporting stride & Sec.~\ref{sec:refresh-method}\\
$s_0$, $s_j$, $c$, $t_j$, $\Rint_j$ & first solve start, start of solve $j$, solve duration $K\tau$, report time $s_j+c$ ($\xinit$ is reported before $t_0$), report interval $[t_j,t_{j+1})$ & Sec.~\ref{sec:refresh-method}\\
$h_{m,j}$, $n_{m,j}$ & length $\min\{s_j,h_m\}$ and sample count of the frozen buffer of client $m$ at solve $j$ & Sec.~\ref{sec:refresh-method}\\
$\batch_{m,j,k}$, $n_{m,j,k}$, $t_{m,i}$ & $k$th batch of that buffer, its size, timestamp of sample $i$ & Alg.~\ref{alg:refresh}\\
$\xinit$, $x^k,v^k,y^k$, $\theta_k$, $g^k$ & reset point, accelerated iterates, momentum weight $2/(k+1)$, aggregated batch gradient & Alg.~\ref{alg:refresh}\\
$V_{j,k}$, $S_j$, $\eta_j$, $x_j$ & batch variance bound, $\sum_kk^2V_{j,k}$, stepsize, output of solve $j$ & Alg.~\ref{alg:refresh}\\
$\beta_0$, $\transient_T$ & noise constant $4\sqrt{2/3}$, finite-horizon transient & Thm.~\ref{thm:refresh-upper}\\
$z_j^k$, $\bar z_j$, $U_j$, $\eta_j^{\rm sgd}$, $\beta_{\rm sgd}$ & SGD iterates and output, $\sum_kV_{j,k}$, SGD stepsize, noise constant $2\sqrt2$ & Prop.~\ref{prop:refresh-sgd}\\
$\Filt_k$, $\delta_k$, $\bias$, $\varepsilon_k$, $V_k$ & filtration, gradient bias (oracle inexactness) and its bound, centered noise and its variance bound (Lemmas) & App.~\ref{app:refresh-upper}\\
$\bar V_j$, $\bias_j(t)$ & buffer-level variance proxy and bias bound of solve $j$ at time $t$ & App.~\ref{app:refresh-upper}\\
$\Delta^{\rm src}_m$ & source-age bound of client $m$ & App.~\ref{app:formal-statements}\\
$\Cg$ & count-floor and communication charge $4\sum_ma_m/\lambda_m+2\bar\gamma\tau$ & Cor.~\ref{cor:refresh-joint}\\
\bottomrule
\end{tabular}
\end{table}

\begin{table}[!htbp]
\caption{Symbols of the lower bounds, the selection procedure, and the periodic streams.}
\label{tab:notation-b}
\centering\footnotesize
\renewcommand{\arraystretch}{1.1}
\begin{tabular}{@{}p{0.26\linewidth}p{0.54\linewidth}p{0.13\linewidth}@{}}
\toprule
Symbol & Meaning & Where\\
\midrule
$\sched$, $\Pclass_T(\sched;\gamma)$, $\Aclass_T(\sched)$ & schedule, admissible instances, causal rules & Sec.~\ref{sec:refresh-lower}\\
$\cX_D$, $\Risk_T(\alg,\inst\mid\sched)$, $\Risk_T^\star(\gamma)$ & segment $[-D/2,D/2]\times\{0\}^{d-1}$, risk of rule $\alg$ on instance $\inst$, minimax risk & Sec.~\ref{sec:refresh-lower}\\
$\mathcal J_j$, $J$, $q_m$, $Q_j$ & client groups (dyadic bins), their number, memory scale $\sqrt{b_m/a_m}$, bin information rate & Cor.~\ref{cor:refresh-dyadic}\\
$T_j$, $A_m$, $C_j$, $\mathcal S_j$, $c_{\rm het}$ & block length and amplitudes of group $j$, $\sum_{m\in\mathcal J_j}w_mA_m$, bin mass $\sum_{m\in\mathcal J_j}\sqrt{a_mb_m}$, constant & Thm.~\ref{thm:refresh-lower}\\
$\Lambda$, $\Lambda_j(t)$, $\bsign_{j,k}$, $\bsign_j(t)$ & triangular profile, its block version, block signs & App.~\ref{app:refresh-lower}\\
$\zeta_j(t)$, $\Theta(t)$, $\Filt_t$, $\bar\bsign_j(t)$ & $C_j\Lambda_j(t)$, population coefficient, pooled transcript, posterior mean of the sign & App.~\ref{app:refresh-lower}\\
\addlinespace
$B_m$, $B$ & loss-oscillation bounds, $\sum_mw_mB_m$ & Sec.~\ref{sec:adaptive-memory}\\
$E$, $e$, $h^e$, $\grid$, $\grid_m$ & number of experts, expert index, its memory vector, joint and per-client grids & Sec.~\ref{sec:adaptive-memory}\\
$R$, $x_j^e$, $\pi_{j,e}$, $\etaew$, $\vscore_{j,e}$ & number of report intervals, expert output, mixture weights, learning rate, validation score & App.~\ref{app:adaptive-memory}\\
$\mathcal V_{m,j}$, $n^{\rm val}_{m,j}$, $\overline F_j$, $F_j^{\rm val}$ & validation samples of interval $j$, their count, centered interval-average and validation-time objectives & App.~\ref{app:adaptive-memory}\\
$\Risk_T^e$, $\Risk_T^{\rm mix}$, $\widetilde b_m$, $\widetilde\Xi$, $\widetilde{\transient}_T$ & expert and mixture risks, $3b_m$, and the corresponding pooled uncertainty and transient & Thm.~\ref{thm:adaptive-memory}\\
\addlinespace
$Y_{m,i}$, $\bar\vartheta_m$, $A_m$, $\omega_m$, $\Phi_m$, $\varepsilon_{m,i}$ & observation, mean level, amplitude, frequency, phase, Rademacher sign & \eqref{eq:periodic-stream}\\
$r$, $t_{\mathrm{s}}$, $\vartheta(t)$, $\widehat\vartheta_m(t)$ & moving-client rate, start of the data-based reports ($\max_m1/\lambda_m$), population and empirical coefficients & App.~\ref{app:periodic-memory}\\
$\mathcal M_m(t)$, $n_m(t)$, $n_m$, $\mathcal K_m(t)$ & retained indices, their count, window count $\lambda_mh_m$, complex coefficient of the average & App.~\ref{app:periodic-memory}\\
$\varphi_m$, $\mathcal K_n(\varphi)$, $\Upsilon_n(\varphi)$, $\sinc$ & phase increment $\omega_m/\lambda_m$, window filter, bias factor, $\sin(z)/z$ & Thm.~\ref{thm:periodic-risk}\\
$\Qrisk_T$, $\Qrisk_\infty$, $\Qrisk_\infty^{\rm sep}$ & phase-averaged risk, its limit, limit of the separate policy & App.~\ref{app:periodic-memory}\\
\bottomrule
\end{tabular}
\end{table}

\paragraph{Public records.}
Appendix~\ref{app:station-temperature} uses two public temperature
records. The {U.S.} Climate Reference Network five-minute
air-temperature product \citep{diamond2013uscrn} is the
\texttt{subhourly01} release,
\url{https://www.ncei.noaa.gov/pub/data/uscrn/products/subhourly01},
Decembers and Julys 2010--2019 at the named stations and, for the
network study, every station with sufficient December coverage. The
Beijing control uses hourly temperature (\texttt{TEMP}) from the
twelve-site archive
\citep{zhang2017beijing,chen2017beijing}; meteorological fields in
that archive are matched to the nearest China Meteorological
Administration station. Appendix~\ref{app:clinics} uses synthetic
visits generated with Synthea \citep{walonoski2018synthea}. Mix
weights follow public {CDC} {ILINet} and {HHS} {COVID}-19 hospital
series \citep{cdc-ilinet,hhs-covid-hosp} retrieved from Delphi Epidata
\citep{reinhart2021covidcast}. No real patient record is used.

\clearpage
\section{Supporting statements}
\label{app:formal-statements}

This appendix collects the statements that the main text refers to but
does not display: the raw-buffer implementation, the projected-SGD
counterpart of
Theorem~\ref{thm:refresh-upper}, the general unequal-clock lower bound
behind Corollary~\ref{cor:refresh-dyadic}, the explicit regime in which
the upper and lower bounds are comparable, and the source-time extension
of the upper bounds. Proofs are given in
Appendices~\ref{app:refresh-upper}--\ref{app:refresh-regime}.

\subsection{The raw-buffer implementation}
\label{app:raw-buffer}

Messages are unrestricted, so clients with finitely encodable samples may
upload their frozen buffers, timestamps, and counts in one
synchronization at the start of each solve. The server then runs the
same ordered batch updates of Algorithm~\ref{alg:refresh} locally and
either waits to reproduce the report times of the gradient-message
implementation or reports after $\tau$. The latter schedule uses
$c=\tau$ and $\Delta\ge\tau$ in the proofs, with unchanged sample-count
floors and optimization error, so Theorem~\ref{thm:refresh-upper} and
Corollary~\ref{cor:refresh-opt} hold for it with $K\tau$ replaced by
$\tau$. The elapsed time $K\tau$ therefore describes the
gradient-message implementation and is not intrinsic to the model.

\subsection{Projected SGD with the same memories}

Acceleration improves the optimization remainder of
Theorem~\ref{thm:refresh-upper} from order $K^{-1}$ to order $K^{-2}$;
it does not change the statistical-drift quantity. The following
comparison keeps the reset rule, the disjoint batches, and the reporting
times of Algorithm~\ref{alg:refresh} and replaces only the solver.

\begin{proposition}[Projected SGD with the same memories]
\label{prop:refresh-sgd}
Under the assumptions of Theorem~\ref{thm:refresh-upper}, let $V_{j,k}$
be the batch variance bounds of Algorithm~\ref{alg:refresh} and set
\[
 U_j=\sum_{k=1}^K V_{j,k},\qquad
 \eta_j^{\rm sgd}=\min\{1/(2L),\,D/\sqrt{2U_j}\},
\]
with $\eta_j^{\rm sgd}=1/(2L)$ when $U_j=0$. In solve $j$, reset
$z_j^0=\xinit$ and let $g_j^k$ be the weighted gradient of the same
batches $\batch_{m,j,k}$ evaluated at $z_j^{k-1}$. Define
\begin{equation}
\label{eq:refresh-sgd-update}
 z_j^k=\proj\bigl(z_j^{k-1}-\eta_j^{\rm sgd}g_j^k\bigr),\qquad
 \bar z_j=\frac1K\sum_{k=1}^K z_j^k .
\end{equation}
Report $\bar z_j$ on $\Rint_j=[s_j+K\tau,s_{j+1}+K\tau)$ and $\xinit$
before the first report. Then the pointwise and average guarantees
\eqref{eq:refresh-pointwise}--\eqref{eq:refresh-transient} hold for this
trajectory with $x_j$ replaced by $\bar z_j$ and with the substitutions
\begin{equation}
\label{eq:refresh-sgd-replacements}
 \frac{4LD^2}{K(K+1)}\ \longmapsto\ \frac{LD^2}{K},
 \qquad \beta_0\ \longmapsto\ \beta_{\rm sgd}:=2\sqrt2 .
\end{equation}
\end{proposition}

Appendix~\ref{app:refresh-sgd} proves the proposition and optimizes its
windows: the statistical-drift value becomes $6\cdot2^{-1/3}D\Hg$, with
the same count-floor and timing charges as in
Corollary~\ref{cor:refresh-opt}. The average in
\eqref{eq:refresh-sgd-update} runs over post-update points, so every batch
affects the report. The first query $z_j^0=\xinit$ is fixed, so the batch
counts and the first gradient share one message, exactly as in
Algorithm~\ref{alg:refresh}.

\subsection{The general unequal-clock construction}

\begin{theorem}[Different drift clocks in one dimension]
\label{thm:refresh-lower}
Fix positive $(w_m,\lambda_m,\sigma_m)_{m}$, $D>0$, the drift vector
$\gamma$, and a $\lambda$-regular schedule. Choose disjoint nonempty
client groups $\mathcal J_1,\ldots,\mathcal J_J$, block lengths
$0<T_j\le T$, and amplitudes $A_m\ge0$ for the clients in these groups
such that
\[
 T_j\ge\frac{2}{\min_{m\in\mathcal J_j}\lambda_m},\qquad
 A_m\le\frac{\gamma_mT_j}{2}\ \ (m\in\mathcal J_j),\qquad
 C_j:=\sum_{m\in\mathcal J_j}w_mA_m,\qquad
 \sum_{j=1}^JC_j\le G,
\]
and
\[
 4T_j\sum_{m\in\mathcal J_j}\frac{\lambda_mA_m^2}{\sigma_m^2}
 \le\frac1{128}\qquad\text{for every }j .
\]
Then
\begin{equation}
\label{eq:refresh-multiscale}
 \Risk_T^\star(\gamma)
 \ge\frac{1/\sqrt3-1/4}{8}\,
    D\Bigl(\sum_{j=1}^J C_j^2\Bigr)^{1/2}.
\end{equation}
For every fixed rule and horizon, one deterministic admissible drift path
attains this bound.
\end{theorem}

The proof (Appendix~\ref{app:refresh-lower}) runs independent triangular
sign blocks with a different clock for each group on the same
one-dimensional segment. Posterior uncertainty about the current signs
stays separate across groups; a second- and fourth-moment argument turns
it into expected linear optimization loss; and a vector integral
inequality preserves the loss over the horizon. A hard bump at a single
time would not justify the integrated statement.

\subsection{An explicit joint comparison regime}
\label{sec:refresh-regime}

The next statement lists the conditions under which every term of the
upper bound is of order $D\Hg$, so that the upper and lower bounds are
comparable. Write
\[
 \Cg=4\sum_m\frac{a_m}{\lambda_m}+2\bar\gamma\tau
\]
for the per-round count-floor and communication charge.

\begin{corollary}[Joint upper and lower regime]
\label{cor:refresh-joint}
Consider the model of Section~\ref{sec:refresh-model} on $\cX_D$ with
positive $(\lambda_m,\sigma_m)_m$, $\Hg>0$, and $\tau>0$. Fix a
$\lambda$-regular schedule and the dyadic partition that defines $T_j$
and $C_j$ in Corollary~\ref{cor:refresh-dyadic}, and suppose
\[
 T_j\ge\frac{2}{\min_{m\in\mathcal J_j}\lambda_m}\quad(j\in[J]),
 \qquad G\ge\sum_{j=1}^JC_j .
\]
Choose
\[
 K=\Bigl\lceil\max\{1,\,2\sqrt{LD/\Hg}\}\Bigr\rceil,
 \qquad \Delta=K\tau,\qquad s_0=2K/\lambda_{\min},
\]
suppose $K\Cg\le\Hg$, and, with $\beta_0=4\sqrt{2/3}$, assume
\begin{equation}
\label{eq:refresh-joint-horizon}
\begin{aligned}
 T&\ge\max\Bigl\{\max_jT_j,\ s_0+K\tau,\
                 \frac{16\beta_0^2\Xi^2}{\Hg^2}\Bigr\},\\
 T&\ge\frac{2\bigl[\beta_0\Xi K\tau/\sqrt{s_0}+G(s_0+K\tau)\bigr]}{\Hg}.
\end{aligned}
\end{equation}
With the windows of Corollary~\ref{cor:refresh-opt},
Algorithm~\ref{alg:refresh} satisfies $\avgerr_T\le[6(2/3)^{1/3}+3]D\Hg$
on every admissible instance, and
\begin{equation}
\label{eq:refresh-joint}
 c_{\rm het}\frac{D\Hg}{J^{1/6}}
 \le\Risk_T^\star(\gamma)
 \le[6(2/3)^{1/3}+3]\,D\Hg .
\end{equation}
\end{corollary}

The condition $K\Cg\le\Hg$ controls the sample-count and communication
charges, and the horizon conditions control the startup and statistical
transients; neither follows from the lower-bound hypotheses alone.
Appendix~\ref{app:refresh-regime} gives a family with positive
communication time that satisfies all requirements for arbitrarily many
memory scales, and shows that on this family no retuning of the
disjoint-group construction removes the $J^{1/6}$ factor. The latter is a
limitation of the construction, not a proof that the factor is minimax
necessary.

\subsection{Source time and usable time}
\label{sec:refresh-delayed}

Suppose that a sample usable at time $t$ was generated at an exogenous
source time in $[\max\{0,t-\Delta^{\rm src}_m\},t]$, and that the independent
unbiased oracle assumptions~\eqref{eq:refresh-noise} hold relative to
the population at the source time. Then the bounds of
Theorem~\ref{thm:refresh-upper} and Corollary~\ref{cor:refresh-opt} gain
the additive term
\begin{equation}
\label{eq:refresh-source-age}
 D\sum_mw_m\gamma_m\Delta^{\rm src}_m .
\end{equation}
The proof (Appendix~\ref{app:refresh-upper}) adds the mean-gradient bias
$\gamma_m\Delta^{\rm src}_m$ at each client. Positive-only or censored labels would
require a separate unbiased correction and variance analysis; the
extension assumes such an oracle rather than deriving it.

\section{Proofs for client-specific refresh}
\label{app:refresh-upper}

The proof of Theorem~\ref{thm:refresh-upper} has two parts. We first
prove a projected accelerated potential for a stochastic oracle with
bounded gradient bias (Lemma~\ref{lem:refresh-accel}); as in the
inexact-oracle literature, $\delta$ denotes the inexactness of the
oracle, here a predictable bias of the batch gradient. We then apply the
lemma to one reset solve, with a filtration that exposes only the
disjoint batches of that solve; this is what allows data buffers to be
reused across solves. Corollary~\ref{cor:refresh-opt} follows from an elementary
window-functional lemma, and the section closes with the source-time
extension of Appendix~\ref{app:formal-statements}.

\subsection{Projected acceleration with bounded gradient bias}

\begin{lemma}[A biased stochastic oracle on a compact set]
\label{lem:refresh-accel}
Let $F$ be convex and $L$-smooth, $L>0$, on a compact convex set $\cX$
of diameter $D>0$. Start from $x^0=v^0=\xinit\in\cX$ and, for
$k=1,\ldots,K$, let
\begin{align*}
\theta_k&=\frac{2}{k+1},&
y^k&=(1-\theta_k)x^{k-1}+\theta_kv^{k-1},\\
v^k&=\proj(v^{k-1}-k\eta\,g^k),&
x^k&=(1-\theta_k)x^{k-1}+\theta_kv^k .
\end{align*}
Let $(\Filt_k)_{k\ge0}$ be a filtration such that $y^k,v^{k-1},x^{k-1}$
are $\Filt_{k-1}$-measurable, and suppose
\[
g^k=\nabla F(y^k)+\delta_k+\varepsilon_k,
\qquad \norm{\delta_k}\le\bias,
\]
where $\delta_k$ is $\Filt_{k-1}$-measurable and
\[
\E[\varepsilon_k\mid\Filt_{k-1}]=0,
\qquad
\E[\norm{\varepsilon_k}^2\mid\Filt_{k-1}]\le V_k
\]
for deterministic $V_k\ge0$. Then, for every $0<\eta\le1/(4L)$,
\begin{equation}
\label{eq:refresh-accel-general}
\E[F(x^K)-\min_\cX F]
\le\frac{2}{K(K+1)}
\Bigl(\frac{D^2}{2\eta}+\eta\sum_{k=1}^Kk^2V_k\Bigr)+\bias D .
\end{equation}
If $S=\sum_{k=1}^Kk^2V_k>0$, the choice $\eta=\min\{1/(4L),D/\sqrt{2S}\}$
gives
\begin{equation}
\label{eq:refresh-accel-tuned}
\E[F(x^K)-\min_\cX F]
\le\frac{4LD^2+2\sqrt2D\sqrt S}{K(K+1)}+\bias D ,
\end{equation}
and for $S=0$ the same bound holds with $\eta=1/(4L)$.
\end{lemma}

\begin{proof}
Fix $\xstar\in\argmin_\cX F$ and write $r_k=\norm{v^k-\xstar}^2$.
Since $x^k-y^k=\theta_k(v^k-v^{k-1})$ and
$y^k=(1-\theta_k)x^{k-1}+\theta_kv^{k-1}$, smoothness together with the
two supporting-hyperplane inequalities at $y^k$ gives
\begin{align}
F(x^k)-F(\xstar)
&\le(1-\theta_k)[F(x^{k-1})-F(\xstar)]
+\theta_k\inner{\nabla F(y^k)}{v^k-\xstar}
+\frac{L\theta_k^2}{2}\norm{v^k-v^{k-1}}^2 .
\label{eq:refresh-accel-smooth}
\end{align}
The variational inequality of the projection that defines $v^k$ gives
\begin{equation}
\label{eq:refresh-accel-projection}
\inner{g^k}{v^k-\xstar}
\le\frac{r_{k-1}-r_k-\norm{v^k-v^{k-1}}^2}{2k\eta}.
\end{equation}
Substitute $\nabla F(y^k)=g^k-\delta_k-\varepsilon_k$ into
\eqref{eq:refresh-accel-smooth}. Since $v^k,\xstar\in\cX$, the bias
contributes $-\theta_k\inner{\delta_k}{v^k-\xstar}\le\theta_k\bias D$.
For the noise, write $v^k-\xstar=(v^{k-1}-\xstar)+(v^k-v^{k-1})$. The
stepsize restriction gives $L\theta_k\cdot k\eta\le2L\eta\le1/2$, so
after \eqref{eq:refresh-accel-projection} the coefficient of
$-\norm{v^k-v^{k-1}}^2$ is at least $\theta_k/(4k\eta)$, and Young's
inequality yields
\[
-\theta_k\inner{\varepsilon_k}{v^k-v^{k-1}}
\le\theta_k k\eta\norm{\varepsilon_k}^2
+\frac{\theta_k}{4k\eta}\norm{v^k-v^{k-1}}^2 .
\]
The squared-step terms cancel, and
$\E\inner{\varepsilon_k}{v^{k-1}-\xstar}=0$ by conditional centering.
Hence, with $\mathcal E_k=\E[F(x^k)-F(\xstar)]$,
\begin{equation}
\label{eq:refresh-accel-recursion}
\mathcal E_k\le(1-\theta_k)\mathcal E_{k-1}
+\frac{\theta_k}{2k\eta}\E[r_{k-1}-r_k]
+\theta_kk\eta V_k+\theta_k\bias D .
\end{equation}
Let $\Gamma_k=2/[k(k+1)]$. For $k\ge2$ we have
$(1-\theta_k)/\Gamma_k=1/\Gamma_{k-1}$, and for $k=1$ the coefficient of
$\mathcal E_0$ vanishes. Moreover,
\[
\frac{\theta_k}{2k\eta\,\Gamma_k}=\frac1{2\eta},
\qquad
\frac{\theta_kk\eta}{\Gamma_k}=k^2\eta,
\qquad
\frac{\theta_k}{\Gamma_k}=k .
\]
Dividing \eqref{eq:refresh-accel-recursion} by $\Gamma_k$, summing over
$k$, and using $r_0\le D^2$, $r_K\ge0$, and
$\Gamma_K\sum_{k=1}^Kk=1$ proves \eqref{eq:refresh-accel-general}.

If the tuned stepsize equals $D/\sqrt{2S}$, the two terms in parentheses
in \eqref{eq:refresh-accel-general} sum to $\sqrt2D\sqrt S$. If the cap
$1/(4L)$ is active, then $S\le8L^2D^2$ and
\[
\frac{D^2}{2\eta}+\eta S
=2LD^2+\frac S{4L}
\le2LD^2+\frac{D\sqrt S}{\sqrt2}.
\]
Both cases give \eqref{eq:refresh-accel-tuned}; the case $S=0$ follows
directly from \eqref{eq:refresh-accel-general}.
\end{proof}

The bias bound uses the compactness of $\cX$ at the projected auxiliary
point $v^k$, not only at the reported point $x^k$. An unprojected
acceleration recurrence therefore cannot be substituted without a
separate bound on its auxiliary iterates.

\subsection{Proof of Theorem~\ref{thm:refresh-upper}}

\begin{proof}
Condition on the complete exogenous arrival schedule. All clocks, buffer
choices, batch partitions, and variance-based stepsizes are then fixed
without revealing any sample value. Recall $s_j=s_0+j\Delta$, $c=K\tau$,
and $h_{m,j}=\min\{s_j,h_m\}$ with $\min\{s_j,\infty\}=s_j$.

\emph{Step 1: buffer ages and sample counts.}
If $s_j<h_m$, the frozen buffer of client $m$ is the prefix $(0,s_j]$, of
length $h_{m,j}=s_j$. If $h_m\le s_j$ and $h_m<\infty$, let
$p=\lfloor s_j/h_m\rfloor$; the frozen buffer is $((p-1)h_m,ph_m]$, of
length $h_{m,j}=h_m$, and since $ph_m\le s_j<(p+1)h_m$ every sample in
it has age less than $2h_m$ at the solve start. Infinite windows always
use the prefix. Hence every selected timestamp $t_{m,i}$ satisfies
\begin{equation}
\label{eq:refresh-age}
s_j-2h_{m,j}<t_{m,i}\le s_j .
\end{equation}
By $s_0=2K/\lambda_{\min}$ and $h_m\ge2K/\lambda_m$ we have
$h_{m,j}\ge2K/\lambda_m$, so the lower regularity inequality applies to
the frozen interval and gives
\begin{equation}
\label{eq:refresh-buffer-count}
n_{m,j}\ge\frac{\lambda_mh_{m,j}}2\ge K .
\end{equation}
Splitting the $n_{m,j}$ observations into $K$ batches whose sizes differ
by at most one yields batch sizes
\begin{equation}
\label{eq:refresh-split-count}
n_{m,j,k}\ge\Bigl\lfloor\frac{n_{m,j}}K\Bigr\rfloor
\ge\frac{n_{m,j}}{2K},
\end{equation}
where the last inequality uses $n_{m,j}/K\ge1$.

\emph{Step 2: the filtration of one solve.}
Fix a solve $j$. Its iterates reset to the public point $\xinit$, and
its updates use only the batches of this solve. Let $\Filt_{j,k}$ be
generated by the fixed schedule and the values of batches $1,\ldots,k$ of
this solve at all clients. Then $y^k$ and the earlier iterates are
$\Filt_{j,k-1}$-measurable, and conditional independence of the
time-stamped samples makes batch $k$ independent of the earlier batches
of the solve.

The same samples may have appeared in earlier solves. We do not condition
on those uses or on earlier reported outputs: the current solve is a
measurable function of its own frozen buffers, its partition, and the
fixed initialization, so this marginal filtration bounds its expected
risk even when buffers overlap across solves. Earlier outputs change
neither the marginal laws of the buffer samples nor the fact that each
batch is used once within this solve. Warm starts or sample-dependent
windows would require a different argument.

\emph{Step 3: the conditional oracle at a fixed report time.}
Fix $t\in\Rint_j=[s_j+c,s_{j+1}+c)$. The conditional mean of the
aggregate gradient at step $k$ is
\[
\E[g^k\mid\Filt_{j,k-1}]
=\sum_m\frac{w_m}{n_{m,j,k}}
\sum_{i\in\batch_{m,j,k}}\nabla f_m^{t_{m,i}}(y^k) .
\]
By the clientwise drift bounds and \eqref{eq:refresh-age},
\begin{align*}
\bigl\|\E[g^k\mid\Filt_{j,k-1}]-\nabla F^t(y^k)\bigr\|
&\le\sum_m\frac{w_m\gamma_m}{n_{m,j,k}}
\sum_{i\in\batch_{m,j,k}}(t-t_{m,i})\\
&\le 2\sum_m a_mh_{m,j}+\bar\gamma(t-s_j)
=:\bias_j(t).
\end{align*}
Conditionally on $\Filt_{j,k-1}$, the centered noises of different
samples and clients are independent, so their second moments add:
\[
\E\Bigl[\bigl\|g^k-\E[g^k\mid\Filt_{j,k-1}]\bigr\|^2
\Bigm|\Filt_{j,k-1}\Bigr]
\le V_{j,k}:=\sum_m\frac{w_m^2\sigma_m^2}{n_{m,j,k}} .
\]
Apply Lemma~\ref{lem:refresh-accel} with $F=F^t$ and $\bias=\bias_j(t)$.
With $S_j=\sum_{k=1}^Kk^2V_{j,k}$ and the solve output $x_j$, this gives
the exact-count bound
\begin{equation}
\label{eq:refresh-exact-proof}
\E[F^t(x_j)-\min_\cX F^t]
\le\frac{4LD^2+2\sqrt2D\sqrt{S_j}}{K(K+1)}
+D\Bigl(2\sum_m a_mh_{m,j}+\bar\gamma(t-s_j)\Bigr).
\end{equation}
The argument applies separately to every deterministic report time; it
yields the pointwise expectation needed below without any uniform
continuous-time high-probability statement.

\emph{Step 4: the heterogeneous count bound.}
Put $\bar V_j=\sum_mw_m^2\sigma_m^2/n_{m,j}$. From
\eqref{eq:refresh-split-count}, $V_{j,k}\le2K\bar V_j$, whence
\[
S_j\le2K\bar V_j\sum_{k=1}^Kk^2
\le\frac23\bar V_jK^2(K+1)^2 .
\]
Thus the statistical term in \eqref{eq:refresh-exact-proof} is at most
$4D\sqrt{\bar V_j}/\sqrt3$, and by \eqref{eq:refresh-buffer-count}
\[
\bar V_j\le2\sum_m\frac{b_m}{h_{m,j}} .
\]
Using $t-s_j<c+\Delta$ we arrive at
\begin{align}
\E[F^t(x_j)-\min_\cX F^t]
&\le\frac{4LD^2}{K(K+1)}
+\beta_0D\sqrt{\sum_m\frac{b_m}{h_{m,j}}}
+2D\sum_m a_mh_{m,j}
+D\bar\gamma(c+\Delta),
\label{eq:refresh-r-proof}
\end{align}
with $\beta_0=4\sqrt{2/3}$; this is \eqref{eq:refresh-pointwise}.

\emph{Step 5: time averaging, including prefixes and stationary clients.}
For $h_m=\infty$ we use the conventions $1/h_m=0$ and $a_mh_m=0$, the
latter only when $\gamma_m=0$. Then
\[
\frac1{h_{m,j}}=\frac1{\min\{s_j,h_m\}}
\le\frac1{s_j}+\frac1{h_m},
\qquad
\sum_m a_mh_{m,j}\le\sum_m a_mh_m,
\]
and therefore, with $\Xi^2=\sum_mb_m$,
\[
\sqrt{\sum_m\frac{b_m}{h_{m,j}}}
\le\sqrt{\sum_m\frac{b_m}{h_m}}+\frac{\Xi}{\sqrt{s_j}} .
\]
In particular, a moving client's prefix count is never replaced by
$\lambda_mh_m/2$ before its first complete window exists.

When $T\ge s_0+c$, the report intervals that meet $[0,T]$ have
consecutive indices $0,1,\ldots,j_T$. Since $s\mapsto s^{-1/2}$ is
decreasing,
\begin{align*}
\sum_{j=0}^{j_T}\frac{|\Rint_j\cap[0,T]|}{\sqrt{s_j}}
&\le\frac{\Delta}{\sqrt{s_0}}
+\sum_{j=1}^{j_T}\frac{\Delta}{\sqrt{s_0+j\Delta}}
\le\frac{\Delta}{\sqrt{s_0}}
+\int_{s_0}^{s_{j_T}}s^{-1/2}\,ds
\le\frac{\Delta}{\sqrt{s_0}}+2\sqrt T ,
\end{align*}
with the usual convention that an empty sum is zero. Before the first
report, the uniform gradient bound gives
$F^t(\xinit)-\min_\cX F^t\le GD$. Integrating \eqref{eq:refresh-r-proof}
over $[0,T]$, applying the preceding inequalities, and dividing by $T$
proves
\begin{align*}
\avgerr_T
\le{}&\frac{4LD^2}{K(K+1)}
+\beta_0D\sqrt{\sum_m\frac{b_m}{h_m}}
+2D\sum_m a_mh_m+D\bar\gamma(c+\Delta)\\
&+\beta_0D\,\Xi\Bigl(\frac2{\sqrt T}
+\frac{\Delta}{T\sqrt{s_0}}\Bigr)
+\frac{GD(s_0+c)}T ,
\end{align*}
which is \eqref{eq:refresh-average}--\eqref{eq:refresh-transient}.
All expectations were conditional on the fixed schedule; the same
inequalities hold unconditionally if the regularity assumption holds
almost surely. The clocks and snapshots use arrived data only, and the
$K$ rounds cost $c=K\tau\le\Delta$, so the solves and the reporting
trajectory are causal.
\end{proof}

\subsection{Proof of Corollary~\ref{cor:refresh-opt}}

\begin{lemma}[The window functional]
\label{lem:window-functional}
Let $\alpha,\beta>0$ and suppose $\Sg>0$. Over all $h_m>0$ indexed by
the clients with $a_mb_m>0$,
\[
\inf_h\Bigl\{\alpha\sum_m a_mh_m
+\beta\sqrt{\sum_m b_m/h_m}\Bigr\}
=3\cdot2^{-2/3}\alpha^{1/3}\beta^{2/3}\,\Hg ,
\]
attained at $h_m=\varrho_\star\sqrt{b_m/a_m}$ with
$\varrho_\star=(\beta/2\alpha)^{2/3}\Sg^{-1/3}$.
\end{lemma}

\begin{proof}
By the Cauchy--Schwarz inequality,
\[
\Bigl(\sum_m a_mh_m\Bigr)
\Bigl(\sum_m\frac{b_m}{h_m}\Bigr)
\ge\Bigl(\sum_m\sqrt{a_mb_m}\Bigr)^2
=\Sg^2 ,
\]
with equality when $h_m=\varrho\sqrt{b_m/a_m}$ for a common $\varrho>0$.
For a fixed value of the second sum, the first sum is therefore at least
the value attained by these proportional windows, so the infimum is
attained on them. Along the proportional windows the objective equals
$\alpha\varrho\Sg+\beta\sqrt{\Sg/\varrho}$, whose minimizer and minimum
are
\[
\varrho_\star=\Bigl(\frac{\beta}{2\alpha}\Bigr)^{2/3}\Sg^{-1/3},
\qquad
3\cdot2^{-2/3}\alpha^{1/3}\beta^{2/3}\Sg^{2/3}. \qedhere
\]
\end{proof}

\begin{proof}[Proof of Corollary~\ref{cor:refresh-opt}]
First suppose $\Sg>0$ and consider the clients with $a_mb_m>0$. The
coefficients in \eqref{eq:refresh-average} are $\alpha=2$ and
$\beta=\beta_0=4\sqrt{2/3}$, for which
$\varrho_\star=(2/3)^{1/3}\Sg^{-1/3}$. Lemma~\ref{lem:window-functional}
therefore gives the ideal windows
\[
h_m^\star=(2/3)^{1/3}\Sg^{-1/3}\sqrt{b_m/a_m},\qquad a_mb_m>0,
\]
with statistical-drift value $6(2/3)^{1/3}\Hg$.

To meet the finite-arrival constraint, set
$h_m=\max\{h_m^\star,2K/\lambda_m\}$ on these clients. Increasing a
window cannot increase its contribution to $\sum_mb_m/h_m$, and the
increase in the weighted drift cost is at most
\[
2\sum_m a_m(h_m-h_m^\star)
\le4K\sum_m\frac{a_m}{\lambda_m} .
\]
A client with $a_m>0=b_m$ receives $h_m=2K/\lambda_m$; its entire drift
cost is contained in the same bound. A client with $a_m=0$ receives
$h_m=\infty$; it has no limiting drift or variance contribution, while
its finite-prefix variance remains in the transient of
Theorem~\ref{thm:refresh-upper}. Restoring the factor $D$ and the
remaining terms of the theorem proves \eqref{eq:refresh-optimized}.

If $\Sg=0$, every client with positive drift has zero noise. Use
$h_m=2K/\lambda_m$ for those clients and $h_m=\infty$ for the zero-drift
clients. Then $\sum_mb_m/h_m=0$ and
$2\sum_ma_mh_m=4K\sum_ma_m/\lambda_m$, so the same bound holds with
$\Hg=0$, without interpreting a singular window formula.

Finally, inequality~\eqref{eq:refresh-comparison} is the Cauchy--Schwarz
bound
\[
\Sg=\sum_m\sqrt{a_mb_m}
\le\sqrt{\sum_ma_m}\sqrt{\sum_mb_m}
=\sqrt{\bar\gamma}\,\Xi ,
\]
raised to the power $2/3$; equality holds exactly when $(a_m)_m$ and
$(b_m)_m$ are proportional.
\end{proof}

\subsection{Separate source and usable timestamps}
\label{app:source-time}

For the extension of Appendix~\ref{app:formal-statements}
(Section~\ref{sec:refresh-delayed}), condition on the joint exogenous
array of usable and source timestamps. The buffer construction and the
count bounds use the usable times $t_{m,i}$. If the source time satisfies
$\max\{0,t_{m,i}-\Delta^{\rm src}_m\}\le t^{\rm src}_{m,i}\le t_{m,i}$, then every
selected sample obeys
\[
t-t^{\rm src}_{m,i}
=(t-t_{m,i})+(t_{m,i}-t^{\rm src}_{m,i})
\le t-s_j+2h_{m,j}+\Delta^{\rm src}_m .
\]
Its conditional mean gradient is $\nabla f_m^{t^{\rm src}_{m,i}}(y^k)$ by
the source-time oracle assumption, so the bias in Step~3 of the proof of
Theorem~\ref{thm:refresh-upper} is at most
$\bias_j(t)+\sum_m a_m\Delta^{\rm src}_m$. The conditional noise variance and all
usable-time counts are unchanged. Lemma~\ref{lem:refresh-accel}
consequently adds exactly $D\sum_ma_m\Delta^{\rm src}_m$ to each pointwise bound,
and at most that amount to the average bound. Since the source-age term
does not depend on the buffer lengths, the same addition applies to
Corollary~\ref{cor:refresh-opt}. The independence used here is
conditional on both timestamp arrays; the calculation does not supply an
unbiased oracle for censored or outcome-dependent label selection.

\section{Projected SGD with client-specific memories}
\label{app:refresh-sgd}

We first prove a projected-SGD potential for the bounded-bias oracle and
then apply the same marginal filtration, age bound, and count estimates
as in the accelerated analysis. This isolates the effect of the optimizer
from the clientwise memory allocation.

\begin{lemma}[A projected-SGD potential with bounded bias]
\label{lem:refresh-sgd-bias}
Let $F$ be convex and $L$-smooth, $L>0$, on a compact convex set $\cX$
of diameter $D>0$. Fix $z^0=\xinit\in\cX$, an integer $K\ge1$, and a
deterministic $0<\eta\le1/(2L)$. For $k=1,\ldots,K$ let
\[
 z^k=\proj(z^{k-1}-\eta g^k),\qquad
 g^k=\nabla F(z^{k-1})+\delta_k+\varepsilon_k,
\]
where $z^{k-1}$ and $\delta_k$ are $\Filt_{k-1}$-measurable,
$\|\delta_k\|\le\bias$, and
\[
 \E[\varepsilon_k\mid\Filt_{k-1}]=0,
 \qquad
 \E[\|\varepsilon_k\|^2\mid\Filt_{k-1}]\le V_k
\]
for finite deterministic $V_k\ge0$. Then, for
$\bar z=K^{-1}\sum_{k=1}^Kz^k$,
\begin{equation}
\label{eq:refresh-sgd-general}
 \E[F(\bar z)-\min_\cX F]
 \le\frac{D^2}{2\eta K}+\frac\eta K\sum_{k=1}^K V_k+\bias D .
\end{equation}
Write $U=\sum_{k=1}^K V_k$. With $\eta=\min\{1/(2L),D/\sqrt{2U}\}$ if
$U>0$, and $\eta=1/(2L)$ if $U=0$,
\begin{equation}
\label{eq:refresh-sgd-tuned}
 \E[F(\bar z)-\min_\cX F]
 \le\frac{LD^2}{K}+\frac{\sqrt2D\sqrt U}{K}+\bias D .
\end{equation}
\end{lemma}

\begin{proof}
Fix $\xstar\in\argmin_\cX F$, write $r_k=\|z^k-\xstar\|^2$, and combine smoothness at $z^{k-1}$ with the
supporting-hyperplane inequality for $F(\xstar)$:
\[
 F(z^k)-F(\xstar)
 \le\inner{\nabla F(z^{k-1})}{z^k-\xstar}
       +\frac L2\|z^k-z^{k-1}\|^2 .
\]
The projection variational inequality gives
\[
 \inner{g^k}{z^k-\xstar}
 \le\frac{r_{k-1}-r_k-\|z^k-z^{k-1}\|^2}{2\eta}.
\]
After substituting $\nabla F(z^{k-1})=g^k-\delta_k-\varepsilon_k$, the
bias term is bounded by $-\inner{\delta_k}{z^k-\xstar}\le\bias D$,
because both $z^k$ and $\xstar$ lie in $\cX$. Split the noise term as
\[
 -\inner{\varepsilon_k}{z^k-\xstar}
 =-\inner{\varepsilon_k}{z^{k-1}-\xstar}
  -\inner{\varepsilon_k}{z^k-z^{k-1}}.
\]
Since $\eta\le1/(2L)$, we have $1/(2\eta)-L/2\ge1/(4\eta)$, and Young's
inequality gives
\[
 -\inner{\varepsilon_k}{z^k-z^{k-1}}
 \le\eta\|\varepsilon_k\|^2+\frac1{4\eta}\|z^k-z^{k-1}\|^2 ,
\]
so the squared-step terms have a nonpositive sum. Conditional centering
removes $\E\inner{\varepsilon_k}{z^{k-1}-\xstar}$. Therefore
\begin{equation}
\label{eq:refresh-sgd-recursion}
 \E[F(z^k)-F(\xstar)]
 \le\frac{\E r_{k-1}-\E r_k}{2\eta}
       +\eta V_k+\bias D .
\end{equation}
Summing over $k$, using $r_0\le D^2$ and $r_K\ge0$, and applying
convexity to the average of the post-update points proves
\eqref{eq:refresh-sgd-general}; the accumulated bias term remains
$\bias D$ after division by $K$.

If the chosen stepsize is $D/\sqrt{2U}$, then
$D^2/(2\eta)+\eta U=\sqrt2D\sqrt U$. If the cap $1/(2L)$ is active, then
$U\le2L^2D^2$ and
\[
 \frac{D^2}{2\eta}+\eta U
 =LD^2+\frac U{2L}
 \le LD^2+\frac{D\sqrt U}{\sqrt2}.
\]
Both cases imply \eqref{eq:refresh-sgd-tuned}. When $U=0$, the cap in
\eqref{eq:refresh-sgd-general} gives $LD^2/K+\bias D$ directly.
\end{proof}

\subsection{Proof of Proposition~\ref{prop:refresh-sgd}}

\begin{proof}
Condition on the fixed exogenous regular schedule. The frozen windows,
timestamp-only partitions, counts, $V_{j,k}$, and $\eta_j^{\rm sgd}$ are
then deterministic. In solve $j$, reveal only its batches $1,\ldots,k$ at
step $k$, starting from the fixed $z_j^0=\xinit$. Then $z_j^{k-1}$ is
measurable with respect to the earlier batches, and the next batch is
independent of them. As in the accelerated proof, this is a marginal
analysis of one reset solve; earlier uses of the same observations in
other solves are not conditioned upon.

Fix a deterministic $t\in\Rint_j=[s_j+c,s_{j+1}+c)$, where $c=K\tau$.
The pointwise oracle assumptions, independence within the current batch,
and the age bound \eqref{eq:refresh-age} give
\begin{align*}
 \bigl\|\E[g_j^k\mid\Filt_{j,k-1}]-\nabla F^t(z_j^{k-1})\bigr\|
 &\le 2\sum_m a_mh_{m,j}+\bar\gamma(t-s_j)=:\bias_j(t),\\
 \E\Bigl[\bigl\|g_j^k-\E[g_j^k\mid\Filt_{j,k-1}]\bigr\|^2
       \Bigm|\Filt_{j,k-1}\Bigr]
 &\le V_{j,k}.
\end{align*}
The mean-gradient discrepancy is predictable and bounded by the
deterministic $\bias_j(t)$, so Lemma~\ref{lem:refresh-sgd-bias} applies
with $F=F^t$. Let $\bar V_j=\sum_mw_m^2\sigma_m^2/n_{m,j}$. The
balanced-batch and regular-count bounds \eqref{eq:refresh-split-count}
and \eqref{eq:refresh-buffer-count} imply
\[
 V_{j,k}\le2K\bar V_j,\qquad
 U_j\le2K^2\bar V_j\le4K^2\sum_m\frac{b_m}{h_{m,j}} .
\]
Thus, with $\beta_{\rm sgd}=2\sqrt2$ and $t-s_j<c+\Delta$,
\begin{align}
 \E[F^t(\bar z_j)-\min_\cX F^t]
 &\le\frac{LD^2}{K}
  +\beta_{\rm sgd}D\sqrt{\sum_m\frac{b_m}{h_{m,j}}}
  +2D\sum_m a_mh_{m,j}+D\bar\gamma(c+\Delta),
 \label{eq:refresh-sgd-pointwise}
\end{align}
which is the asserted pointwise replacement, including the zero-variance
case.

For the time average, the finite and infinite window conventions give,
as before,
\[
 \sum_m a_mh_{m,j}\le\sum_m a_mh_m,\qquad
 \sqrt{\sum_m\frac{b_m}{h_{m,j}}}
 \le\sqrt{\sum_m b_m/h_m}+\frac{\Xi}{\sqrt{s_j}} ,
\]
and the estimate from the proof of Theorem~\ref{thm:refresh-upper} gives
$\sum_j|\Rint_j\cap[0,T]|/\sqrt{s_j}\le\Delta/\sqrt{s_0}+2\sqrt T$.
During the initial interval of length $s_0+c$, the report is $\xinit$
and its gap is at most $GD$. Integrating \eqref{eq:refresh-sgd-pointwise},
for every $T\ge s_0+c$ we obtain
\begin{align}
 \avgerr_T
 &\le\frac{LD^2}{K}
  +D\Bigl(2\sum_m a_mh_m
       +\beta_{\rm sgd}\sqrt{\sum_m b_m/h_m}\Bigr)
  +D\bar\gamma(c+\Delta)+\transient_T^{\rm sgd},
 \label{eq:refresh-sgd-average}\\
 \transient_T^{\rm sgd}
 &:=\beta_{\rm sgd}D\,\Xi
      \Bigl(\frac2{\sqrt T}+\frac\Delta{T\sqrt{s_0}}\Bigr)
     +\frac{GD(s_0+c)}T .
 \label{eq:refresh-sgd-transient}
\end{align}
These are precisely the replacements
\eqref{eq:refresh-sgd-replacements} in
\eqref{eq:refresh-average}--\eqref{eq:refresh-transient}. Infinite windows
are used only when $a_m=0$; their stationary noise is retained in
$\transient_T^{\rm sgd}$.

Finally, the first query $z_j^0=\xinit$ does not depend on the stepsize,
so clients can return their batch-size vectors with its gradient, and the
server computes $U_j$ and $\eta_j^{\rm sgd}$ before the first update. The
implementation therefore uses the same $K$ gradient-message rounds and
report times as Algorithm~\ref{alg:refresh}, and consults no sample that
arrives after $s_j$.
\end{proof}

\subsection{Optimizing the SGD windows}

Under the positivity conditions of Corollary~\ref{cor:refresh-opt}, set
\begin{equation}
\label{eq:refresh-sgd-windows}
 h_m^{\star,\rm sgd}=2^{-1/3}\Sg^{-1/3}\sqrt{b_m/a_m},\qquad
 h_m=\max\{h_m^{\star,\rm sgd},2K/\lambda_m\}
\end{equation}
for $a_m>0$, and $h_m=\infty$ for $a_m=0$. Then
\begin{align}
 \avgerr_T
 \le{}&6\cdot2^{-1/3}D\Hg+\frac{LD^2}{K}
  +4KD\sum_m\frac{a_m}{\lambda_m}
 +D\bar\gamma(K\tau+\Delta)+\transient_T^{\rm sgd}.
 \label{eq:refresh-sgd-optimized}
\end{align}
Indeed, Lemma~\ref{lem:window-functional} with $\alpha=2$ and
$\beta=\beta_{\rm sgd}=2\sqrt2$ gives $\varrho_\star=2^{-1/3}\Sg^{-1/3}$
and the value $6\cdot2^{-1/3}\Hg$ for the statistical-drift expression
in \eqref{eq:refresh-sgd-average}. Raising an ideal window to its count
floor cannot increase the variance term and raises the weighted drift
cost by at most $4K\sum_m a_m/\lambda_m$, which proves
\eqref{eq:refresh-sgd-optimized}. As in the accelerated bound, a client
with $a_m>0=b_m$ receives $h_m=2K/\lambda_m$, and when $\Sg=0$ the
moving clients receive these finite windows and the stationary clients
infinite ones; the same bound then holds with $\Hg=0$.

The statistical-drift dependence therefore comes from the memory balance
and the independent-batch oracle alone. For $\Hg>0$, the optimization
term is at most $D\Hg$ under the sufficient integer choices
\[
 K_{\rm sgd}=\Bigl\lceil\max\{1,LD/\Hg\}\Bigr\rceil,
 \qquad
 K_{\rm acc}=\Bigl\lceil\max\{1,2\sqrt{LD/\Hg}\}\Bigr\rceil ,
\]
which for $LD/\Hg\ge1$ scale linearly and as a square root,
respectively. In both cases the sample-count, timing, and transient
terms must be controlled as well.

\section{Lower bounds under fixed clientwise drift}
\label{app:refresh-lower}

We prove Theorem~\ref{thm:refresh-lower} on the segment
$\cX_D=[-D/2,D/2]\times\{0\}^{d-1}$ of Section~\ref{sec:refresh-lower}.
The vectors indexed by client groups below organize the proof; they do
not enlarge the optimization domain, which uses one coordinate only. We
restrict the instance class to the public loss $\ell(x;\xi)=\xi x_1$.
This subclass is contained in the minimax supremum, and knowing its loss
map reveals nothing about its time-indexed sample laws.

Fix the exogenous schedule $\sched$, the horizon $T$, and one (possibly
randomized) causal rule $\alg\in\Aclass_T(\sched)$ before introducing
any prior. Every probability and divergence in this appendix is
conditional on the fixed schedule. A random sign array serves only to
average risks over admissible instances: each deterministic realization
of the array, not the mixture over arrays, is a member of the
independent-stream class $\Pclass_T(\sched;\gamma)$.

\subsection{Proof of Theorem~\ref{thm:refresh-lower}}

\paragraph{The instance and its drift bounds.}
Let
\[
 \Lambda(u)=1-|2u-1|,\qquad u\in[0,1],
\]
be the triangular profile: it is nonnegative, vanishes at the endpoints,
has Lipschitz constant $2$, and satisfies $\int_0^1\Lambda(u)\,du=1/2$.
For every group $j\in[J]$ and block $k=0,1,\ldots$, draw an independent
uniform sign $\bsign_{j,k}\in\{-1,+1\}$, and for $t\in[kT_j,(k+1)T_j)$
put
\[
 \Lambda_j(t)=\Lambda\bigl((t-kT_j)/T_j\bigr),\qquad
 \bsign_j(t)=\bsign_{j,k}.
\]
A client $m\in\mathcal J_j$ has sample law and loss
\[
 \xi\sim\mathcal N\bigl(A_m\bsign_j(t)\Lambda_j(t),\sigma_m^2\bigr)
 \quad\text{at time }t,
 \qquad \ell(x;\xi)=\xi x_1 ,
\]
and a client outside the selected groups has law
$\mathcal N(0,\sigma_m^2)$ at every time. For each fixed sign array, all
samples are independent across arrivals and clients. Their laws depend
on their timestamps but not on the query points, so repeated or adaptive
queries reveal no more than the raw arrived sample, which the pooled
transcript already contains.

For $m\in\mathcal J_j$ the population gradient is
$A_m\bsign_j(t)\Lambda_j(t)e_1$. The path is continuous across block
boundaries and piecewise linear with slopes of magnitude at most
$2A_m/T_j\le\gamma_m$, so it satisfies the clientwise Lipschitz
bound~\eqref{eq:refresh-drift} for all pairs of times, including pairs in
different blocks. Unselected clients have zero drift. All objectives are
linear, hence convex and $L$-smooth for every $L\ge0$, and the gradient
noise variance is exactly $\sigma_m^2$.

Set
\[
 \zeta_j(t)=C_j\Lambda_j(t),\qquad
 \zeta(t)=(\zeta_j(t))_{j=1}^J,\qquad
 \Theta(t)=\sum_{j=1}^J \zeta_j(t)\bsign_j(t).
\]
The population objective is $F^t(x)=\Theta(t)x_1$, and
\[
 \sup_{t,x}\norm{\nabla F^t(x)}
 =\sup_t|\Theta(t)|
 \le\sum_{j=1}^J C_j\le G .
\]
Thus every deterministic sign path meets all constraints of the class.
Each client's drift bound is respected separately; no public aggregate
budget is redistributed.

\paragraph{Information about a group's current sign.}
Fix a time $t$ and a group $j$. Conditionally on either value of
$\bsign_j(t)$, the samples of that group in its current block form a
product of Gaussians. The block has length $T_j\ge2/\lambda_m$ for every
$m\in\mathcal J_j$, so two-sided regularity bounds its count at client
$m$ by $2\lambda_mT_j$. Endpoint observations, if any, have mean zero and
carry no information about either adjacent sign. The KL divergence
between the two laws of a full block is bounded by
\begin{align}
 \label{eq:refresh-lb-block-kl}
 \kl\bigl(\Prob_{j,+}^{\rm block}\,\big\|\,\Prob_{j,-}^{\rm block}\bigr)
 &\le\sum_{m\in\mathcal J_j}\sum_{s\text{ in block}}
       \frac{2A_m^2\Lambda_j(s)^2}{\sigma_m^2}
 \le4T_j\sum_{m\in\mathcal J_j}
       \frac{\lambda_mA_m^2}{\sigma_m^2}
 \le\frac1{128}.
\end{align}
The law of the samples arrived by $t$ is a marginal of the full-block
law, so it obeys the same bound, and Pinsker's inequality gives
\begin{equation}
 \label{eq:refresh-lb-sign-tv}
 \tv\bigl(\Prob_{j,+}^t,\Prob_{j,-}^t\bigr)\le\frac1{16},
\end{equation}
where $\Prob_{j,\pm}^t$ are the laws of the group's current-block samples
arrived by $t$ under its two signs.

Let $\Filt_t$ be the pooled transcript: all samples arrived by $t$
together with the rule's private seed. Define the posterior mean
\[
 \bar\bsign_j(t)=\E[\bsign_j(t)\mid\Filt_t].
\]
Samples from earlier blocks of group $j$ are independent of its current
sign under the prior, and all raw streams of the other groups and the
private seed are independent of both that sign and its current-block
samples. Hence $\bar\bsign_j(t)$ is the posterior mean given the
current-block samples of group $j$ alone; the variables $\bar\bsign_j(t)$
are independent across groups and satisfy
$\E\bar\bsign_j(t)=\E\bsign_j(t)=0$.

For a uniform binary prior with posterior mean $\bar\bsign$, the identity
\begin{equation}
 \label{eq:refresh-lb-posterior-tv}
 \E|\bar\bsign|=\tv(\Prob_+,\Prob_-)
\end{equation}
follows from the mixture density: if $\rho_+,\rho_-$ are densities with
respect to a common dominating measure, then
$\bar\bsign=(\rho_+-\rho_-)/(\rho_++\rho_-)$ while the mixture density is
$(\rho_++\rho_-)/2$. Since $|\bar\bsign_j(t)|\le1$,
\eqref{eq:refresh-lb-sign-tv}--\eqref{eq:refresh-lb-posterior-tv} imply
\begin{equation}
 \label{eq:refresh-lb-posterior-second}
 \E\bar\bsign_j(t)^2\le\E|\bar\bsign_j(t)|\le\frac1{16}.
\end{equation}
These statements use the product structure of the raw group streams; they
do not assume that an algorithm's messages from different groups remain
independent after adaptive processing.

\paragraph{Posterior uncertainty implies optimization error.}
The minimum of $\Theta(t)x_1$ on $\cX_D$ is $-D|\Theta(t)|/2$. For any
causal report $x_t\in\cX_D$, conditioning its first coordinate on the
pooled transcript gives
\begin{align}
 \label{eq:refresh-lb-bayes-gap}
 \E\bigl[F^t(x_t)-\min_{\cX_D}F^t\bigr]
 &=\E[\Theta(t)(x_t)_1]+\frac D2\E|\Theta(t)|
 \ge\frac D2\Bigl(
       \E|\Theta(t)|-
       \E\bigl|\E[\Theta(t)\mid\Filt_t]\bigr|
       \Bigr).
\end{align}
This is the Bayes-optimal decision bound for a linear objective; it
already allows arbitrary randomized causal rules.

At a fixed time the current signs are independent Rademacher variables,
so
\[
 \E\Theta(t)^2=\|\zeta(t)\|^2,\qquad
 \E\Theta(t)^4
 =3\|\zeta(t)\|^4-2\sum_j\zeta_j(t)^4\le3\|\zeta(t)\|^4 .
\]
H\"older's inequality in the form
$(\E Z^2)^{3/2}\le\E|Z|\,(\E Z^4)^{1/2}$ therefore yields
\begin{equation}
 \label{eq:refresh-lb-prior-magnitude}
 \E|\Theta(t)|\ge\frac1{\sqrt3}\|\zeta(t)\| ;
\end{equation}
this is immediate when $\zeta(t)=0$, and otherwise follows by dividing
by the positive fourth-moment bound. On the other hand,
$\E[\Theta(t)\mid\Filt_t]=\sum_j\zeta_j(t)\bar\bsign_j(t)$, and the zero
means, independence across groups, and
\eqref{eq:refresh-lb-posterior-second} imply
\begin{align}
 \label{eq:refresh-lb-posterior-magnitude}
 \E\Bigl|\sum_j\zeta_j(t)\bar\bsign_j(t)\Bigr|
 &\le\Bigl(\sum_j\zeta_j(t)^2\,\E\bar\bsign_j(t)^2\Bigr)^{1/2}
 \le\frac14\|\zeta(t)\| .
\end{align}
Combining \eqref{eq:refresh-lb-bayes-gap}--\eqref{eq:refresh-lb-posterior-magnitude} proves the pointwise Bayes
bound
\begin{equation}
 \label{eq:refresh-lb-pointwise-bayes}
 \E\bigl[F^t(x_t)-\min_{\cX_D}F^t\bigr]
 \ge\frac D2\Bigl(\frac1{\sqrt3}-\frac14\Bigr)\|\zeta(t)\| .
\end{equation}

\paragraph{Different clocks and the horizon average.}
Because $T\ge T_j$, complete blocks of group $j$ cover at least half of
the horizon. The triangular profile has mean $1/2$ on a complete block
and is nonnegative on the remaining partial block, so
\[
 \frac1T\int_0^T\Lambda_j(t)\,dt
 \ge\frac{\lfloor T/T_j\rfloor T_j}{2T}\ge\frac14 .
\]
The triangle inequality for the vector integral of
$\zeta(t)=(C_j\Lambda_j(t))_j$ gives
\begin{align}
 \label{eq:refresh-lb-time-integral}
 \frac1T\int_0^T\|\zeta(t)\|\,dt
 &\ge\Bigl(\sum_{j=1}^J C_j^2
       \Bigl[\frac1T\int_0^T\Lambda_j(t)\,dt\Bigr]^2\Bigr)^{1/2}
 \ge\frac14\Bigl(\sum_{j=1}^J C_j^2\Bigr)^{1/2}.
\end{align}
This is the step that keeps the different group clocks inside one
time-averaged hard instance, instead of replacing the time average by a
different hard instance at every evaluation time.

Integrating \eqref{eq:refresh-lb-pointwise-bayes} and using
\eqref{eq:refresh-lb-time-integral} shows that the sign-averaged risk of
the fixed rule $\alg$ is at least
\[
 \frac{1/\sqrt3-1/4}{8}\,
 D\Bigl(\sum_{j=1}^J C_j^2\Bigr)^{1/2}.
\]
Only finitely many group-block signs affect $[0,T]$, so at least one
deterministic realization of them attains the average risk. Extend that
realization arbitrarily on later blocks, keeping the triangular profiles;
the extension preserves the drift and gradient bounds and is not revealed
to the rule. Because the rule and the horizon were fixed before the sign
averaging, this proves
\[
 \forall\alg\in\Aclass_T(\sched)\quad
 \exists\inst\in\Pclass_T(\sched;\gamma)
\]
with the asserted risk, and hence Theorem~\ref{thm:refresh-lower}. The
selected path may depend on the rule and the horizon; no single path is
claimed to be worst for all rules or horizons.\hfill$\qedsymbol$

\subsection{Proof of Corollary~\ref{cor:refresh-dyadic}}

Consider the clients with positive prescribed drift and variance. Since
$q_m=\sqrt{b_m/a_m}$, elementary substitution gives
\[
 a_m=\frac{\sqrt{a_mb_m}}{q_m},\qquad
 b_m=\sqrt{a_mb_m}\,q_m,\qquad
 \frac{\lambda_m\gamma_m^2}{\sigma_m^2}
 =\frac{a_m^2}{b_m}=\frac{\sqrt{a_mb_m}}{q_m^3}.
\]
In each bin set $A_m=\gamma_mT_j/2$, as in the corollary. The drift
constraint is then saturated, and
\[
 4T_j\sum_{m\in\mathcal J_j}
       \frac{\lambda_mA_m^2}{\sigma_m^2}
 =T_j^3Q_j=\frac1{128}.
\]
The remaining hypotheses of Theorem~\ref{thm:refresh-lower} are exactly
the count, horizon, and $G$ conditions of the corollary.

Let $\underline q_j$ be the lower endpoint of the $j$th dyadic bin, so
that $\underline q_j\le q_m\le2\underline q_j$ for all
$m\in\mathcal J_j$. Then
\begin{align}
 \label{eq:refresh-lb-bin-amplitude}
 C_j
 &=\frac1{2\cdot128^{1/3}}
    \frac{\sum_{m\in\mathcal J_j}\sqrt{a_mb_m}/q_m}
         {\bigl(\sum_{m\in\mathcal J_j}\sqrt{a_mb_m}/q_m^3\bigr)^{1/3}}
 \ge\frac1{2\cdot128^{1/3}}
    \frac{\mathcal S_j/(2\underline q_j)}
         {(\mathcal S_j/\underline q_j^3)^{1/3}}
 =\frac{\mathcal S_j^{2/3}}{4\cdot128^{1/3}}.
\end{align}
Inserting this into Theorem~\ref{thm:refresh-lower} gives the
group-sensitive bound
\[
 \Risk_T^\star(\gamma)
 \ge c_{\rm het}D
       \Bigl(\sum_{j=1}^J\mathcal S_j^{4/3}\Bigr)^{1/2},
 \qquad
 c_{\rm het}=\frac{1/\sqrt3-1/4}{32\cdot128^{1/3}},
\]
which is \eqref{eq:refresh-group-lower}. H\"older's inequality gives
\[
 \sum_j \mathcal S_j
 \le J^{1/4}
      \Bigl(\sum_j\mathcal S_j^{4/3}\Bigr)^{3/4},
\]
and raising to the power $2/3$ and rearranging yields
\[
 \Bigl(\sum_j\mathcal S_j^{4/3}\Bigr)^{1/2}
 \ge J^{-1/6}\Bigl(\sum_j\mathcal S_j\Bigr)^{2/3}
 =J^{-1/6}\Hg ,
\]
which is \eqref{eq:refresh-dyadic-lower}. Dyadic bins starting at
$q_{\min}$ need at most $1+\lceil\log_2(q_{\max}/q_{\min})\rceil$
nonempty bins, and empty bins may be omitted. The argument proves a
comparison with this bin factor; it does not show that the factor is
necessary.\hfill$\qedsymbol$

\subsection{A finite-horizon statistical term}

The fixed-vector class contains stationary objectives even when some
public drift bounds are positive. The next bound retains the statistical
cost of every client on a finite horizon, including clients with
$\gamma_m=0$.

\begin{proposition}[Stationary contribution at a finite horizon]
\label{prop:refresh-stationary}
Fix $D>0$, $T>0$, positive $(w_m,\lambda_m,\sigma_m)_m$, nonnegative
$(\gamma_m)_m$, and a $\lambda$-regular schedule on $[0,\infty)$. For
the fixed-vector class and the pooled causal rules of
Section~\ref{sec:refresh-lower},
\[
 \Risk_T^\star(\gamma)
 \ge\frac D4\min\Bigl\{G,\
       \frac1{\sqrt8}
       \Bigl(\sum_m\frac{w_m^2\sigma_m^2}{\lambda_mT+2}\Bigr)^{1/2}
       \Bigr\}.
\]
For every fixed rule, one of two stationary linear instances attains the
bound.
\end{proposition}

\begin{proof}
Choose amplitudes $A_m\ge0$ and one uniform sign $\bsign\in\{-1,+1\}$. At
every time, client $m$ supplies an independent sample from
$\mathcal N(\bsign A_m,\sigma_m^2)$ with loss $\ell(x;\xi)=\xi x_1$. Let
$C=\sum_mw_mA_m$. The objectives have zero drift and population gradient
bound $C$, so both signs give admissible instances whenever $C\le G$.

For every $T>0$, regularity implies $N_m(T)\le2\lambda_mT+4$: when
$T\ge2/\lambda_m$ the upper inequality applies directly, and otherwise
$N_m(T)\le N_m(2/\lambda_m)\le4$. The two full transcripts through time
$T$ therefore have KL divergence at most
\[
 \sum_m\frac{2A_m^2N_m(T)}{\sigma_m^2}
 \le4\sum_m\frac{(\lambda_mT+2)A_m^2}{\sigma_m^2}.
\]
If the last sum, without the factor $4$, is at most $1/8$, the full
transcripts have total variation at most $1/2$, and so do the partial
transcripts at every earlier time. At every time, the two objective gaps
of any common candidate action add up to $DC$, so the two-point Bayes
argument gives expected gap at least $DC/4$ at every time and hence for
the average over $[0,T]$.

By the Cauchy--Schwarz inequality,
\[
 \sup\Bigl\{\sum_mw_mA_m:\ \sum_m\frac{(\lambda_mT+2)A_m^2}{\sigma_m^2}\le\frac18\Bigr\}
 =\frac1{\sqrt8}\Bigl(\sum_m\frac{w_m^2\sigma_m^2}{\lambda_mT+2}\Bigr)^{1/2},
\]
with equality for nonnegative $A_m$ proportional to
$w_m\sigma_m^2/(\lambda_mT+2)$. If
this value exceeds $G$, scale the amplitudes down until $C=G$; the
information constraint remains valid. Maximizing the Bayes bound over
these amplitudes and selecting one of the two signs proves the claim.
When $G=0$ the bound is zero and there is nothing to prove.
\end{proof}

The drifting and stationary constructions give separate lower bounds on
the same class, so their maximum is a valid lower bound. We do not add
their values, which would require a joint construction.

\section{The joint comparison regime and the remaining scale gap}
\label{app:refresh-regime}

\subsection{Proof of Corollary~\ref{cor:refresh-joint}}

The choice of $K$ gives $4LD/[K(K+1)]\le4LD/K^2\le\Hg$. Since
$\Delta=K\tau$, the sum of the count-floor and timing charges in
\eqref{eq:refresh-optimized} equals $DK\Cg\le D\Hg$. The statistical
threshold in \eqref{eq:refresh-joint-horizon} gives
$2\beta_0\Xi/\sqrt T\le\Hg/2$, and its second line bounds the remaining
transient terms, divided by $D$, by $\Hg/2$; thus
$\transient_T\le D\Hg$. Corollary~\ref{cor:refresh-opt} now proves the
upper guarantee uniformly over the fixed-vector class. The algorithm
belongs to the pooled causal class, so its worst-case risk bounds
$\Risk_T^\star(\gamma)$ from above. The count and gradient hypotheses,
together with $T\ge\max_jT_j$, give the lower inequality by
Corollary~\ref{cor:refresh-dyadic}.\hfill$\qedsymbol$

For orientation, minimizing the continuous relaxation
$4LD/K^2+\Cg K$ over $K$ gives $K=(8LD/\Cg)^{1/3}$ and the value
$3(LD\,\Cg^2)^{1/3}$. The sufficient integer condition of the corollary
therefore expresses, up to constants, the two requirements
$\Cg\lesssim\Hg$ and $LD\,\Cg^2\lesssim\Hg^3$. The lower-bound count
conditions yield only $\sum_m a_m/\lambda_m\le\sum_jC_j$ and impose no
restriction on $L$ or $\tau$.

\subsection{A feasible family with positive communication cost}

For any integer $J\ge1$, take $M=J$ clients and set
\begin{align}
\label{eq:refresh-regime-family}
 q_m&=4^{m-1},& w_m&=\frac1J,&
 \lambda_m&=\frac{100J^{1/3}}{q_m},&
 \gamma_m&=\frac J{q_m},& \sigma_m^2&=100J^{7/3}.
\end{align}
Then $a_m=1/q_m$, $b_m=q_m$, and $\Hg=J^{2/3}$. Use $D=1$, $L=\Hg$,
$G=J$, the regular arrivals $N_m(t)=\lfloor\lambda_mt\rfloor$, and
\[
 K=2,\qquad \tau=\frac{\Hg}{200\bar\gamma}>0,
 \qquad\Delta=K\tau,\qquad
 T=10^4q_{\max}J^2 .
\]
Each nonempty dyadic bin contains exactly one client, and the lower
construction has $T_m=128^{-1/3}q_m$ and $C_m=128^{-1/3}/2$. Hence
$\lambda_mT_m=100\cdot128^{-1/3}J^{1/3}>2$,
$\sum_mC_m=128^{-1/3}J/2<G$, and $T\ge\max_mT_m$. The upper count floors
are inactive, because $\lambda_mh_m^\star=100(2/3)^{1/3}>2K$. Moreover
$K\Cg/\Hg=0.10$, and the three nontransient remainder terms in
\eqref{eq:refresh-optimized}, divided by $D\Hg$, are
\[
 \frac{4LD}{K(K+1)\Hg}=\frac23,\qquad
 \frac{4K\sum_m a_m/\lambda_m}{\Hg}=0.08,\qquad
 \frac{\bar\gamma(K\tau+\Delta)}{\Hg}=0.02 .
\]
Here $\Xi^2=\sum_mq_m\le4q_{\max}/3$, $s_0=q_{\max}/(25J^{1/3})$, and
$1\le\bar\gamma<4/3$. The statistical horizon threshold is at most
$64\beta_0^2T/(3\cdot10^4)<0.023\,T$, and each other threshold is below
$0.001\,T$, so all four horizon requirements hold. We also obtain
\[
 \frac{\transient_T}{D\Hg}
 \le\frac{\beta_0}{25\sqrt3}J^{-5/3}
   +\frac{\beta_0}{10^5\sqrt3}J^{-11/6}
   +\frac1{250000J^2}+\frac1{10^6J}
 <0.076 .
\]
Consequently all upper remainders together are less than $0.843\,D\Hg$.
The parameters vary with $J$: the example certifies that the joint
regime is nonempty, not that it models a fixed empirical population.
Note also that on this segment the fixed midpoint already gives
$F^t(0)-\min_{\cX_D}F^t\le GD/2=J/2$, since every feasible point lies
within $D/2$ of zero and $\|\nabla F^t\|\le G$; this no-data guarantee
is smaller than the complete upper expression at the displayed values of
$J$. The family therefore establishes feasibility, not numerically
informative constants at those scales.

\subsection{A limitation of the disjoint-group construction}

\begin{proposition}[Retuning does not close the scale gap]
\label{prop:refresh-construction-barrier}
Suppose $M=J$, $a_m=1/q_m$, $b_m=q_m$, and $q_m=4^{m-1}$. For any
disjoint nonempty groups, block lengths, and amplitudes satisfying the
drift and information constraints of Theorem~\ref{thm:refresh-lower},
\[
 \Bigl(\sum_{j}C_j^2\Bigr)^{1/2}
 \le c_{\rm grp}\sqrt J,\qquad
 c_{\rm grp}=\frac1{2(1-1/4)}
       +\frac1{\sqrt{512}\,(1-1/2)}<0.756 .
\]
Since $\Hg=J^{2/3}$ in this family, the lower expression of the theorem
cannot be of order $D\Hg$ uniformly in $J$, whatever the partition and
amplitudes.
\end{proposition}

\begin{proof}
Fix one group $j$ with block length $T_j$, and write $\bar A_m=w_mA_m$
for its clients. The drift condition gives $\bar A_m\le T_j/(2q_m)$. The
information condition $4T_j\sum_m\bar A_m^2/b_m\le1/128$ separately
gives $\bar A_m\le\sqrt{q_m/(512T_j)}$ for every selected client.
Therefore
\[
 C_j\le\sum_{q_m>T_j}\frac{T_j}{2q_m}
          +\sum_{q_m\le T_j}\sqrt{\frac{q_m}{512T_j}}
 \le\frac1{2(1-1/4)}
       +\frac1{\sqrt{512}\,(1-1/2)} .
\]
Both sums are dominated by geometric series, and omitting clients only
reduces them. There are at most $J$ disjoint nonempty groups, which
proves the claim.
\end{proof}

For the singleton groups of \eqref{eq:refresh-regime-family}, the
amplitude norm equals $128^{-1/3}\sqrt J/2$ exactly. The existing
construction thus already achieves its best order on this family, and
the $J^{1/6}$ loss in the final H\"older step of
Appendix~\ref{app:refresh-lower} has the correct order for it. Settling the unrestricted minimax dependence would require
a different lower construction or a sharper upper method;
Proposition~\ref{prop:refresh-construction-barrier} is not a minimax
upper bound.

\subsection{Numerical evaluation of the complete guarantees}

Figure~\ref{fig:refresh-regime} evaluates the displayed bounds without
simulating arrival paths. The left panel uses the feasible family above;
the right panel evaluates \eqref{eq:refresh-average} with the selected
heterogeneous windows of the three fixed-path experiments, which are not
claimed to satisfy Corollary~\ref{cor:refresh-joint}. These calculations distinguish a proved comparison
regime from an exploratory instance on which the bound can be loose.

\begin{figure}[htbp]
 \centering
 \includegraphics[width=\linewidth]{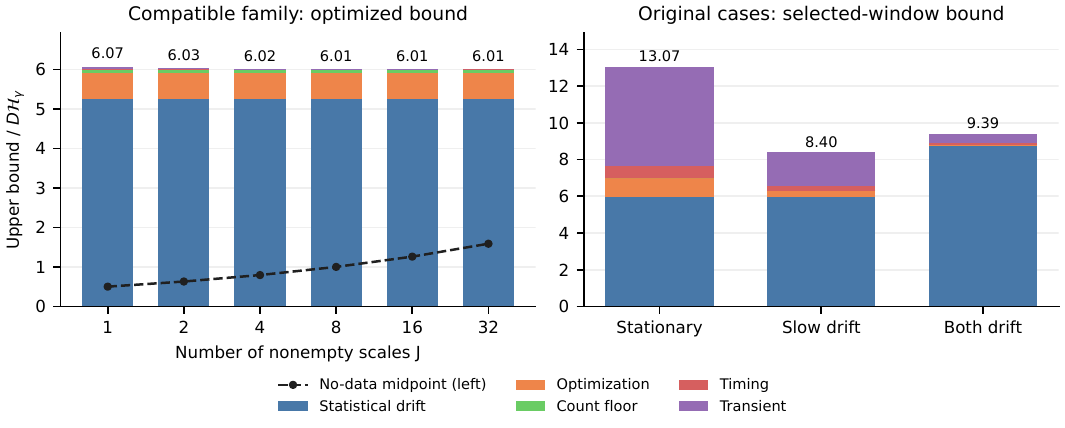}
 \caption{Complete upper guarantees divided by $D\Hg$.
 Left: Corollary~\ref{cor:refresh-opt} on the family
 \eqref{eq:refresh-regime-family}, including its conservative count-floor
 surcharge even though the floors are inactive; the dashed line is the
 no-data midpoint guarantee $GD/2$, which is smaller at all shown scales.
 Right: Theorem~\ref{thm:refresh-upper} at the selected heterogeneous
 windows of the three original synthetic cases, where no separate floor
 surcharge arises; the stationary case has a substantial finite-horizon
 transient. These are evaluations of inequalities, not observed errors,
 and not evidence that the upper and lower numerical constants are
 close.}
 \label{fig:refresh-regime}
\end{figure}

\clearpage
\section{When forgetting helps on periodic streams}
\label{app:periodic-memory}

The drift upper bound protects against every admissible path, whereas a
particular oscillating path may favor the full prefix. This appendix
gives an exact diagnostic for quadratic losses and uniformly weighted
completed buffers. It identifies a noise--frequency threshold for a single
client and proves the separation between shared and client-specific
windows stated in Proposition~\ref{prop:periodic-main}. The frequency
response of a rectangular average is classical
\citep[Chapter~15]{smith1997dsp}; the calculation here includes the
interval over which a completed buffer is held and the noise of the
stationary clients.

\paragraph{Model and estimator class.}
Let $L>0$ and use the public loss
\[
 \ell(x;Y)=\frac L2x^2-Yx,\qquad x\in\cX=[-D/2,D/2].
\]
At times $i/\lambda_m$, $i=1,2,\ldots$, client $m$ observes
\begin{equation}
\label{eq:periodic-stream}
 Y_{m,i}=\bar\vartheta_m+A_m\cos(\omega_m i/\lambda_m+\Phi_m)
                      +\sigma_m\varepsilon_{m,i},
\end{equation}
where $A_m\ge0$, $\omega_m>0$, and the signs $\varepsilon_{m,i}$ are
independent Rademacher variables; stationary clients have $A_m=0$. For
the phase-averaged diagnostic, the phases $\Phi_m$ are independent and
uniform on $[0,2\pi]$. The phase average is a prior over fixed population
paths: conditionally on any fixed phase vector, all observations are
independent and the model of Section~\ref{sec:refresh-model} applies.
It does not assert unconditional independence after mixing over phases.

Assume
\begin{equation}
\label{eq:periodic-interior}
 \frac D2>\frac1L\sum_mw_m(|\bar\vartheta_m|+A_m+\sigma_m),
\end{equation}
so that all aggregate empirical and population minimizers are interior.
The smoothness constant is $L$, the noise bounds are $\sigma_m^2$, the
certified drift bounds are $\gamma_m=A_m\omega_m$, and an aggregate
gradient bound is $G=LD/2+\sum_mw_m(|\bar\vartheta_m|+A_m)$.

A finite memory at client $m$ has length $h_m=n_m/\lambda_m$ with
$n_m\in\mathbb N$. At time $ph_m$ its estimate is replaced by the average
of the observations with indices $(p-1)n_m+1,\ldots,pn_m$, and that
average is held until $(p+1)h_m$; before the first complete buffer, the
arrived prefix is used. A prefix memory always averages all arrived
observations. Set $\tau=0$ and let the server update at the union of the
client update times, reporting the exact weighted empirical minimizer
$x_t=L^{-1}\sum_mw_m\widehat\vartheta_m(t)$, where
$\widehat\vartheta_m(t)$ is client $m$'s current average. A fixed point
$\xinit$ is reported until $t_{\mathrm{s}}=\max_m1/\lambda_m$, when every client has
at least one sample. These rules define an admissible empirical
estimator; they are not the batch-splitting algorithm of
Theorem~\ref{thm:refresh-upper}.

\paragraph{An exact finite-horizon identity.}
Let $\mathcal M_m(t)$ be the set of observation indices used by client
$m$ at time $t\ge t_{\mathrm{s}}$, let $n_m(t)=|\mathcal M_m(t)|$, and let
\[
 \mathcal K_m(t)=\frac1{n_m(t)}\sum_{i\in\mathcal M_m(t)}
             e^{\ii\omega_m i/\lambda_m}
\]
be the complex coefficient of its average. Write $\Qrisk_T$ for the
current-objective risk averaged over both the phases and the observation
noise. For $T\ge t_{\mathrm{s}}$,
\begin{align}
\label{eq:periodic-finite}
 \Qrisk_T
 &=\frac{t_{\mathrm{s}}}{2LT}
     \Bigl[\Bigl(L\xinit-\sum_mw_m\bar\vartheta_m\Bigr)^2
                      +\frac12\sum_mw_m^2A_m^2\Bigr]\nonumber\\
 &\quad+\frac1{2LT}\sum_mw_m^2\int_{t_{\mathrm{s}}}^T
       \Bigl[\frac{A_m^2}{2}
          \bigl|\mathcal K_m(t)-e^{\ii\omega_mt}\bigr|^2
                     +\frac{\sigma_m^2}{n_m(t)}\Bigr]dt .
\end{align}
Finite-horizon initialization and prefix variance can therefore be
retained exactly rather than approximated by a steady-state formula.

\begin{theorem}[Exact periodic risk of completed windows]
\label{thm:periodic-risk}
Under \eqref{eq:periodic-stream}--\eqref{eq:periodic-interior}, the
long-run phase-averaged risk
$\Qrisk_\infty:=\lim_{T\to\infty}\Qrisk_T$ exists. Define the phase
increment per sample $\varphi_m=\omega_m/\lambda_m$,
\[
 \mathcal K_n(\varphi)=\frac1n\sum_{l=0}^{n-1}e^{-\ii l\varphi},\qquad
 \sinc(z)=\begin{cases}\sin(z)/z,&z\ne0,\\1,&z=0,\end{cases}
\]
and, for a finite window of $n$ samples, the bias factor
\begin{equation}
\label{eq:periodic-factor}
 \Upsilon_n(\varphi)=1+|\mathcal K_n(\varphi)|^2
       -2\operatorname{Re}\Bigl\{
          \mathcal K_n(\varphi)\,e^{-\ii n\varphi/2}
                       \sinc(n\varphi/2)\Bigr\}.
\end{equation}
If $\sin(\varphi/2)\ne0$, an equivalent real expression is
\begin{equation}
\label{eq:periodic-factor-real}
 \Upsilon_n(\varphi)=1+\Bigl(\frac{\sin(n\varphi/2)}{n\sin(\varphi/2)}\Bigr)^2
       -2\,\frac{\sin(n\varphi/2)}{n\sin(\varphi/2)}\,
          \sinc(n\varphi/2)\cos\bigl((2n-1)\varphi/2\bigr).
\end{equation}
A client with a finite window of $n_m$ samples contributes
\begin{equation}
\label{eq:periodic-client-risk}
 \frac{w_m^2}{2L}
      \Bigl[\frac{A_m^2}{2}\Upsilon_{n_m}(\varphi_m)
                                      +\frac{\sigma_m^2}{n_m}\Bigr]
\end{equation}
to $\Qrisk_\infty$; a prefix client contributes the same expression
with the bracket replaced by
$\frac12A_m^2(1+\mathbf1\{\varphi_m\in2\pi\mathbb Z\})$. The long-run
risk is the sum of these contributions over clients. In particular, a
stationary prefix contributes zero, and a nonaliased moving prefix
contributes $w_m^2A_m^2/(4L)$.
\end{theorem}

\begin{proof}
The population objective is $F^t(x)=Lx^2/2-\vartheta(t)x$ with
$\vartheta(t)=\sum_mw_m[\bar\vartheta_m+A_m\cos(\omega_mt+\Phi_m)]$, and
its gap at $x$ is $(Lx-\vartheta(t))^2/(2L)$. For $t\ge t_{\mathrm{s}}$, the error
of client $m$'s empirical mean is
\[
 A_m\operatorname{Re}\Bigl\{
       e^{\ii\Phi_m}\bigl[\mathcal K_m(t)-e^{\ii\omega_mt}\bigr]
                        \Bigr\}
 +\frac{\sigma_m}{n_m(t)}\sum_{i\in\mathcal M_m(t)}\varepsilon_{m,i}.
\]
It has zero mean over phase and noise, and independence across clients
removes all cross terms even when the replacement times differ. The
identity $\E_\Phi(\operatorname{Re}(e^{\ii\Phi}z))^2=|z|^2/2$ and the
independence of the noise give \eqref{eq:periodic-finite}, including the
fixed-initialization interval.

For a completed window of length $h=n/\lambda$ and $t=ph+u$ with
$0\le u<h$, the noiseless complex coefficient is
$\mathcal K(t)=e^{\ii\omega ph}\mathcal K_n(\varphi)$, so the phase-averaged
squared bias, divided by $A^2/2$, equals
$|\mathcal K_n(\varphi)-e^{\ii\omega u}|^2$. Averaging over $u\in[0,h)$ and
using
\[
 \frac1h\int_0^h e^{-\ii\omega u}\,du
       =e^{-\ii n\varphi/2}\sinc(n\varphi/2)
\]
gives \eqref{eq:periodic-factor}; the geometric sum for $\mathcal K_n$ gives
\eqref{eq:periodic-factor-real}. The noise variance is exactly
$\sigma^2/n$. The initial finite prefix and the incomplete last holding
interval vanish after division by $T$.

For a prefix with $N(t)=\lfloor\lambda t\rfloor$ samples, the complex
sample average tends to zero if $\varphi\notin2\pi\mathbb Z$, by the
finite geometric-sum formula, and is identically one if
$\varphi\in2\pi\mathbb Z$; in the latter case $\omega>0$ and the time
average of $|1-e^{\ii\omega t}|^2$ equals two. The prefix noise variance
$\sigma^2/N(t)$ has vanishing time average. Boundedness justifies
passing to these limits in \eqref{eq:periodic-finite}.
\end{proof}

\paragraph{A necessary and sufficient improvement test.}
For a nonaliased moving client with $\sigma_m>0$, replacing only its
prefix by a finite completed window of $n$ samples strictly improves the
long-run risk if and only if
\begin{equation}
\label{eq:periodic-benefit}
 \frac{A_m^2n}{2\sigma_m^2}\bigl[1-\Upsilon_n(\varphi_m)\bigr]>1 .
\end{equation}
This exact test includes the hold delay and the discrete arrival rate.
If $\Upsilon_n\ge1$, the window cannot help at any noise level; if
$\Upsilon_n<1$, the display gives the required signal-to-noise ratio.
The weights cancel when a single client's window is changed, so the
other clients may remain on their prefixes.

For interpretation, take the high-rate limit $\varphi\to0$ with
$n\varphi\to\omega h$. Then
\[
 \Upsilon_n(\varphi)\longrightarrow
 1+\sinc^2(\omega h/2)\bigl(1-2\cos(\omega h)\bigr),
\]
and the corresponding high-rate bias--variance expression, with
$\sigma^2/(\lambda h)$ for the sampling noise, favors the finite window
exactly when
\begin{equation}
\label{eq:periodic-continuous-test}
 \frac{A^2\lambda}{\sigma^2\omega}\,
 \frac{\omega h}{2}\,\sinc^2(\omega h/2)\bigl(2\cos(\omega h)-1\bigr)>1 .
\end{equation}
Equation~\eqref{eq:periodic-continuous-test} tests the limiting
expression; \eqref{eq:periodic-benefit} remains the exact finite-rate
criterion. A long window attenuates noise but may average over several
changes of direction; a shorter window helps only if its signal gain
exceeds both its remaining hold bias and its variance cost.

\begin{proposition}[A shared window can be forced to keep the prefix]
\label{prop:periodic-common-prefix}
Suppose one client has a nonaliased sinusoidal mean with amplitude $A$,
frequency $\omega$, and weight $w$, while all other clients are
stationary, and recall $\Xi^2=\sum_mw_m^2\sigma_m^2/\lambda_m$.
Restrict the estimator to uniformly weighted completed-window empirical
minimization with one fixed shared length $h$, or to the full prefix at
every client. For a real length $h$, the completed interval
$((p-1)h,ph]$ contains the actual observations of each client, and its
sample count may vary with $p$. For every real $h$ with $\lambda_mh\ge1$
at every client, and for every fixed phase,
\begin{equation}
\label{eq:periodic-common-all-real}
 \liminf_{T\to\infty}\avgerr_T^{\rm shared}
 \ge\frac{w^2A^2}{4L}
             +\frac{\Xi^2-2w^2A^2/\omega}{2Lh}.
\end{equation}
Hence $\Xi^2\ge2w^2A^2/\omega$ makes every such shared window no better
than the prefix, and strict inequality makes every finite shared window
worse. For the phase-averaged risk at commensurate lengths
$\lambda_mh\in\mathbb N$, the sharper sufficient condition
$\Xi^2\ge w^2A^2/\omega$ holds.
\end{proposition}
\begin{proof}
First let $h$ be any real length satisfying the count condition, so that
every completed interval contains at least one observation per client.
The empirical minimizer is constant on each holding interval, even if
the server reports more often. Whatever its deterministic conditional
mean on that interval, its integrated squared bias is at least the
within-interval variance of $wA\cos(\omega t+\Phi)$. The long-run
average of $\cos^2(\omega t+\Phi)$ is $1/2$, whereas the square of the
mean cosine over any interval of length $h$ is at most
$\sinc^2(\omega h/2)$. Therefore the long-run squared bias has lower
limit at least $w^2A^2[1/2-\sinc^2(\omega h/2)]$.

If $n_{m,p}$ is the count in client $m$'s $p$th completed interval, the
independent noise contributes variance $\sum_mw_m^2\sigma_m^2/n_{m,p}$.
The average count satisfies
$P^{-1}\sum_{p=1}^Pn_{m,p}=\lfloor\lambda_mPh\rfloor/P\to\lambda_mh$,
so Jensen's inequality gives
$\liminf_P P^{-1}\sum_{p=1}^P1/n_{m,p}\ge1/(\lambda_mh)$. Finally,
\[
 h\sinc^2(\omega h/2)
 \le\min\{h,\,4/(\omega^2h)\}\le2/\omega .
\]
Adding bias and variance and dividing by $2L$ proves
\eqref{eq:periodic-common-all-real}; bounded startup and incomplete last
intervals do not affect the lower limit. Nonaliasing ensures that the
full prefix has risk $w^2A^2/(4L)$ for every fixed phase.

For the sharper phase-averaged statement, use a commensurate length and
write $v=e^{-\ii\omega h/2}\sinc(\omega h/2)$. Completing the square in
\eqref{eq:periodic-factor} yields
\[
 1-\Upsilon_n(\varphi)
 =2\operatorname{Re}(\mathcal K_nv)-|\mathcal K_n|^2
 \le |v|^2=\sinc^2(\omega h/2).
\]
The shared-window risk minus the prefix risk is therefore at least
\[
 \frac1{2L}\Bigl[
          \frac{\Xi^2}h-\frac{w^2A^2}{2}\bigl(1-\Upsilon_n(\varphi)\bigr)\Bigr]
 \ge\frac{\Xi^2-w^2A^2/\omega}{2Lh},
\]
and the assertion follows.
\end{proof}

\begin{corollary}[An unbounded ratio within the shared-window class]
\label{cor:periodic-separation}
Take two clients with $w_1=w_2=1/2$, $L=1$, $\cX=[-2,2]$ (so $D=4$),
$\bar\vartheta_1=\bar\vartheta_2=0$, $A_1=\omega_1=1$, and $A_2=0$. For
an integer $r\ge2$, set
\[
 \lambda_1=r,\quad\sigma_1=1/r,\qquad
 \lambda_2=1,\quad\sigma_2=\sqrt2 .
\]
Among all fixed real shared lengths $h\ge1$ in
Proposition~\ref{prop:periodic-common-prefix}, including the full prefix,
the smallest lower limit of the average risk is $1/16$, for every fixed
phase. The client-specific choice $h_1=1/r$ with a prefix at client two
has risk
\begin{equation}
\label{eq:periodic-separation}
 \Qrisk_\infty^{\rm sep}
 =\frac18\Bigl[1-\sinc(1/r)+\frac1{r^2}\Bigr]
 \le\frac7{48r^2}.
\end{equation}
Thus the ratio between the best shared-window risk and this
client-specific risk is at least $3r^2/7$. The client-specific limit
also holds for every fixed phase $\Phi_1$.
\end{corollary}
\begin{proof}
Condition \eqref{eq:periodic-interior} holds because
$\frac12(1+1/r+\sqrt2)<2$. In the model of
Section~\ref{sec:refresh-model} this instance has $D=4$, smoothness
constant $L=1$, $G=5/2$, and drift vector $(1,0)$. Every real shared length
$h\ge1$ has nonempty completed buffers. Here
$\Xi^2=1/(4r^3)+1/2>1/2=2w_1^2A_1^2/\omega_1$, so
Proposition~\ref{prop:periodic-common-prefix} makes every finite shared
window worse than the prefix, whose risk is $w_1^2A_1^2/4=1/16$.

For the client-specific rule, $n_1=1$ and
$\Upsilon_1(\varphi)=2-2\sinc(\varphi)$, while the stationary prefix
contributes nothing. Substitution into \eqref{eq:periodic-client-risk}
gives \eqref{eq:periodic-separation}; the final inequality uses
$1-\sinc(u)\le u^2/6$, obtained by integrating $1-\cos(v)\le v^2/2$
from zero to $u>0$.

To remove the phase averaging for the client-specific rule, expand the
moving client's squared bias over successive holding intervals. The
phase-dependent remainder is a fixed multiple of $e^{2\ii p\omega h}$
(and its conjugate), whose Ces\`aro mean vanishes whenever
$e^{2\ii\omega h}\ne1$; this holds for $h_1=1/r$, which is not a
nonzero integer multiple of $\pi$. The shared-window lower bound already
holds for every phase, without any resonance exclusion on $h$, and the
stationary noise is centered, so no cross-client bias term remains.
\end{proof}

The separation concerns uniformly weighted completed-window empirical
estimators with one fixed common length. It does not apply to arbitrary
learners that store a common full history, since such learners can
implement the client-specific effective weights themselves, and it is
not a communication lower bound, since $\tau=0$. What it isolates is the
concrete statistical restriction that the common-window comparison
imposes.

\section{Adapting over a finite collection of client memories}
\label{app:adaptive-memory}

The prescribed drift vector determines the oracle windows, but it need
not be an input to a finite model-selection procedure. This appendix
gives such a procedure under an additional bounded-loss assumption and
proves Theorem~\ref{thm:adaptive-summary} in its complete form
(Theorem~\ref{thm:adaptive-memory}). The procedure retains the pooled
variance of each fixed-memory candidate and pays an explicit selection
cost; it aggregates candidates and neither estimates the drift vector
nor tracks the best changing candidate.

\paragraph{Additional assumption and data split.}
Assume that known constants $B_m\ge0$ satisfy
\begin{equation}
\label{eq:adaptive-oscillation}
 |\ell(x;\xi)-\ell(y;\xi)|\le B_m,
 \qquad x,y\in\cX,
\end{equation}
almost surely for a sample from client $m$ at any source time, and put
$B=\sum_mw_mB_m>0$. The assumption bounds loss differences, not absolute
loss offsets. For logistic loss with feature vectors of norm at most
$\bar z_m$ on a domain of diameter $D$ one may take $B_m=D\bar z_m$; if
the sample loss contains a known deterministic regularizer, its
oscillation on $\cX$ is added to $B_m$.

Assign the odd-indexed arrivals of each client to a training stream and
the even-indexed arrivals to a validation stream, breaking ties in the
timestamps by an arbitrary fixed order. Conditionally on the arrival
schedule, the two sample arrays are independent. Validation observations
are never used in an expert's gradient updates.

Fix $K\ge1$, $c=K\tau$, and
\begin{equation}
\label{eq:adaptive-timing}
 s_0=6K/\lambda_{\min},\qquad
 \Delta\ge\max\{c,\,6/\lambda_{\min}\},\qquad
 T=s_0+c+R\Delta,
\end{equation}
with an integer $R\ge1$. Let $\grid=\{h^1,\ldots,h^E\}$, $E\ge2$, be a
deterministic collection of memory vectors with $6K/\lambda_m\le
h_m^e\le T$, chosen without knowledge of $\gamma$; the value $h_m^e=T$
keeps the training prefix throughout the horizon. Each expert $e$ runs
the projected-SGD procedure of Proposition~\ref{prop:refresh-sgd} on its
training buffers, with timestamp-only balanced batches, the same fixed
reset point $\xinit$, and solve starts $s_j=s_0+j\Delta$; its variance
inputs use its actual training-batch sizes. Write $x_j^e$ for its output
at $t_j=s_j+c$. Experts may share training observations with one
another and with their own later solves, but no expert uses another
expert's iterates or the aggregation weights.

\paragraph{Causal validation and aggregation.}
On the report interval $\Rint_j=[t_j,t_{j+1})$ the server reports
$x_j=\sum_{e=1}^E\pi_{j,e}x_j^e$. It maintains two exponential-weight
instances, one for even and one for odd $j$, each initialized with
uniform weights, with learning rate and weights
\[
 \etaew=\frac{2}{B}\sqrt{\frac{\log E}{R}},\qquad
 \pi_{j,e}=
 \frac{\exp\bigl\{-\etaew\sum_{0\le i\le j-2:\ i\equiv j\ ({\rm mod}\ 2)}
                         \vscore_{i,e}\bigr\}}
      {\sum_{v=1}^E\exp\bigl\{-\etaew\sum_{0\le i\le j-2:\ i\equiv j\ ({\rm mod}\ 2)}
                         \vscore_{i,v}\bigr\}},
\]
where empty sums are zero. The score $\vscore_{j,e}$ uses the validation
arrivals of client $m$ in $(t_j,t_{j+1}]$; call this set
$\mathcal V_{m,j}$ and its size $n_{m,j}^{\rm val}$. Then
\begin{equation}
\label{eq:adaptive-validation}
 \vscore_{j,e}=
 \sum_m\frac{w_m}{n_{m,j}^{\rm val}}
 \sum_{(u,\xi)\in\mathcal V_{m,j}}
       \bigl[\ell(x_j^e;\xi)-\ell(\xinit;\xi)\bigr].
\end{equation}
All these counts are positive by Lemma~\ref{lem:adaptive-counts}. The
interval $(t_j,t_{j+1}]$ has ended by the start of solve $j+2$, since
$t_{j+1}=s_j+\Delta+c\le s_j+2\Delta$. The server includes the points
$(x_j^e)_{e=1}^E$ in the first broadcast of that solve, and the clients
return the corresponding validation sums and counts with their first
gradient messages, so $\vscore_{j,e}$ is available before the report at
$t_{j+2}$. No validation sample later than the current report is used.
Because message size and local computation are uncharged, all experts'
gradient messages and the validation statistics can be bundled into the
same $K$ rounds. The report is $\xinit$ before $t_0$.

The construction uses the standard exponential-weight principle and the
standard reduction of delayed feedback to parallel forecasters
\citep{freund1997hedge,joulani2013delayed}. The result below accounts for
the independent validation stream, its nonuniform source times, and the
finite-sample cost of the split.

\begin{lemma}[Counts after alternating assignment]
\label{lem:adaptive-counts}
On a $\lambda$-regular schedule, either assigned stream of client $m$
has at least $\lambda_m(v-u)/6$ observations in $(u,v]$ whenever
$v-u\ge6/\lambda_m$.
\end{lemma}
\begin{proof}
If the full interval contains $n$ observations, either parity contains
at least $(n-1)/2$ of them, and regularity gives
\[
 \frac{n-1}{2}\ge\frac{\lambda_m(v-u)}4-\frac12
                  \ge\frac{\lambda_m(v-u)}6 .
\]
Only this lower count inequality is needed for the upper analysis; we do
not assert that either thinned schedule is regular with rate
$\lambda_m/3$ in the two-sided sense of
Definition~\ref{def:refresh-regular}.
\end{proof}

Write $\Risk_T^e$ for the expected average current-objective gap of
expert $e$'s whole reported trajectory, including initialization, and
$\Risk_T^{\rm mix}$ for the same risk of the aggregate.

\begin{theorem}[Unknown-drift adaptation over joint memories]
\label{thm:adaptive-memory}
Under Section~\ref{sec:refresh-model}, except that $\gamma$ need not be
known to the learner, and under
\eqref{eq:adaptive-oscillation}--\eqref{eq:adaptive-timing}, the causal
procedure above satisfies
\begin{equation}
\label{eq:adaptive-oracle}
 \Risk_T^{\rm mix}
 \le\min_{1\le e\le E}\Risk_T^e
       +2B\sqrt{\frac{\Delta\log E}{T}}
       +D\bar\gamma\Delta .
\end{equation}
In particular, with $\widetilde b_m=3b_m$,
$\widetilde\Xi^2=\sum_m\widetilde b_m$, and
\[
 \widetilde{\transient}_T
 =\beta_{\rm sgd}D\widetilde\Xi
       \Bigl(\frac2{\sqrt T}+\frac{\Delta}{T\sqrt{s_0}}\Bigr)
       +\frac{GD(s_0+c)}T ,
\]
we have
\begin{align}
\label{eq:adaptive-explicit}
 \Risk_T^{\rm mix}
 &\le \frac{LD^2}{K}
  +D\min_{h\in\grid}
     \Bigl\{2\sum_m a_mh_m
          +\beta_{\rm sgd}\sqrt{\sum_m\widetilde b_m/h_m}\Bigr\}
       \nonumber\\
 &\quad+D\bar\gamma(c+2\Delta)
       +\widetilde{\transient}_T
       +2B\sqrt{\frac{\Delta\log E}{T}} .
\end{align}
The algorithm uses the noise and rate bounds but not the values
$a_m=w_m\gamma_m$ that appear in the oracle bound.
\end{theorem}

The selection inequality \eqref{eq:adaptive-oracle} requires only that
the experts' reports lie in $\cX$ and are causal functions of the
training array, independent of the validation observations; it therefore
also applies to fixed-memory empirical-risk minimizers. The explicit
gradient-based bound \eqref{eq:adaptive-explicit} additionally requires
the projected-SGD experts specified above.

\begin{proof}
We first prove the selection inequality, which needs no gradient
concentration event. Condition on the full arrival schedule and on the
entire training array. Every expert point is then fixed, the validation
observations remain mutually independent, and at report $j$ the weight
vector $\pi_j$ depends only on previously used validation intervals and
is independent of the observations in $(t_j,t_{j+1}]$.

Population losses are centered at $\xinit$ before integrating over time;
this removes arbitrary loss offsets, which the assumptions do not bound.
Define the centered average population objective on a reporting interval
and the centered population objective at the validation source times by
\[
 \overline F_j(x)=\frac1\Delta\int_{t_j}^{t_{j+1}}
                         [F^t(x)-F^t(\xinit)]\,dt,
 \qquad
 F_j^{\rm val}(x)=\sum_m\frac{w_m}{n_{m,j}^{\rm val}}
                  \sum_{(u,\xi)\in\mathcal V_{m,j}}
                         [f_m^u(x)-f_m^u(\xinit)] ;
\]
the latter uses only the timestamps $u$. Conditionally on the stated
information and on past validation feedback,
\begin{equation}
\label{eq:adaptive-unbiased}
 \E[\vscore_{j,e}]=F_j^{\rm val}(x_j^e),
 \qquad |\vscore_{j,e}|\le B .
\end{equation}
This conditional expectation evaluates finitely many training-measurable
points on a separate sample array; it asserts no uniform concentration
over queries.

Consider either parity of reporting intervals. The exponential-weight
potential and Hoeffding's lemma for losses in $[-B,B]$ give, for any
expert $e$ and $n$ intervals of that parity,
\[
 \sum_j\Bigl(\sum_v\pi_{j,v}\vscore_{j,v}
                         -\vscore_{j,e}\Bigr)
 \le\frac{\log E}{\etaew}+\frac{\etaew B^2n}{2} :
\]
the one-step upper bound for the log potential ratio is
$-\etaew\sum_v\pi_{j,v}\vscore_{j,v}+\etaew^2B^2/2$, and the terminal weight of
expert $e$ gives the matching lower bound
$-\etaew\sum_j\vscore_{j,e}-\log E$. All earlier losses of a parity are
available before its next report. Adding the inequalities of the two
parities (one of which may be empty) and substituting $\etaew$ yields the
pathwise bound
\begin{equation}
\label{eq:adaptive-hedge}
 \sum_{j=0}^{R-1}\Bigl(\sum_v\pi_{j,v}\vscore_{j,v}
                          -\vscore_{j,e}\Bigr)
 \le 2B\sqrt{R\log E}.
\end{equation}

Source-time drift gives
$\|\nabla F_j^{\rm val}(x)-\nabla\overline F_j(x)\|\le\bar\gamma\Delta$
uniformly on $\cX$, because every source time and every integration time
lie in an interval of length $\Delta$. Consequently, for all $x,y\in\cX$,
\[
 \bigl|[F_j^{\rm val}(x)-F_j^{\rm val}(y)]
       -[\overline F_j(x)-\overline F_j(y)]\bigr|
 \le D\bar\gamma\Delta ,
\]
and convexity of each $F^t$, hence of $\overline F_j$, gives
\[
 \overline F_j(x_j)-\overline F_j(x_j^e)
 \le\sum_v\pi_{j,v}\bigl[F_j^{\rm val}(x_j^v)-F_j^{\rm val}(x_j^e)\bigr]
                      +D\bar\gamma\Delta .
\]
Take expectations, use \eqref{eq:adaptive-unbiased} and
\eqref{eq:adaptive-hedge}, and multiply by $\Delta/T$. The current
minimum cancels in the risk difference, and both procedures report
$\xinit$ before $t_0$. We obtain
\[
 \Risk_T^{\rm mix}-\Risk_T^e
 \le\frac{2B\Delta\sqrt{R\log E}}T
          +\frac{R\Delta}{T}D\bar\gamma\Delta
 \le2B\sqrt{\frac{\Delta\log E}{T}}+D\bar\gamma\Delta .
\]
This holds for every deterministic index $e$, so taking the minimum of
the expected expert risks proves \eqref{eq:adaptive-oracle}.

Finally, Lemma~\ref{lem:adaptive-counts} supplies the only arrival-count
input used in Proposition~\ref{prop:refresh-sgd}, with lower rate
$\lambda_m/3$. Frozen training intervals have length
$\min\{s_j,h_m^e\}\ge6K/\lambda_m$ and contain at least $K$
observations. Each expert retains the original reset and disjoint-batch
filtration, and shared observations between experts impose no
independence requirement on their marginal bounds. The proof of that
proposition therefore applies with $b_m$ replaced by $3b_m$ and with the
initialization time of \eqref{eq:adaptive-timing}. Substituting into
\eqref{eq:adaptive-oracle} proves \eqref{eq:adaptive-explicit}.
\end{proof}

\paragraph{A geometric joint grid.}
Let $\underline h_m=6K/\lambda_m$ and take
\[
 \grid_m=\bigl\{\min(2^i\underline h_m,T):
          i=0,\ldots,\lceil\log_2(T/\underline h_m)\rceil\bigr\},
 \qquad \grid=\prod_m\grid_m .
\]
For every $h_m\in[\underline h_m,T]$, rounding upward to this grid gives
$h_m\le h_m'\le2h_m$, so the minimum in \eqref{eq:adaptive-explicit} is
at most twice the continuous minimum of the same bias--pooled-variance
expression on the box. Moreover $\log E=\sum_m\log|\grid_m|$: the
selection cost depends on this logarithm, whereas maintaining all experts
explicitly can require exponentially many gradient evaluations in $M$.
If the grid has a single point, that expert is used and the selection
penalty is omitted.

For stationary clients the grid contains the full-prefix candidate,
whose statistical contribution vanishes with $T$. The guarantee is
adaptation to a fixed joint memory vector over the stated horizon, not
to a sequence of locally optimal windows, and the implementation must
store the union of the candidate buffers, potentially the full training
history: it gives no storage or computation saving.

\paragraph{Other horizons and scope.}
For an unaligned horizon $T>s_0+c+\Delta$, let
$R=\lfloor(T-s_0-c)/\Delta\rfloor$, use the same learning rate with this
$R$, and keep the last mixture through the remaining fraction of an
interval. The same argument on the $R$ complete intervals gives
\eqref{eq:adaptive-oracle}--\eqref{eq:adaptive-explicit} with one
additional term $GD\Delta/T$, since the unmatched part contributes at
most $GD\Delta$ to the cumulative risk difference; the expert bound is
applied at the actual horizon $T$.

Unknown-drift window selection is available under uniform function-class
convergence assumptions \citep{mazzetto2023adaptive}, and federated drift
adaptation is studied by \citet{ganguly2024nonstationary}. The extension
here is the finite joint-memory oracle inequality with independent
validation, asynchronous source-time correction, and a causal
communication schedule. It does not replace those results, and it does
not establish efficient unrestricted adaptation under the original
pointwise moment assumptions alone.

\clearpage
\providecommand{\refor}[2]{\ifcsname r@#1\endcsname\ref{#1}\else #2\fi}

\section{Clinics with different specialties: when forgetting pays}
\label{app:clinics}

The logistic study of Section~7 marks a boundary: along smooth population paths on which both
the covariates and the label mechanism change, the best common memory is the full history, and
client-specific windows bring no measurable benefit. That result says what the theory does not promise.
This appendix asks the complementary question. What does a network of data sources look like in which
forgetting does pay, and does the memory rule of Corollary~\refor{cor:refresh-opt}{4.2} capture that
benefit when the data are realistic and the solver is Algorithm~\refor{alg:refresh}{1}?

The answer we construct is a regional network of twelve clinics. Its acute services see many patients and
follow epidemic waves; its specialized clinics see few patients from stable populations; and in between,
cardiology drifts slowly as its population ages. The shared risk model is the same for all of them. With
the solver held fixed, client-specific windows lower the time-averaged tracking error by 33\% relative to
the best single memory, and they attain almost nine tenths of what an oracle obtains by tuning every
clinic's window separately.

\subsection{The network}
\label{app:clinics-network}

Table~\ref{tab:clinics} lists the clinics. There are six specialties, each with an urban and a rural site.
All clinics estimate one quantity: the probability that a patient seen at a visit has an emergency or
inpatient encounter within the following 90 days. They differ in three respects that matter for memory.

The \emph{acute services}, pediatrics, the emergency department and primary care, see between 25 and 120
patients a week. Their case mix changes with the season and, for the emergency department, with the waves
of COVID-19. The drift is fast, but so is the flow of new data, so a short memory is affordable. The
\emph{specialized clinics}, endocrinology and oncology, see between one and a half and eight patients a
week, and the populations they serve do not change over the three years we follow. Every discarded record
costs them precision and buys nothing. \emph{Cardiology} is the intermediate case: 15 to 40 patients a week
and a population whose risk profile shifts slowly as it ages, with winter peaks on top. A single memory
length has to choose between these situations; a client-specific memory does not.

\begin{table}[!ht]
\centering
\caption{The twelve clinics. The surge cohort is defined by a variable the model does not observe. Urban
and rural sites follow surveillance curves from different regions. All clinics have weight $w_m=1/12$.}
\label{tab:clinics}
\small
\setlength{\tabcolsep}{4pt}
\renewcommand{\arraystretch}{1.15}
\begin{tabular}{@{}>{\raggedright\arraybackslash}p{0.13\linewidth}
                   >{\raggedright\arraybackslash}p{0.25\linewidth}
                   >{\raggedright\arraybackslash}p{0.17\linewidth}
                   >{\raggedright\arraybackslash}p{0.27\linewidth}
                   >{\raggedright\arraybackslash}p{0.11\linewidth}@{}}
\toprule
Specialty & Patients seen & Surge cohort & Shape of the drift & Visits per week (urban / rural) \\
\midrule
Pediatrics    & age below 18 & recent acute care & ILINet, youngest age group & 60 / 25 \\
Emergency     & emergency encounters & recent acute care & COVID-19 hospital admissions & 120 / 40 \\
Primary care  & adults at wellness, ambulatory and outpatient visits & recent acute care & weighted ILINet activity & 80 / 30 \\
Cardiology    & age 40 or older with hypertension, heart failure or coronary disease & recent acute care & slow ageing trend with winter peaks & 40 / 15 \\
Endocrinology & diabetes or prediabetes & none & stationary & 8 / 3 \\
Oncology      & active cancer diagnosis & none & stationary & 3 / 1.5 \\
\bottomrule
\end{tabular}
\end{table}

\subsection{Patients, outcome, and what the model does not see}
\label{app:clinics-data}

The patients are synthetic electronic health records generated with Synthea
\citep{walonoski2018synthea}. Every ambulatory, wellness, outpatient, urgent-care or emergency visit in the
last eight years of the simulated period, with at least 90 days of follow-up, is a candidate index visit; at most four
are kept per patient. The outcome of a visit is an emergency or inpatient encounter within the next 90
days. The features are those available when the patient is seen: age, sex, body-mass index, the latest
systolic and diastolic blood pressure within three years, indicators of hypertension, diabetes or
prediabetes, asthma, chronic obstructive pulmonary disease, heart failure, coronary disease, cancer and
chronic kidney disease active at the visit, and indicators of missing measurements. A fifth of the
patients is set aside once and used only to fit a public starting model; none of them appears in any
clinic's data.

The design turns on one choice. Drift can help a forgetting rule only if it moves the \emph{optimum}. If
the patients who become more frequent during a wave were distinguished by the features, the change would
be a covariate shift, and a logistic model that already accounts for those features would barely need to
move. We therefore define the surge cohort by a variable the model never receives: an emergency or
inpatient encounter in the 90 days \emph{before} the visit. In the generated population, patients with
such a recent episode are several times as likely to return to acute care in cardiology and in the
emergency department, less so in primary care, and only mildly so among children. When a respiratory wave brings
more of them to a clinic, the same observed features carry a higher risk, and the optimal parameter follows
the wave. This is concept drift as the model experiences it, and its speed differs from clinic to clinic.

For every specialty the eligible visits of the remaining patients form two finite cohort banks, a
baseline bank and a surge bank, each holding up to 3000 records; the stationary specialties have a single
bank. A clinic's population at time $t$ is the mixture
\begin{equation}
\label{eq:clinics-mixture}
(1-p_m(t))\,\mathrm{baseline}_m+p_m(t)\,\mathrm{surge}_m,
\end{equation}
and each arrival draws one record from the corresponding bank. Because the banks are finite, the current
objective, its minimizer, and the drift and noise of every clinic can be computed exactly at any time.

\subsection{How the patient mix changes}
\label{app:clinics-drift}

The shares $p_m(t)$ follow public surveillance data retrieved from the Delphi Epidata service
\citep{reinhart2021covidcast}
(Figure~\ref{fig:clinics-scenario}a). Pediatrics follows the share of influenza-like-illness visits among
the youngest children in the {CDC} {ILINet} network \citep{cdc-ilinet}, nationally for the urban site and for HHS Region~6 for the
rural one. Primary care follows weighted ILINet activity in HHS Regions~1 and~9. The emergency department
follows weekly COVID-19 hospital admissions reported to {HHS} \citep{hhs-covid-hosp} in New York and in Montana, whose waves arrived
at different times: the rural site peaks in autumn 2020 and during the Delta wave, the urban one in winter
2020--21 and during the Omicron wave. Cardiology combines an S-shaped ageing trend with winter peaks whose
timing is taken from the regional ILINet curve, as cardiovascular admissions rise in winter. Influenza and
cardiology curves start in October 2016 and COVID-19 curves in August 2020; each is smoothed over three
weeks and rescaled to the range of shares given for its specialty. Time is measured in weeks, the horizon
is $T=156$, and patients arrive at the regular times $n/\lambda_m$ of Definition~\refor{def:refresh-regular}{3.1}.

\begin{figure}[t]
\centering
\includegraphics[width=\linewidth]{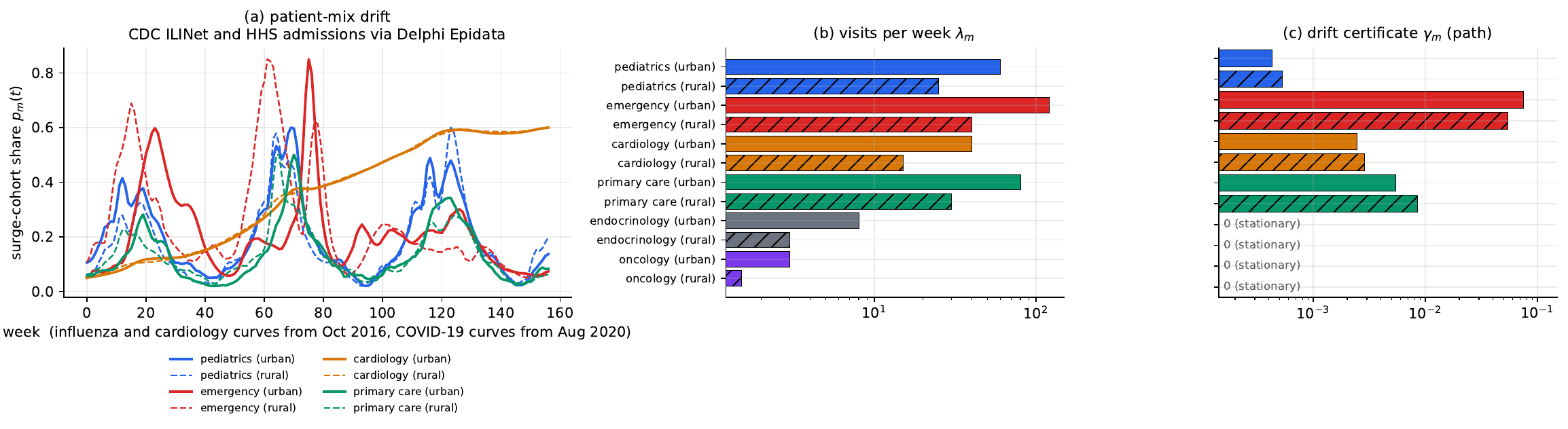}
\caption{The network. (a) Share of the surge cohort in each drifting clinic; solid lines are urban sites,
dashed lines rural ones. (b) Visits per week. (c) Drift bound $\gamma_m$ along the optimal trajectory; the
four specialized clinics are stationary. Drift speeds span two orders of magnitude and are not ordered
like the visit rates, which is what makes a single memory length a compromise.}
\label{fig:clinics-scenario}
\end{figure}

\subsection{Model, objective, and exact quantities}
\label{app:clinics-model}

The model is a logistic regression with coefficients shared by all clinics and one intercept per clinic,
the centre effect of multicentre studies. With $z$ the observed features of a record and $e_m$ the indicator
of clinic $m$, the parameter $\theta$ enters through $\theta^\top(z,\,0.8\,e_m)$, and the features are scaled
so that the whole input has norm at most one. The tracked objective is
\begin{equation}
\label{eq:clinics-objective}
F^t(\theta)=\sum_{m=1}^{12}w_m\bigl[(1-p_m(t))\,L_m^{\mathrm{base}}(\theta)+p_m(t)\,L_m^{\mathrm{surge}}(\theta)\bigr]
+\frac{\rho}{2}\|\theta\|^2,\qquad \rho=0.005,
\end{equation}
where $L_m^{\mathrm{base}}$ and $L_m^{\mathrm{surge}}$ are the mean logistic losses of the two banks of clinic
$m$. Every minimizer lies in the ball of radius $1/\rho=200$, which serves as the domain. The clinic
intercept localizes a clinic's drift: when the surge share of the emergency department rises, the
optimum moves mainly in that clinic's intercept, while the shared coefficients adjust only as far as the
other clinics allow.

Since~\eqref{eq:clinics-objective} is a finite sum, the minimizer $\theta^\star_t$ and the value
$\min F^t$ are computed exactly on a grid of four points per week, and the gap of any reported parameter
is evaluated without Monte Carlo error. The memory rule needs the quantities of Theorem~\refor{thm:refresh-upper}{4.1}:
the drift bound $\gamma_m$ and the gradient noise $\sigma_m^2$ of every clinic. Bounds that hold uniformly over
the parameter ball are valid but loose by one to two orders of magnitude, and unevenly so across clinics.
We therefore evaluate both exactly along the optimal trajectory,
\begin{equation}
\label{eq:clinics-certificates}
\gamma_m=\max_t\,|p_m'(t)|\,\bigl\|\nabla L_m^{\mathrm{surge}}(\theta^\star_t)-\nabla L_m^{\mathrm{base}}(\theta^\star_t)\bigr\|,
\qquad
\sigma_m^2=\max_t\,\mathrm{Var}_m^t\bigl[\nabla\ell(\theta^\star_t)\bigr],
\end{equation}
where the variance is that of a single-record gradient under the clinic's mixture at time $t$. These are
the drift and noise the algorithm actually meets near the optimum, and they are the inputs of
$a_m=w_m\gamma_m$, $b_m=w_m^2\sigma_m^2/\lambda_m$ and of the windows of Corollary~\refor{cor:refresh-opt}{4.2}.
Figure~\ref{fig:clinics-scenario}c shows the resulting drift bounds: the emergency department drifts
fastest, followed by primary care and cardiology, while the hidden cohort moves children's risk only
slightly.

\subsection{Memories, solvers, and evaluation}
\label{app:clinics-protocol}

Five memories are compared. \emph{One common window} gives every clinic the same memory length, the
comparator of the theory. \emph{Common window with a data floor} uses one length but never lets a clinic's
window hold fewer than 100 of its records, which protects the logistic fit of small clinics from buffers
with no events; it is a stronger practical baseline. \emph{Full history} keeps everything.
\emph{Client-specific windows} use the shape $h_m\propto\sqrt{b_m/a_m}$ of Corollary~\refor{cor:refresh-opt}{4.2},
with the full history for stationary clinics and the same data floor. The \emph{per-clinic oracle} starts
from the client-specific windows and adjusts each clinic's window separately by a coordinate search; it
is an upper benchmark, not a method.

Each memory is combined with two solvers. Exact minimization of the buffered empirical risk isolates the
quality of the memory. Algorithm~\refor{alg:refresh}{1} is the method of the paper: $K=8$ accelerated
rounds on disjoint batches, restarted at every solve from a public point, with solves every week and a
round duration of $1/16$ week. Its public point is the model fitted once on the held-out pilot patients.
Every memory has one free scale, selected on six tuning streams; the evaluation uses twenty further
independent streams, on which all memories receive the same patients at the same times. Because the
objective is exact, the scale is selected by the true tracking error; this compares memories at their
best scale and is not a data-driven procedure for choosing windows. The score is the time-averaged gap
$F^t(x_t)-\min F^t$ after the first eleven weeks, during which every memory reports the same public model.

\subsection{Results}
\label{app:clinics-results}

Figure~\ref{fig:clinics-gain} gives the headline. With exact minimization, client-specific windows reduce
the tracking error from $3.36\cdot10^{-3}$ to $2.26\cdot10^{-3}$, by 33\% relative to one common window;
the paired difference over the twenty evaluation streams is $-1.10\cdot10^{-3}$ with a 95\% half-width of
$8.7\cdot10^{-5}$. The best common window is the full history, so the same 33\% is the gain over keeping
everything. Against the common window with a data floor the reduction is 26\%. The per-clinic oracle reaches
$2.13\cdot10^{-3}$; the rule of Corollary~\refor{cor:refresh-opt}{4.2} thus obtains 89\% of the
improvement that separate tuning of all twelve windows achieves.

\begin{figure}[t]
\centering
\includegraphics[width=\linewidth]{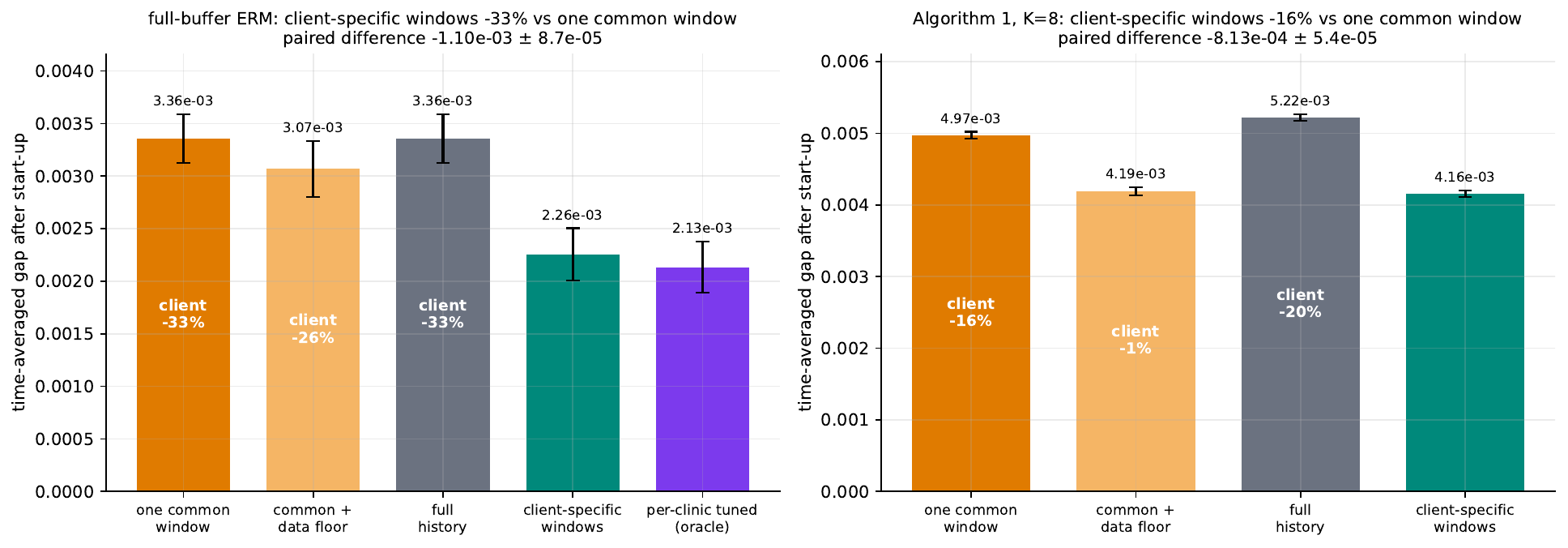}
\caption{Time-averaged gap after start-up, mean over twenty evaluation streams with 95\% intervals. Left:
exact minimization of the buffered risk. Right: Algorithm~\refor{alg:refresh}{1} with $K=8$. The label in
each baseline bar is the change obtained by client-specific windows relative to that baseline.}
\label{fig:clinics-gain}
\end{figure}

Figure~\ref{fig:clinics-windows} shows why. Tuned as one number, the common memory settles on the whole
history: most of the network is stable or slow, and any shorter single window would add noise to all of
it. Client-specific windows keep the whole history for the four specialized clinics and forget quickly
where the optimum moves fast and data are plentiful: about one week for the urban emergency department,
two for the rural one, three to four weeks for primary care, and between seven and fourteen weeks for
pediatrics and cardiology. The oracle agrees on the emergency departments and on urban cardiology. It
parts from the prescription on pediatrics and primary care, where it returns to the full history, and on
rural cardiology, where it chooses a window twice as long; we return to this below.

\begin{figure}[t]
\centering
\includegraphics[width=\linewidth]{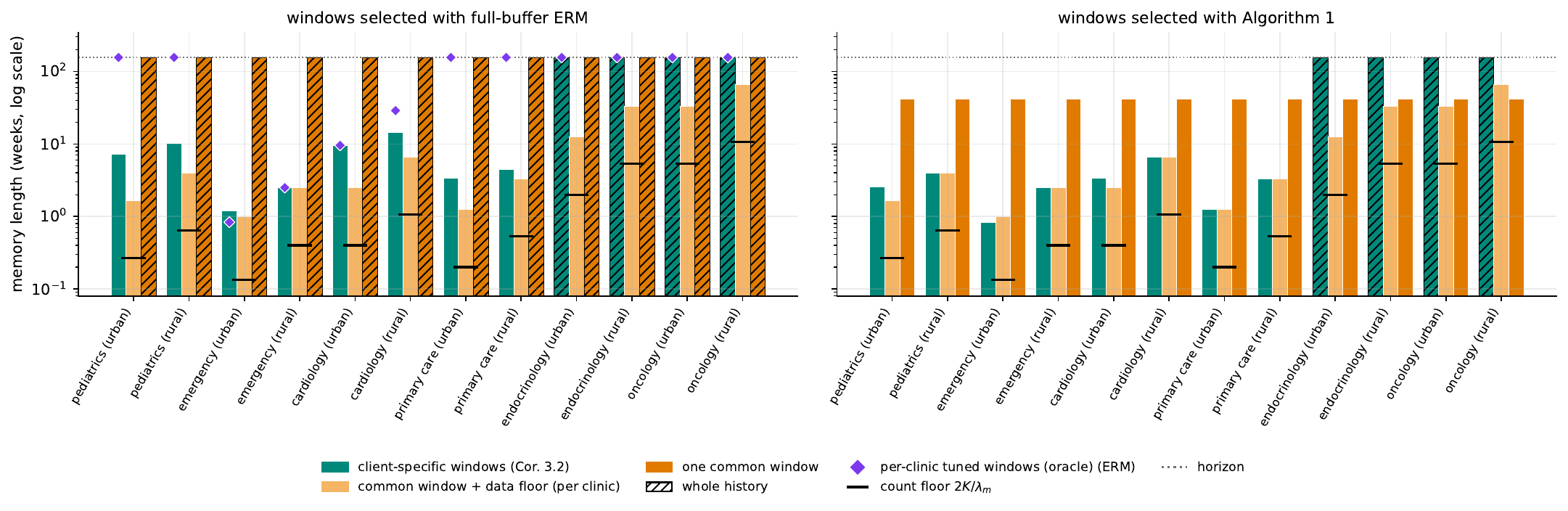}
\caption{Memory lengths selected for each clinic. Hatched bars mean that the whole history is kept.
Diamonds mark the per-clinic oracle, dashes the minimum length $2K/\lambda_m$ needed to form $K$ batches.}
\label{fig:clinics-windows}
\end{figure}

The time course in Figure~\ref{fig:clinics-time} shows where the error is made. The full history is
worst during the waves that reach the emergency departments, around weeks 20 and 60--80, when its error
rises to about $10^{-2}$, and again in the last year, when cardiology keeps ageing and every old record
drifts further from the present. Client-specific windows stay close to the oracle throughout and remain
near $1.3\cdot10^{-3}$ from week 90 on. The full history is better only for short spells between waves,
when the current mix happens to pass through its long-run average. The common window with a data floor
tracks the waves, but pays a constant excess in the calm periods, because its short common length also
shortens the memory of clinics that did not need it.

\begin{figure}[t]
\centering
\includegraphics[width=\linewidth]{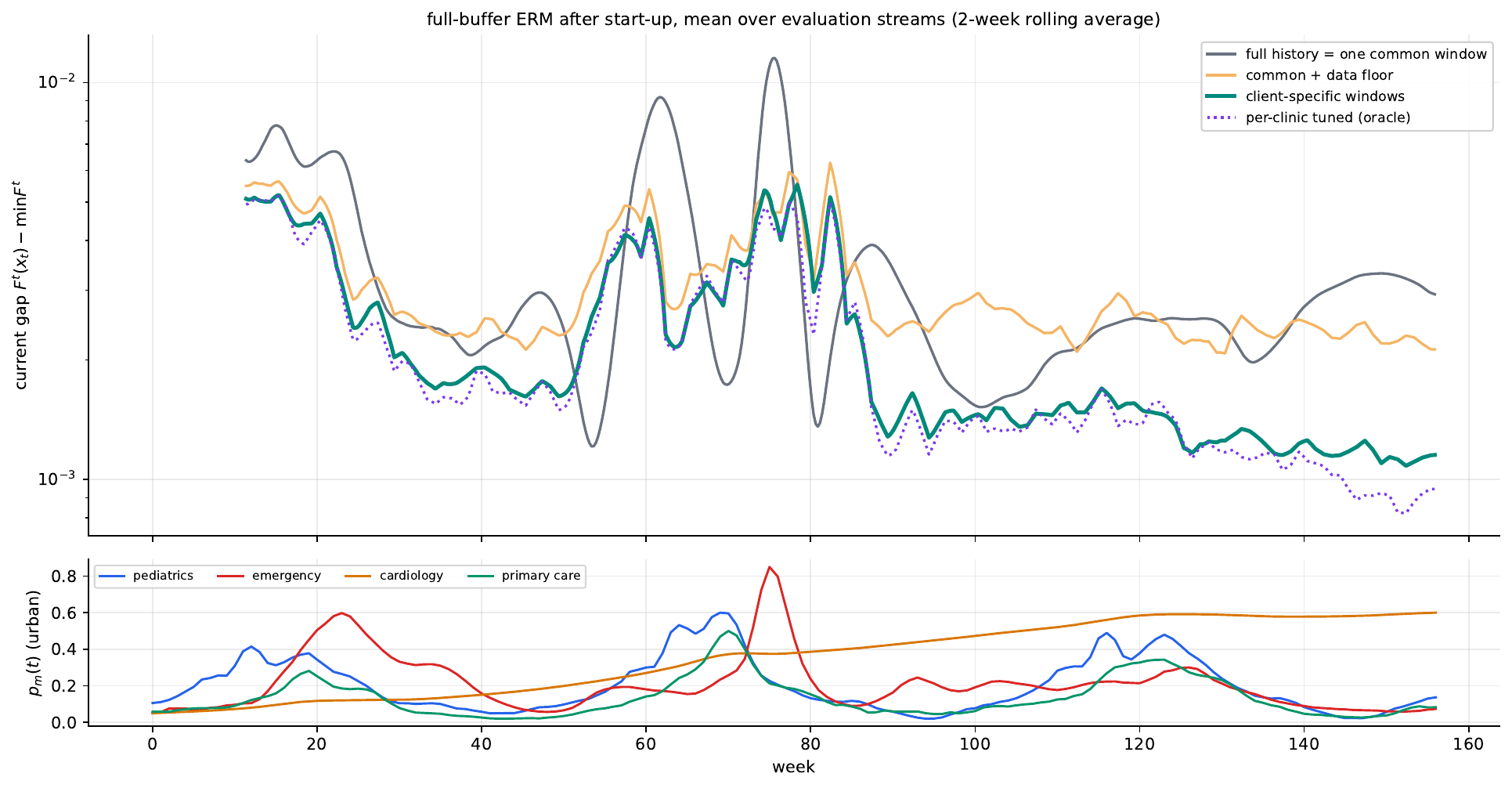}
\caption{Current gap over time under exact minimization, after start-up; two-week rolling mean over the
evaluation streams. The lower panel repeats the surge shares of the urban clinics. The full history and
the one common window coincide because the common window selects the whole history.}
\label{fig:clinics-time}
\end{figure}

Figure~\ref{fig:clinics-perclinic} splits the gap into the contributions
$w_m[f_m^t(x_t)-f_m^t(\theta^\star_t)]$ of the individual clinics. Almost all of the gain comes from the two
emergency departments, whose contributions fall from about $10^{-3}$ and $1.6\cdot10^{-3}$ to about
$3.6\cdot10^{-4}$ and $6\cdot10^{-4}$, a reduction near two thirds. Several contributions are negative. The joint
optimum $\theta^\star_t$ minimizes the sum over clinics, not each clinic's loss, and through the shared
coefficients it bends towards the emergency departments during their waves; a memory that ignores those
waves therefore fits pediatrics, cardiology and primary care slightly better than the optimum does, at a
larger cost elsewhere. The data-floor baseline shows the other side of the trade-off: it helps the
emergency departments as much as the client-specific rule, but its short windows make the other drifting
clinics positive.

\begin{figure}[t]
\centering
\includegraphics[width=\linewidth]{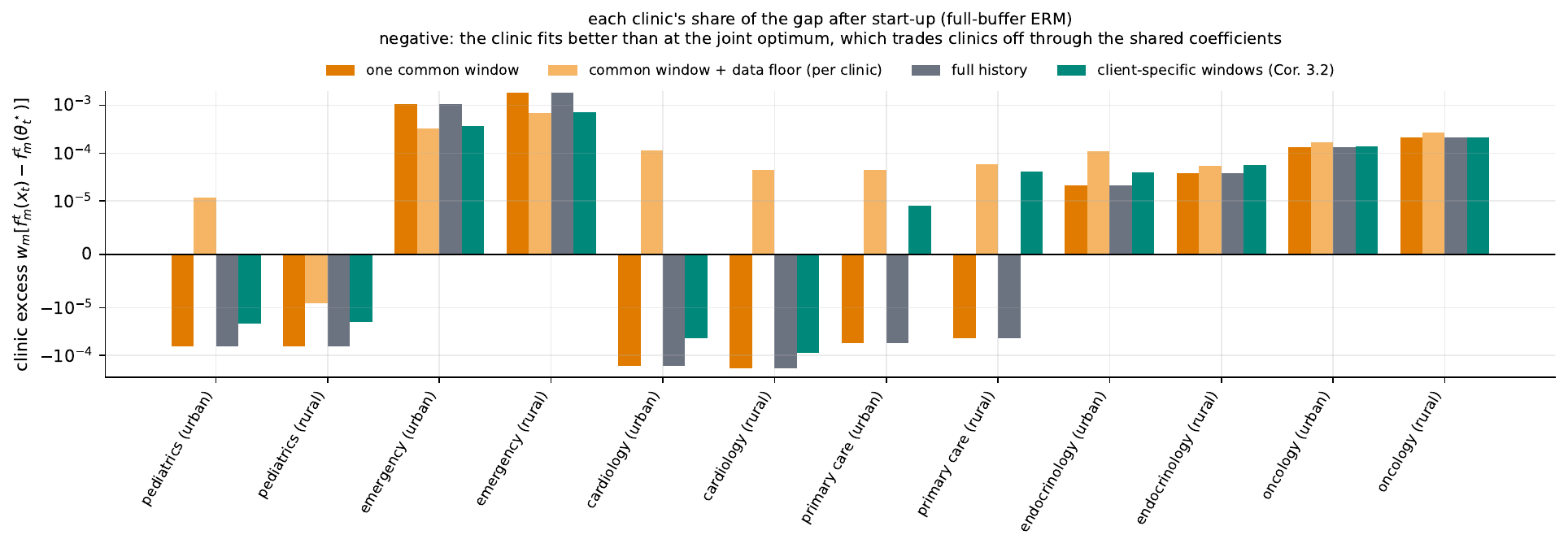}
\caption{Contribution of each clinic to the gap after start-up under exact minimization, on a symmetric
logarithmic scale. Negative values mean that the clinic fits better than at the joint optimum.}
\label{fig:clinics-perclinic}
\end{figure}

With Algorithm~\refor{alg:refresh}{1} every error is larger, between $4.2\cdot10^{-3}$ and
$5.2\cdot10^{-3}$, because eight accelerated rounds do not reach the minimizer of the buffered risk. The
ordering remains: client-specific windows reduce the gap by 16\% relative to the best common window, which
now keeps about forty weeks, and by 20\% relative to the full history; the paired difference is
$-8.1\cdot10^{-4}$ with a 95\% half-width of $5.4\cdot10^{-5}$. Against the common window with a data floor
the two memories are within one percent of each other.

\subsection{Reading the results through the theory}
\label{app:clinics-theory}

Theorem~\refor{thm:refresh-upper}{4.1} charges a memory $h_m$ through two terms per clinic, a drift cost
growing like $a_mh_m$ and a statistical cost falling like $\sqrt{b_m/h_m}$. Their balance is struck at
$h_m\propto\sqrt{b_m/a_m}$, and it moves in opposite directions across this network: $b_m$ is largest in the
small specialized clinics, $a_m$ largest in the emergency departments. One common length has to serve both
ends. Here it is pushed to the full history, because the stable clinics and the slowly drifting ones form
most of the network, and the emergency departments then carry the drift cost of every wave. This is
precisely the situation in which the statistical-drift scale $\mathcal H_\gamma$ of the client-specific
rule falls well below the common-memory scale, since the drift--noise masses $\sqrt{a_mb_m}$ are spread over
very different ratios $b_m/a_m$.

The comparison with the data-floor baseline isolates what the rule adds. The floor $100/\lambda_m$ already
makes the memory depend on the visit rate, which is one half of the prescription. It does not depend on
drift. Under exact minimization this is decisive: the floored common window shortens the memory of
pediatrics, cardiology and primary care to the same few weeks as the emergency departments and pays for
it in noise, whereas Corollary~\refor{cor:refresh-opt}{4.2} lengthens their windows in proportion to
$\gamma_m^{-1/2}$. The remaining 26\% is the value of accounting for drift, clinic by clinic.

Where the prescription and the oracle part, the reason is the shape of the drift. The theorem is built for
a worst-case drift that can keep moving in one direction, so the bias of an old record grows without
bound. Seasonal drift returns to its average. For pediatrics and primary care, whose shares oscillate
around a stable level, the bias of a long memory saturates at the amplitude of the season, and with a
moderate flow of patients the full history is the better compromise. Cardiology, whose ageing trend does not revert, is the opposite case, and
there the oracle keeps a finite window as the theory prescribes. The prescription is thus conservative
exactly where the drift is mean-reverting, and it still captures most of the attainable gain.

The solver changes the size of the effect but not its sign. Proposition~\refor{prop:refresh-sgd}{B.1}
makes precise that the memory scale does not depend on acceleration; what the solver contributes is the
optimization remainder, which with $K=8$ is comparable to the statistical term and dilutes every
difference between memories. Under Algorithm~\refor{alg:refresh}{1} the two memories that shorten the
emergency departments' windows, the data-floor baseline and the client-specific rule, therefore end up
close, while both remain clearly better than any memory that treats the network as one.

Section~7 and this experiment differ in how the optimum moves. There the population changes smoothly in
both covariates and labels, and the best common memory is the full history. Here the optimum is moved by
a cohort the model cannot see, at speeds that differ by two orders of magnitude from
clinic to clinic. Heterogeneous drift bounds justify different memories in both settings; only in the
second does the difference between memories carry a large price, and the client-specific rule then pays
off by a third of the error. Client-specific memory also holds far fewer records: the emergency
departments keep a week or two of patients instead of three years, so every solve touches a small
fraction of the data that the full history requires.

\subsection{Scope}
\label{app:clinics-scope}

The records are synthetic, and each clinic's population is a mixture of two fixed cohorts whose shares
follow real surveillance curves; the experiment tests the memory mechanism of the paper, not a clinical
prediction rule. The drift and noise bounds are evaluated along the exact optimal trajectory, which is
possible only because the objective is known; in deployment they would have to be certified or replaced
by the validation-based selection of Section~6. The single free scale of each memory is chosen with the
true tracking error, for every memory alike. Algorithm~\refor{alg:refresh}{1} is used unchanged, with a
public starting point fitted once on patients who never enter the streams.

\section{Station temperature records}
\label{app:station-temperature}

This appendix runs Algorithm~\ref{alg:refresh} on laws fitted to
measured temperature records; Figure~\ref{fig:station-main} in
Section~\ref{sec:refresh-exp} summarizes it. Two experiments are
reported. The first is inside the model of
Section~\ref{sec:refresh-model}: day-to-day weather is independent noise
around a fitted diurnal clock, the drift and noise bounds are attained
by construction, and Corollary~\ref{cor:refresh-opt} is scored against a
common window. The second replays held-out months in their recorded
order, where weather persists, and asks whether the same corollary
improves on a common length once its drift input matches that path.

\paragraph{Setup.}
Unless stated otherwise the runs use the scalar loss
$\ell(x;\xi)=(L/2)x^2+\xi x$ with $L=10^{-3}$ on $[-1/2,1/2]$, so $D=1$,
where $\xi$ is a scaled reading, and the raw-buffer implementation of
Section~\ref{sec:refresh} with $\tau=0.02$ hours. Weights are $1/3$.
With $\mu_m(t)$ the path of station $m$, the population gradient
$Lx+\theta(t)$, $\theta(t)=\sum_m w_m\mu_m(t)$, is the network path up
to a negligible curvature term, and the minimizer lies at an endpoint
of $[-1/2,1/2]$ whenever $|\theta(t)|>L/2$; the tables call this loss
``endpoint''. The curvature is kept small so that the optimization
remainder is negligible: the number of rounds is
$K=\lceil\max\{1,2\sqrt{LD/\Hg}\}\rceil=1$, the stride is
$\Delta=K\tau$, and $s_0=2K/\lambda_{\min}$, the choices of
Corollary~\ref{cor:refresh-joint}. With $K=1$ the accelerated solve and
averaged SGD coincide, so comparisons are between memories.

Client-specific windows are those of Corollary~\ref{cor:refresh-opt},
without a multiplier (labelled ``formula'' in the plots). The
bound-optimal common length minimizes the common-memory bound of
Appendix~\ref{app:benchmarks}; the tuned common length is chosen on
seeds $0$--$3$ from a fixed grid of common lengths.
Seeds $4$--$15$ are evaluation, and within a seed every policy sees the
same draws. The predicted gain is
$g=1-\Hg/(\bar\gamma^{1/3}\Xi^{2/3})$, the relative gap in the
comparison $\Hg\le\bar\gamma^{1/3}\Xi^{2/3}$ of
Section~\ref{sec:refresh-model}, computed before count floors are
imposed. The score inside the model is the population gap $\avgerr_T$.
On time-ordered months it is the prequential excess over the zero
report: the average, over the recorded readings, of the loss of the
current report on each reading as it arrives, minus the loss of the
report $x=0$ on the same reading; lower is better.

\subsection{Experiment I: inside the model}
\label{app:station-model}

The U.S. Climate Reference Network \citep{diamond2013uscrn} reports air
temperature every five minutes. We take the ten Decembers 2010--2019 at
three stations chosen by insolation before any simulation: Stovepipe
Wells (desert), Sitka (maritime), and Utqiagvik, where the sun does not
rise in December (polar night). The population path of a station is
the truncated Fourier series of its average day;
leave-one-December-out cross-validation selects $2$, $3$, and $0$
harmonics, so the polar-night path is constant. Paths are centred, and
paths and residuals are divided by one common scale,
$4.56^\circ\mathrm{C}$; the bounds below are quoted in degrees, before
this scaling, and the windows do not depend on it. The drift bounds are the
largest slopes of the fitted paths, $(1.58,0.354,0)^\circ\mathrm{C}$ per
hour. A draw at a slot is a reading of a uniformly chosen December day
at that slot, minus the slot mean, plus the fitted path. The noise
scales $(3.41,3.72,6.64)^\circ\mathrm{C}$ are the largest slot standard
deviations. Each station reports $\lambda=12$ draws per hour.

\begin{center}
\includegraphics[width=\linewidth]{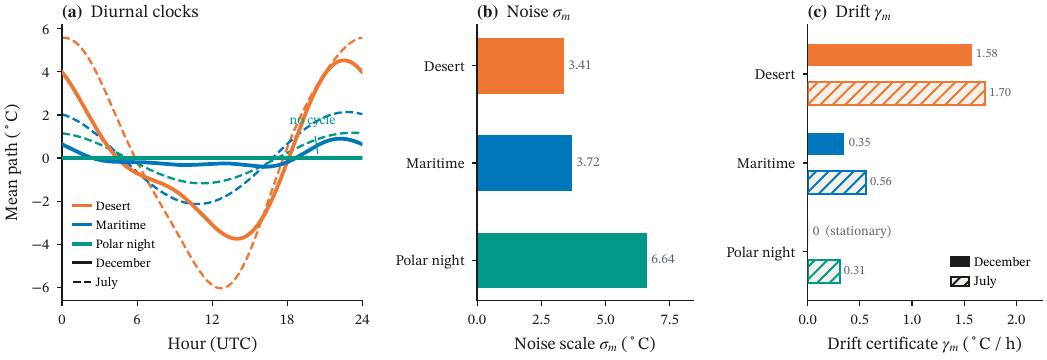}
\captionof{figure}{The two station laws. (a)~Mean paths in UTC,
December solid and July dashed; polar night has no December cycle.
(b)~Noise scales $\sigma_m$. (c)~Drift bounds $\gamma_m$; polar night
is stationary in December and moving in July.}
\label{fig:station-setup}
\end{center}

On this law $\Hg=0.111$. The prescribed windows are
$(0.556,1.28,\infty)$ hours. The bound-optimal common length is $1.01$
hours, the tuning seeds select $0.972$ hours, and $g=0.32$. The count
floor $2K/\lambda=1/6$ hour does not bind, $K\Cg\le\Hg$, and $T=432$
hours exceeds both joint-regime thresholds, so every hypothesis of
Corollary~\ref{cor:refresh-joint} used by the upper guarantee holds.

The client-specific mean is $0.0120$. Against the full history the
paired difference is $-0.0835$ with interval $[-0.0843,-0.0827]$, an
$87\%$ reduction. Against the bound-optimal common length it is
$-0.0125$ with interval $[-0.0139,-0.0112]$, a $51\%$ reduction, and
against the tuned common length it is $-0.0128$ with interval
$[-0.0144,-0.0113]$, a $52\%$ reduction. The client-specific windows
are below both common lengths on all twelve seeds. Per-station tuning,
shown for reference in Figure~\ref{fig:station-compare-law}, reaches
$0.0097$, so the untuned windows recover about $84\%$ of the
improvement over the bound-optimal common length that separate tuning
attains. A single length cannot give the desert a half-hour window and
polar night its full history at once.

The same three stations in July have harmonics $(3,1,1)$: the sun now
rises at all three. Then $g=0.066$ and the windows are
$(0.472,0.645,1.40)$ hours. The client-specific mean is $0.0128$, a
$6\%$ reduction against the bound-optimal common length, close to the
predicted $g$ but within noise (paired difference
$-7.7\cdot10^{-4}\pm8.2\cdot10^{-4}$), and an $11\%$ reduction against
the tuned common length. The Beijing July control groups twelve
districts into three shared-clock clients, $g=0.0002$, and the
client-specific windows tie the common length.

\begin{center}
\includegraphics[width=\linewidth]{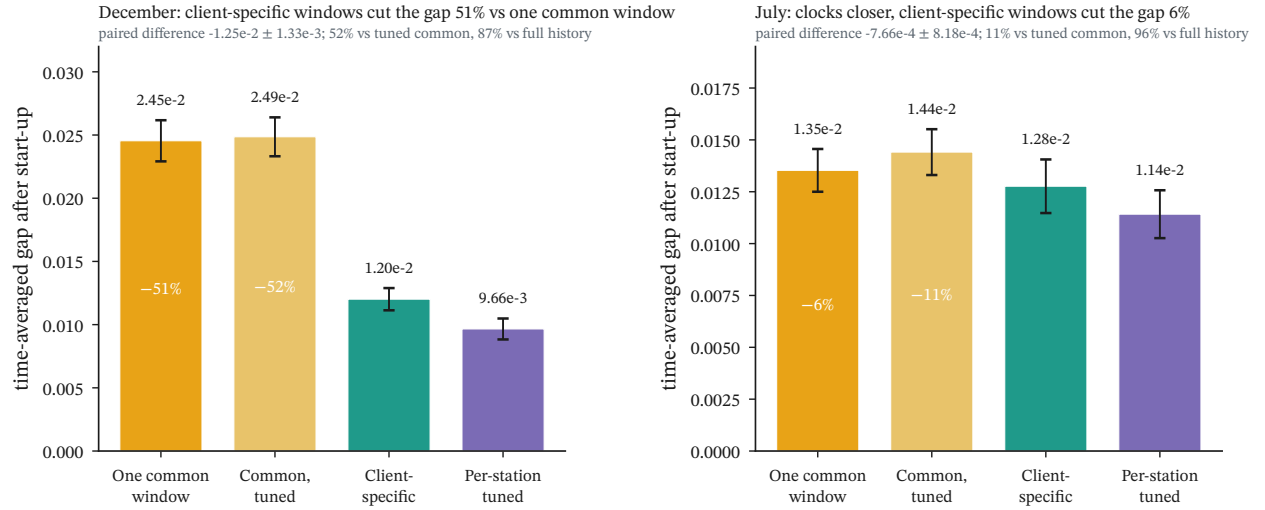}
\captionof{figure}{Mean $\avgerr_T$ over seeds $4$--$15$ with Student
$95\%$ intervals. In December client-specific windows cut the gap by
$51\%$ against the bound-optimal common window; in July, whose clocks
are closer, by $6\%$, within noise. The subtitle of each panel records
the paired difference and the reductions against the tuned common
length and the full history.}
\label{fig:station-compare-law}
\end{center}

\begin{center}
\includegraphics[width=\linewidth]{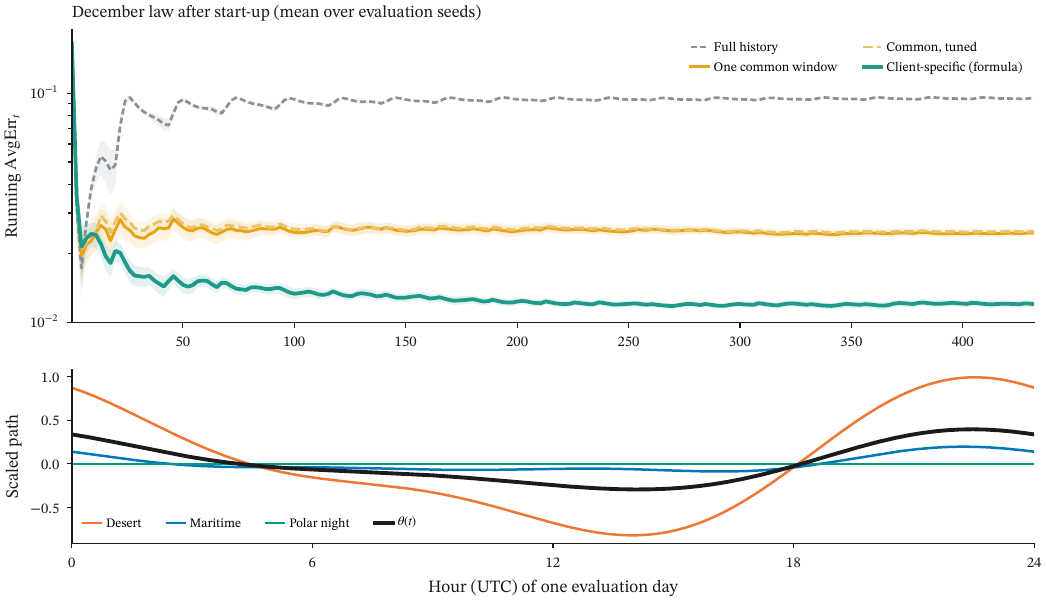}
\captionof{figure}{Top: running $\avgerr_t$ on the December law, mean
and one standard error over seeds $4$--$15$. Client-specific windows
separate from a common length within a day and stay below it. Bottom:
the three station means and the network path
$\theta(t)=\sum_m w_m\mu_m(t)$ on one evaluation day.}
\label{fig:station-gap-roll}
\end{center}

Over $200$ distinct random triplets from the $134$ network stations
with enough December coverage, the measured relative gain of
client-specific windows over the bound-optimal common length grows with
the predicted $g$ (Spearman $0.73$, bootstrap interval $[0.64,0.80]$)
and, for $g\ge5\%$, exceeds it on average: the comparison of upper
bounds is a conservative predictor. For the $31$ triplets with
$g\ge5\%$ the client-specific windows are below the bound-optimal
common length in all and below the tuned common length in $29$. When
the clocks agree, $g$ is small and a tuned common length can beat the
conservative drift input by a few percent.

\begin{center}
\includegraphics[width=\linewidth]{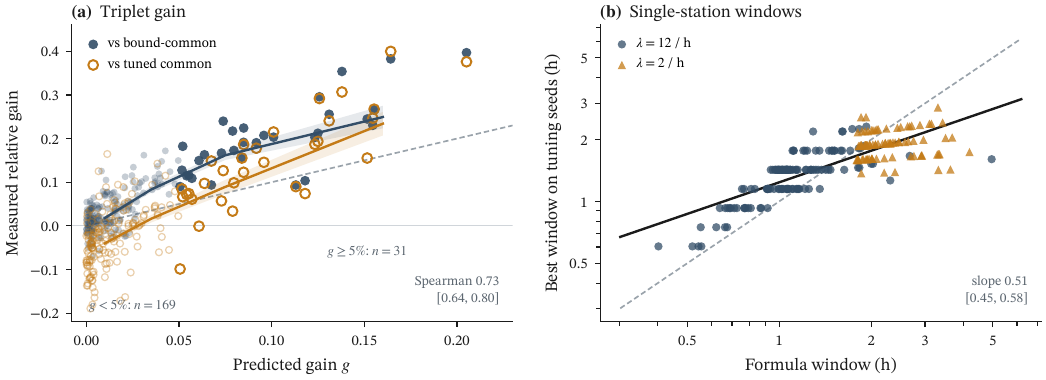}
\captionof{figure}{Left: measured relative gain of client-specific
windows over the bound-optimal and the tuned common length for $200$
random station triplets, against the predicted gain $g$. Opaque points
have $g\ge5\%$; curves are binned means with one standard error; the
dashed line is equality. Right: best window on four tuning seeds
against the prescribed window for single stations at two reading
rates; the log--log slope is $0.51$, not $1$, consistent with the
prescription charging the steepest slope rather than the typical one.}
\label{fig:station-network}
\end{center}

\subsection{Experiment II: time-ordered records}
\label{app:station-order}

The resampled law treats day-to-day weather as independent noise. To
see what that choice hides, each December is held out in turn; the law
is fitted on the other nine years; and the held-out month is replayed
in its recorded order for $T=744$ hours. A common length is tuned on
the nine training months replayed the same way. The climatological
windows, which treat weather as noise, lose to that tuned length in all
ten years: mean $-0.167$ against $-0.238$, paired difference $0.071$
with interval $[0.037,0.104]$. They keep the full history at the polar
site, whereas exponentially weighted memories tuned per station on the
training years keep less than an hour there: in recorded order,
persistent weather acts as drift.

Corollary~\ref{cor:refresh-opt} takes $(\gamma_m,\sigma_m)$ as inputs.
Two other certificates are scored on the same folds, both fixed on the
training folds before their held-out scores were computed.

The stream-scale pair uses five-minute differences for noise and the
largest hourly-mean slope for drift. Polar night is then the fastest
client. Every window binds at the count floor $2K/\lambda=1/6$ hour,
personalization vanishes, and the score is $-0.231$ against $-0.238$
for the tuned length.

The two-scale pair keeps the climatological noise and sets $\gamma_m$
to the sum of the clock Lipschitz constant and the largest change of
successive $12$-hour block means of the climate residual, per hour, a
scale above the hour-scale windows and below a day. Polar night is no
longer stationary. The prescribed windows, averaged over the training
folds, are $(0.356,0.665,1.11)$ hours, $g=0.078$, and the count floor
does not bind.

\begin{center}
\includegraphics[width=\linewidth]{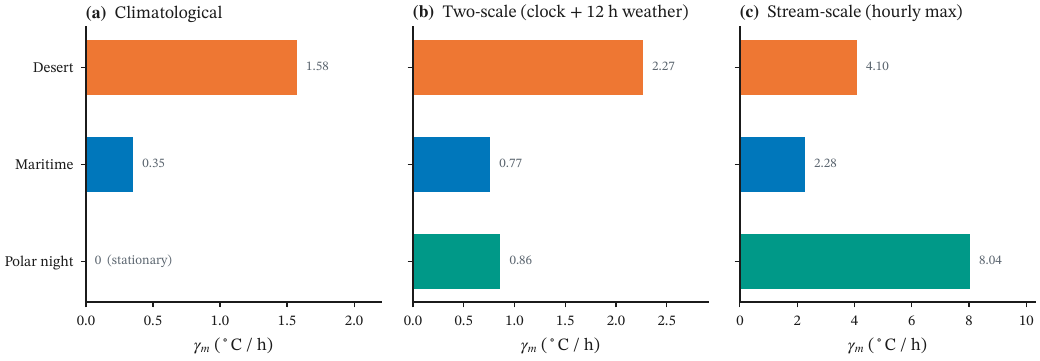}
\captionof{figure}{Drift certificates on the December training folds.
(a)~Climatological slopes: polar night is stationary.
(b)~Two-scale: clock plus the $12$-hour residual Lipschitz constant.
(c)~Stream-scale hourly maxima, the opposite extreme.}
\label{fig:station-certs-three}
\end{center}

\begin{center}
\includegraphics[width=\linewidth]{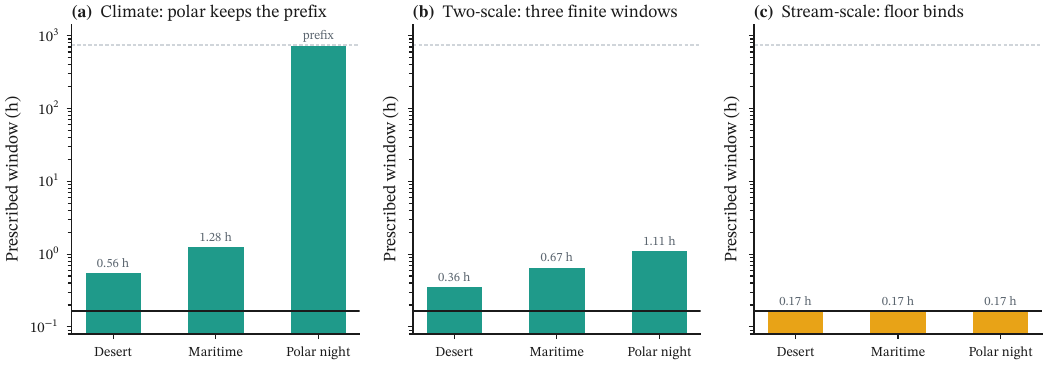}
\captionof{figure}{Windows the three certificates prescribe. Climate
keeps the full history at the polar site. Two-scale gives three finite
lengths $(0.36,0.67,1.11)$ hours. Stream-scale binds at the count
floor, so personalization vanishes.}
\label{fig:station-windows-three}
\end{center}

On the held-out Decembers the two-scale windows have mean $-0.2385$
against $-0.2382$ for the common length tuned on the other nine years,
a paired difference of $-0.0003$ with interval $[-0.0008,0.00008]$,
and are lower in seven of the ten years. Against the climatological
windows they win in all ten years, by $0.071$ with interval
$[0.038,0.105]$. Against the stream-scale floor they win in all ten, by
$0.0075$ with interval $[0.0061,0.0089]$. Squared loss repeats the
pattern.

On the ten Julys the same two-scale certificate matches the tuned common
length: it is lower in all ten years, by $0.00022$ with interval
$[0.00011,0.00033]$, about $0.06\%$ of the score. In July every station
has a diurnal clock, so the climatological windows are already next to
the tuned length. The two-scale windows use no grid: only the
training-fold certificate.

\begin{center}
\includegraphics[width=\linewidth]{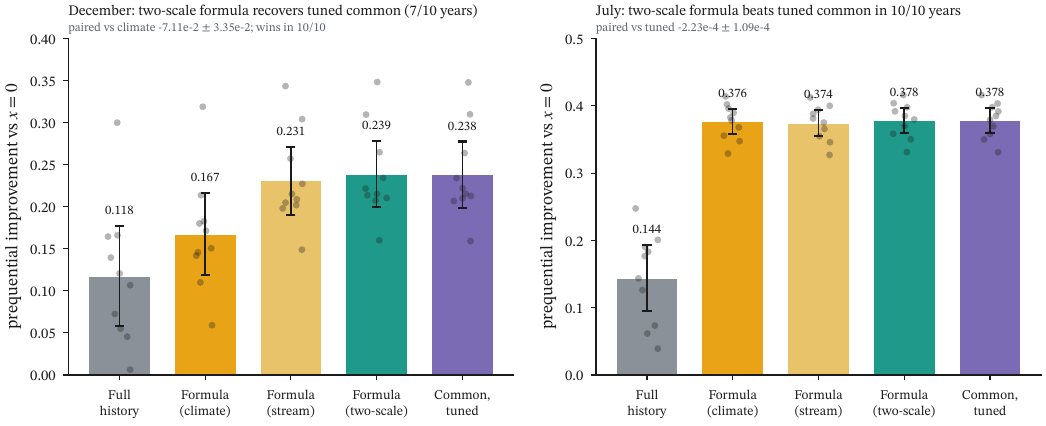}
\captionof{figure}{Prequential improvement over the zero report,
December (left) and July (right); the improvement is the negative of
the prequential excess reported in the text and in
Table~\ref{tab:station-laws}. Points are years; bars are means with
Student $95\%$ intervals. Higher is better. The two-scale windows
recover the tuned common length in December and match it in July.}
\label{fig:station-real-win}
\end{center}

\begin{center}
\includegraphics[width=\linewidth]{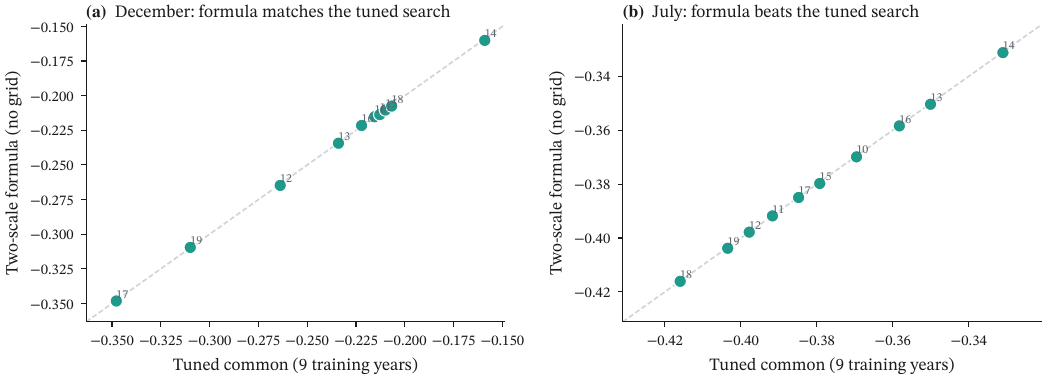}
\captionof{figure}{Two-scale windows against the common length tuned
on the other nine years, one point per held-out year; lower is better.
The two-scale windows use only the training-fold certificate. Both
months sit on the diagonal; in July every point lies slightly below it.}
\label{fig:station-no-tune}
\end{center}

The loss of the climatological windows on the recorded months is
therefore a wrong input, not a failure of the window rule. A
certificate that treats weather as drift, fixed on the training folds,
returns finite windows and matches a search over common lengths.

\subsection{Summary}

\begin{center}
\small
\setlength{\tabcolsep}{4pt}
\begin{tabular}{llcrrr}
\toprule
Law & loss & regime & $g$ & client-specific & minus tuned common\\
\midrule
December stations & endpoint & yes & $0.32$ & $0.0120$ & $-0.0128$ $[-0.0144,-0.0113]$\\
July stations & endpoint & yes & $0.066$ & $0.0128$ & $-0.0017$ $[-0.0025,-0.0008]$\\
July Beijing districts & endpoint & yes & $0.0002$ & $0.0235$ & $-0.0002$ $[-0.0013,0.0009]$\\
December, real order, climate & endpoint & no & $0.32$ & $-0.167$ & $0.071$ $[0.037,0.104]$\\
December, real order, stream-scale & endpoint & no & $0.09$ & $-0.231$ & $0.007$ $[0.006,0.008]$\\
December, real order, two-scale & endpoint & no & $0.078$ & $-0.239$ & $-0.0003$ $[-0.0008,0.00008]$\\
July, real order, climate & endpoint & no & $0.066$ & $-0.376$ & $0.0019$ $[0.0015,0.0023]$\\
July, real order, two-scale & endpoint & no & $0.022$ & $-0.378$ & $-0.00022$ $[-0.00033,-0.00011]$\\
\bottomrule
\end{tabular}
\captionof{table}{The two station experiments. The first three rows
are population $\avgerr_T$ over seeds $4$--$15$; the remaining rows are
prequential excess scores over ten held-out months (lower is better).
Intervals are paired Student $95\%$ intervals. ``Regime'' records
whether every hypothesis of Corollary~\ref{cor:refresh-joint} holds; in
real order independence fails. The predicted gain $g$ is computed
before count floors, which bind for the stream-scale windows.}
\label{tab:station-laws}
\end{center}

\clearpage
\section{Stationary and common-memory benchmarks}
\label{app:benchmarks}

The stationary and common-memory specializations of
Theorem~\ref{thm:refresh-upper} separate the roles of information
arrival and client-specific drift. We record them to make the
fixed-vector comparison of Section~\ref{sec:refresh-model} explicit.

\paragraph{Stationary information.}
When $\gamma_m=0$ for every client, all prefix memories are admissible
and $\Hg=0$. The drift terms disappear, and the leading finite-horizon
statistical term is proportional to $D\Xi/\sqrt T$, with
\[
 \Xi^2=\sum_m\frac{w_m^2\sigma_m^2}{\lambda_m}.
\]
The complete bound retains the optimization and startup terms of
\eqref{eq:refresh-average} and~\eqref{eq:refresh-transient}. This
information quantity weights a client's uncertainty by its contribution
to the target and by its arrival rate.
Proposition~\ref{prop:refresh-stationary} supplies the corresponding
finite-horizon statistical lower bound, including the gradient cap and
the finite-count corrections. Stationary clients therefore retain an
information cost even though they contribute nothing to a persistent
drift term.

\paragraph{Common memory and prescribed drift.}
Giving every client the same finite memory length $h$ produces the
window-dependent upper term
\[
 D\Bigl(2\bar\gamma h+\frac{\beta_0\Xi}{\sqrt h}\Bigr),
 \qquad \bar\gamma=\sum_mw_m\gamma_m,
 \qquad \beta_0=4\sqrt{2/3},
\]
and balancing its two terms yields the common-memory scale
$D\bar\gamma^{1/3}\Xi^{2/3}$, before sample-count floors and the other
remainders are imposed. Separate memories replace this scale by $D\Hg$,
as shown by Corollary~\ref{cor:refresh-opt}
and~\eqref{eq:refresh-comparison}. The comparison concerns the displayed
upper terms; it does not establish that a common-memory policy is optimal
among all causal estimators. The distinction also matters for lower
bounds: an aggregate constraint lets drift be allocated among clients,
whereas Theorem~\ref{thm:refresh-lower} respects every prescribed
clientwise bound, with the group-sensitive consequence in
Corollary~\ref{cor:refresh-dyadic}. The full proofs and the joint
comparison conditions are in Appendices~\ref{app:refresh-lower}
and~\ref{app:refresh-regime}.

\section{Detailed comparison of prior guarantees}
\label{app:prior-comparison}

\begin{table}[htbp]
\caption{Selected theorem assumptions and observation clocks. TV denotes
total variation. Each guarantee retains the conditions of its cited
theorem.}
\label{tab:referee-comparison}
\centering
\small
\setlength{\tabcolsep}{3pt}
\renewcommand{\arraystretch}{1.16}
\begin{tabular}{@{}p{0.22\linewidth}p{0.35\linewidth}p{0.39\linewidth}@{}}
\toprule
Reference and result & Assumptions and criterion & Clock and decisive distinction \\
\midrule
\citet{marfoq2023streams}, Thms.~4.1--4.4
& Fixed local laws; bounded convex smooth sample losses, bounded domain,
a stated pseudo-dimension condition; terminal population excess risk.
& Client arrivals and finite caches; the optimization theorem uses full
participation and full batches. Our populations drift under separate
clientwise bounds. \\
\addlinespace
\citet{ganguly2024nonstationary}, Master-FL
& Bounded convex Lipschitz losses and near-stationary base guarantees;
dynamic regret for changing empirical losses.
& Federated rounds with drift detection and restarts. Its target uses
current data-size weights; ours fixes population weights and models
noisy samples arriving over time. \\
\addlinespace
\citet{tyurin2025tight}, Thms.~H.11--H.12
& Fixed convex smooth objective; unbiased local gradients with bounded
variance; terminal expected objective gap.
& Heterogeneous gradient-computation times. Our clocks govern reusable
data from changing populations. \\
\addlinespace
\citet{cutler2023drift}, Thm.~8
& Strong convexity; fixed convex regularizer; second-moment drift/noise
and decorrelation conditions; current gap of an averaged iterate.
& One current stochastic gradient per iteration. We use a merely convex
target and separate client drifts and arrival rates. \\
\addlinespace
\citet{mazzetto2023density}, Thms.~4.2, 4.9
& Independent finite-support observations; regular endpoint drift for
terminal TV risk, bounded successive drift for long-run average TV risk.
& Observation index. Our criterion is convex suboptimality of a weighted
client aggregate, with unequal group clocks. \\
\addlinespace
This work, Thms.~\ref{thm:refresh-upper}, \ref{thm:refresh-lower}
& Compact convex domain; public clientwise gradient-drift bounds;
expected time-average current-objective gap.
& Exogenous arrivals and paid synchronization; reset solves and sample
splitting permit reuse. Local computation is uncharged. \\
\bottomrule
\end{tabular}
\end{table}

\end{document}